\documentclass[twoside,11pt]{article}

\usepackage{mathtools}
\usepackage{microtype}
\usepackage{multirow}
\usepackage{booktabs}
\usepackage{pifont}
\usepackage{algorithm}
\usepackage{algpseudocode}
\newtheorem{assumption}{Assumption}

\newcounter{figurepanel}[figure]

\newcommand{\figurepanellabel}[1]{\refstepcounter{figurepanel}\label{#1}}

\newcommand{\dalg}{MOMEHA}
\newcommand{\salg}{MB-MOMEHA}

\usepackage{jmlr2e}
\usepackage{capt-of}

\usepackage{lastpage}

\ShortHeadings{Efficient Hessian-Free Methods for Multi-Objective Bilevel Optimization}{Yicong Jiang and Feihu Huang}
\firstpageno{1}

\begin{document}

\title{Efficient Hessian-Free Methods for Multi-Objective Bilevel Optimization with a Nonconvex Lower-Level Objective}

\author{\name Yicong Jiang \email jiangyicong@nuaa.edu.cn \\
       \addr College of Computer Science and Technology\\
       Nanjing University of Aeronautics and Astronautics,
       Nanjing, China
       \AND
       \name Feihu Huang\thanks{Corresponding author.} \email huangfeihu2018@gmail.com\\
       \addr College of Computer Science and Technology\\
       Nanjing University of Aeronautics 
       and Astronautics, 
       Nanjing, China\\
       MIIT Key Laboratory of Pattern Analysis and Machine Intelligence, 
       Nanjing, China}

\editor{My editor}

\maketitle

\begin{abstract}
	Multi-objective bilevel optimization has wide applications in the AI area such as automated learning and multi-task meta-learning.
	Although some recent studies have investigated multi-objective bilevel optimization, many existing methods rely on the (strongly) convex lower level problems. In fact, these multi-objective bilevel learning problems are generally nonconvex, and particularly their lower level problems are nonconvex.
	To fill this gap, we propose a class of Multi-Objective Moreau Envelope based Hessian-free Algorithms (MOMEHA) for the multi-objective bilevel learning problems with a nonconvex lower level. Specifically, our method uses the Moreau envelope to formulate a multi-objective single-level relaxation that enforces lower-level first-order stationarity. In particular, our method retains computational advantages of being single-loop and Hessian-free in the multi-objective setting by incorporating a smooth weighted Tchebycheff scalarization.
	Furthermore, we propose a momentum variant of MOMEHA (i.e., MB-MOMEHA) for the stochastic multi-objective bilevel learning problems. Under the stated assumptions, we establish finite-iteration bounds on the Moreau gap and Pareto stationarity residuals under dynamically relaxed gap constraints in both deterministic and stochastic settings.
	Some experiments on few-shot meta-learning and neural architecture search demonstrate that our methods outperform the existing approaches in terms of Pareto front quality, validating their effectiveness and robustness.
\end{abstract}


\section{Introduction}

Multi-objective optimization (MOO)~\citep{chen2025gradient} is a fundamental mathematical framework for decision-making under conflicting criteria, with applications spanning engineering design~\citep{manguri2025topology}, finance~\citep{gulia2023systematic} and machine learning~\citep{zhong2024panacea}. Unlike single-objective optimization, where the goal is to find a single best solution, MOO seeks to characterize the Pareto front—a set of objective values representing optimal trade-offs among the competing objectives, where improving one objective necessarily degrades another. Recently, 
many gradient-based optimization methods~\citep{sener2018multi,ye2024first,momma2022multi,zhang2025pmgda} have been developed to solve MOO problems in AI applications such as LLM alignment ~\citep{zhong2024panacea} and LLM pruning~\citep{chen2026multi}. For example, 
\citet{sener2018multi} proposed the multiple gradient descent algorithm (MGDA) for MOO by aggregating per-objective gradients to obtain a common descent direction that simultaneously decreases all objectives.  Subsequently, \citet{momma2022multi} proposed a  weighted Chebyshev MGDA (WC-MGDA)  by decomposing the MOO problem into multiple sub-problems, each associated with a predefined preference vector. 

In many real-world scenarios, meanwhile, the decision-making process exhibits a hierarchical structure~\citep{vicente1994bilevel,liu2021investigating}: the evaluation of candidate solutions at the upper level depends on the optimal solution of a nested optimization problem at the lower level, which gives rise to bilevel optimization (BLO)~\citep{zhang2024introduction}. When the upper-level decision involves multiple, potentially conflicting criteria, the  problem becomes a type of multi-objective bilevel learning (MOBL)~\citep{giovannelli2024bilevel}, which can be formulated as:
\begin{align}
	\label{eq:d_mobl}
	& \min_{(x,y)\in\mathbb R^{d_x}\times\mathbb R^{d_y}} F(x,y)\coloneqq[f_i(x,y)]_{i=1}^m  & \mbox{Upper Level (UL)} \\
	& \text{s.t.}\quad y\in S_*(x):=\arg\min_{\theta\in\mathbb R^{d_y}}g(x,\theta) &  \mbox{Lower Level (LL)}. \nonumber
\end{align}
where $m$ is the number of upper-level objectives, $f_i(x,y)$ is the $i$-th upper-level objective, and $g(x,y)$ is the lower-level objective. The variables $(x,y)$ are chosen jointly. This formulation is particularly relevant in some modern machine learning tasks such as federated learning with fairness and robustness trade-off \citep{hu2022federated}, policy alignment in reinforcement learning for LLM \citep{chakraborty2024parl}, and multi-objective differentiable neural architecture search~\citep{sukthanker2024multi}.

\begin{table}[tbp]
	\centering

	\setlength{\tabcolsep}{3pt}
	\begingroup
	\small
	\begin{tabular}{@{}lccccc@{}}
		\toprule
		Algorithm & Scenario & LL Assumption & \shortstack{Loop\\Structure} & \shortstack{Hessian-\\free} & Preference \\
		\midrule
		MOML \citep{ye2021multi}                             & Deterministic     & Singleton Optimality      & Nested            & \ding{55}                 & \ding{55} \\
		gMOBA \citep{yang2024gradient}                       & Deterministic     & Strong Convexity          & Single            & \ding{55}                 & \ding{55} \\
		FORUM \citep{ye2024first}                            & Deterministic     & Strong Convexity          & Nested            & \ding{51}                 & \ding{55} \\
		MoCo \citep{fernando2022mitigating}                  & Stochastic        & Strong Convexity          & Nested            & \ding{55}                 & \ding{55} \\
		WC-MHGD \citep{zhang2026multi}                       & Both              & Strong Convexity          & Nested            & \ding{55}                 & \ding{51} \\
		WC-penalty \citep{zhang2026tale}                     & Both              & General Convexity         & Single            & \ding{55}                 & \ding{51} \\
		\textbf{(MB-)MOMEHA} (\textbf{Ours})                & \textbf{Both}     & \textbf{Nonconvexity}     & \textbf{Single}   & \ding{51}                 & \ding{51} 
		\\	\bottomrule
	\end{tabular}
	\endgroup
	\par
	\caption{Comparison of the gradient-based methods for multi-objective bilevel learning.}
	\label{tab:alg_comparison}
	\smallskip
	\begin{minipage}{\linewidth}
		\footnotesize
		Scenario indicates whether convergence is analyzed in the deterministic or stochastic setting. LL Assumption refers to the property requirement on the lower-level objective. Loop Structure indicates whether an inner loop structure is used to optimize the lower-level variables per iteration. Hessian-free marks whether the algorithm avoids Hessian information. Preference shows whether the method can explore the Pareto front via user-specified preference vectors. For MOMEHA and MB-MOMEHA, the theoretical guarantees concern the lower-level stationary relaxation.
	\end{minipage}
\end{table}

\begin{table}[tb]
	\centering
	\small
	\setlength{\tabcolsep}{3pt}
	\renewcommand{\arraystretch}{1.15}
	\begin{tabular}{@{}llcccc@{}}
		\toprule
		Setting 		& Method 							& \shortstack{Lower-level\\gradient norm} 	& \shortstack{Stationarity\\residual norm} 	& Batch size 							& \shortstack{Oracle\\complexity} \\
		\midrule
		Deterministic 	& \textbf{\dalg~(Ours)} 			& $\mathcal O(T^{-1/4})$ 					& $\mathcal O(T^{-1/4})$ 					& - 									& $\mathcal O(\epsilon^{-4})$ \\
						& WC-penalty \citep{zhang2026tale} 	& $\mathcal O(T^{-1/4})$ 					& $\mathcal O(T^{-1/4})$ 					& - 									& $\mathcal O(\epsilon^{-4})$ \\
						& \textbf{\dalg~(Ours)}				& $\mathcal O(1)$ 							& $\mathcal O(T^{-1/2})$ 					& - 									& $\mathcal O(\epsilon^{-2})$ \\
						& WC-penalty \citep{zhang2026tale} 	& $\mathcal O(1)$ 							& $\mathcal O(T^{-1/4})$ 					& - 									& $\mathcal O(\epsilon^{-4})$ \\
		\midrule
		Stochastic 		& \textbf{\salg~(Ours)}				& $\mathcal O(T^{-1/4})$ 					& $\mathcal O(T^{-1/4})$ 					& $\Theta(T^{3/2})$ 					& $\mathcal O(\epsilon^{-10})$ \\
						& WC-penalty \citep{zhang2026tale} 	& $\mathcal O(T^{-1/4})$ 					& $\mathcal O(T^{-1/4})$ 					& $\Theta(T^{5/4})$						& $\mathcal O(\epsilon^{-9})$ \\
						& \textbf{\salg~(Ours)}				& $\mathcal O(1)$ 							& $\mathcal O(T^{-1/2})$ 					& $\Theta(T)$ 							& $\mathcal O(\epsilon^{-4})$ \\
						& WC-penalty \citep{zhang2026tale} 	& $\mathcal O(1)$ 							& $\mathcal O(T^{-1/4})$ 					& $\Theta(T^{5/4})$\textsuperscript{a} 	& $\mathcal O(\epsilon^{-9})$ \\
		\bottomrule
	\end{tabular}
	\par
	\caption{Comparison of convergence rates and oracle complexity between our methods with WC-penalty~\citep{zhang2026tale} under \textbf{the same-order measures}.}
	\label{tab:wc_gradient_comparison}
	\smallskip
	\begin{minipage}{\linewidth}
		\footnotesize
		Here, $\|\nabla_y g(x,y)\|$ denotes lower-level gradient norm. For our methods, the stationarity residual norm column reports the $\varepsilon_s$ component of our stationarity measure (Please see the following definition~\ref{def:momeha_PS} ).
		For WC-penalty method, the stationarity column reports the square root of the metric (i.e.,$\frac{1}{T}\sum_{t=0}^{T-1}\|\mathcal{K}(\rho_t,z_t,\omega_t,\nu_t,\lambda)\|$) for the sake of fairness. When $\|\nabla_y g(x,y)\|$ is allowed to remain $\mathcal O(1)$, the residual excludes the lower-level feasibility block.
		Oracle complexity counts full oracle calls for deterministic methods and sample calls for stochastic methods, with the numbers of objectives and constraints and the confidence level fixed.
		\par\smallskip
		\noindent\textsuperscript{a}\, The batch-size column reports the largest batch order per iteration. In this row, the lower- and upper-level batches are $\Theta(T^{3/4})$ and $\Theta(T^{5/4})$, respectively. The stochastic WC-penalty joint row uses $\Theta(T^{5/4})$ at both levels.
	\end{minipage}
\end{table}

Since the above problem~(\ref{eq:d_mobl}) is widely used in many AI tasks, more recently, some gradient-based methods have been developed to solve these MOBL problems. For example, multi-objective meta learning (MOML)~\citep{ye2021multi} is the pioneering work that introduces a multi-objective meta-learning formulation and solves it via gradient-based bilevel optimization, but it provides only asymptotic convergence guarantees. Subsequently, \citet{fernando2022mitigating} proposed a stochastic multi-objective gradient with correction (MoCo) and established its non-asymptotic convergence rates by leveraging a momentum-assisted hypergradient. It operates with a nested structure and requires computing expensive Hessian-vector products. \citet{ye2024first} proposed an effective first-order
multi-gradient method (FORUM) for MOBL by eliminating the Hessian requirement, which offers a Hessian-free alternative, yet it maintains the double-loop design and the restrictive strong convexity assumption on the lower level. More recently, \citet{zhang2026multi} proposed a weighted Chebyshev multi-hyper gradient descent (WC-MHGD) by introducing the capability of steering the optimization along a user-specified preference direction, which enables  direct exploration of the Pareto front. However, the WC-MHGD still demands Hessian computations, and also assumes strong convexity on the lower level.
Subsequently, \citet{zhang2026tale} proposed an effective weighted Chebyshev (WC)-penalty algorithm to solve the MOBL problems with general convex lower levels.

Notably, almost all existing MOBL methods reviewed above are limited by the assumption of lower-level convexity or strong convexity (please see Table~\ref{tab:alg_comparison}), which tends to fail in modern deep learning applications such as neural architecture search and policy alignment for LLM.
To address lower-level nonconvexity, we propose an efficient  multi-objective Moreau envelope based Hessian-free method targeting the lower-level stationary relaxation of the bilevel problem.
Our main contributions are summarized as follows:
\begin{itemize}
	\item[1)] We propose an efficient  \dalg~and its stochastic momentum variant, \salg, for the lower-level stationary relaxation of multi-objective bilevel optimization with a nonconvex lower-level objective. In particular, our  algorithms are single-loop and Hessian-free.
	\item[2)] We introduce an $\varepsilon_c$-$\varepsilon_s$ Pareto stationarity measure based on dynamic relaxation of the Moreau gap constraint. The underlying normal-cone condition is necessary for local weak Pareto optimality on the relaxed feasible set without any constraint qualification.
	\item[3)] We establish solid non-asymptotic $\varepsilon_c$-$\varepsilon_s$ Pareto stationarity convergence guarantees for our algorithms under weakly-convex LL objective, and further provide the LL optimality guarantee under additional Polyak--\L{}ojasiewicz (PL) condition.
	\item[4)] We conduct experiments on few-shot meta-learning and neural architecture search, demonstrating that our methods explore better Pareto fronts than the prior methods, which underscores our method's applicability to real-world nonconvex problems. 
\end{itemize}

Table~\ref{tab:wc_gradient_comparison} compares our convergence guarantees with those of WC-penalty~\citep{zhang2026tale}, the most closely related method. When the lower-level gradient norm is allowed to remain bounded, our methods achieve faster stationarity rates and lower oracle complexity bounds in both settings. When joint convergence, the stochastic oracle complexity exponent of our MB-MOMEHA method is slightly higher than WC-penalty. Although our algorithms are not uniformly better in sample complexity, they avoid Hessian computation and can handle more challenging problems with a nonconvex lower level problem.

\section{Related Works}
In this section, we review the gradient-based methods for multi-objective optimization and bilevel optimization, respectively. 

\subsection{Multi-Objective Optimization}
Gradient-based methods for MOO follow three main strategies~\citep{chen2025gradient}:
The first strategy finds a single balanced Pareto-optimal solution, including loss-balancing approaches \citep{liu2021towards, ye2021multi, lin2021reasonable, ye2024first} and gradient-balancing approaches \citep{sener2018multi, yu2020gradient, liu2021conflict, fernando2022mitigating} that aggregate per-objective gradients to obtain a common descent direction. 
The second strategy provides a finite discrete set of trade-off solutions and is further divided into preference-based methods that decompose the problem using predefined preference vectors~\citep{mahapatra2020multi, mahapatra2021exact, momma2022multi, zhang2024gliding, zhang2025pmgda} and preference-free methods that directly optimize a solution set~\citep{deist2021multi, liu2021profiling}. 
The third strategy learns a continuous (infinite) preference-to-solution mapping, using architectures such as hypernetworks~\citep{tuan2024hyper}, preference-conditioned networks~\citep{raychaudhuri2022controllable}, or model combination~\citep{dimitriadis2025pareto}.
Existing gradient-based methods in MOBL are primarily developed under the former two MOO strategies: finding a balanced solution or finding a finite discrete Pareto set. Our method follows the latter strategy and uses preference vectors.

\subsection{Bilevel Optimization}
Gradient-based methods for BLO are commonly categorized into three strategies according to how they approximate the implicit gradient. 
Implicit function (IF) strategy utilized in \citet{hong2020two, ji2021bilevel, xiao2022alternating} computes the hypergradient via the implicit function theorem without explicitly unrolling the lower-level optimizer, which relies on well-defined lower-level Hessian inversion.
The gradient unrolling (GU) strategy implemented in \citet{franceschi2017forward, franceschi2018bilevel, shaban2019truncated} approximates the lower-level solution by unrolling a fixed number of optimizer steps and then back-propagating through the unrolled computational graph.
Classical IF and GU analyses often rely on lower-level strong convexity or a unique lower-level solution. Value function (VF) strategy in~\citet{liu2021value, liu2022bome, shen2023penalty, kwon2023fully} reformulates the BLO problem as a constrained single-level problem using the lower-level value function, which offers greater flexibility for handling non-convex lower-level scenarios. 
Despite their algorithmic differences, all these methods focus on a single upper-level objective. We adopt the VF-based strategy and generalize it to a multi-objective upper level.

\section{Preliminaries}
We introduce the bilevel formulations and the notions of Pareto optimality and stationarity used in the analysis.

\subsection{Problem Setup}
We consider the deterministic MOBL formulation \eqref{eq:d_mobl} with $g(x,\cdot)$ weakly convex, and develop algorithms and residual guarantees for its lower-level stationary relaxation. The corresponding stochastic MOBL problem is formulated as follows.
\begin{align}
	\label{eq:s_mobl}
	& \min_{(x,y)\in\mathbb R^{d_x}\times\mathbb R^{d_y}}F(x,y)\coloneqq\big[\mathbb E_{\xi\sim\mathcal D_i}[f_i(x,y;\xi)]\big]_{i=1}^m, \\
	& \text{s.t.}\quad y\in S_*(x):=\arg\min_{\theta\in\mathbb R^{d_y}}\mathbb E_{\varrho\sim\mathcal D_g}[g(x,\theta;\varrho)]. \nonumber
\end{align}
Here $\mathcal D_i$ and $\mathcal D_g$ are the upper- and lower-level data distributions. In the stochastic analysis, $f_i(x,y):=\mathbb E_{\xi\sim\mathcal D_i}[f_i(x,y;\xi)]$ and $g(x,y):=\mathbb E_{\varrho\sim\mathcal D_g}[g(x,y;\varrho)]$ denote the population objectives. The lower-level population objective is weakly convex in $y$.

\subsection{Related Concepts}
\label{subsec:concepts}
The definitions below use a generic decision vector $z$, with $z=(x,y)$ for the joint formulations in this work. In MOBL, improving one objective may degrade another. This trade-off is formalized through Pareto optimality.
\begin{definition}[Pareto Optimality]
	\label{def:PO}
	(i) A feasible point $z_1$ dominates another feasible point $z_2$ if $F_i(z_1)\le F_i(z_2)$ for every $i\in[m]$, with $F_j(z_1)<F_j(z_2)$ for at least one $j\in[m]$.
	(ii) A feasible point $z^*$ is Pareto optimal if no feasible point $z$ dominates $z^*$.
\end{definition}
The requirement above is actually stringent. In many practical and theoretical settings, one may encounter points that are not Pareto optimal yet cannot be strictly dominated.
Therefore, a more feasible concept is weak Pareto Optimality.
\begin{definition}[Weak Pareto Optimality]
	\label{def:w_PO}
	A feasible point $z^*$ is weakly Pareto optimal if no feasible point $z$ satisfies $F_i(z)<F_i(z^*)$ for every $i\in[m]$.
\end{definition}
\begin{definition}[Pareto Front]
	A (weak) Pareto front is the set of all objective function values of all (weak) Pareto optimal solutions.
\end{definition}
Obviously, every Pareto-optimal point is weakly Pareto optimal, but the converse generally fails. If a common descent direction exists, the current point is not even weakly Pareto optimal. For smooth unconstrained problems, a first-order necessary condition gives rise to Pareto stationarity.
\begin{definition}[Unconstrained Pareto Stationarity]
	\label{def:uc_PS}
	A point $z$ is $\varepsilon$-Pareto stationary if there exists $\lambda \in \Delta_{m - 1}$ such that
	\begin{equation*}
		\left\|\sum^m_{i = 1} \lambda_i \nabla F_i(z)\right\| \le \varepsilon.
	\end{equation*}
\end{definition}
When explicit constraints on $z$ are present, Pareto stationarity must account for the tangent cone of the feasible set: a common descent direction for the objectives may be infeasible.

Throughout, $\mathcal{N}_C(z)$ denotes the limiting (Mordukhovich) normal cone to a closed set $C$: its elements are limits of regular normals at nearby feasible points. The full normal cone gives a constrained Pareto stationarity condition without requiring a constraint qualification.
\begin{definition}[Constrained Pareto Stationarity]
	\label{def:c_PS}
	For MOO problems with finitely many smooth equality constraints $\mathcal{E}$ and inequality constraints $\mathcal{I}$, let
	$C := \{ z \mid g(z) = 0\ \forall g \in \mathcal{E},\ h(z) \le 0\ \forall h \in \mathcal{I} \}$ denote the feasible set, and let $\mathcal{N}_C(z)$ be the normal cone to $C$ at $z$. A point $z\in C$ is $\varepsilon$-Pareto stationary if there exist $\lambda \in \Delta_{m-1}$ and $n \in \mathcal{N}_C(z)$ such that
	\begin{equation}
		\left\| \sum^m_{i=1} \lambda_i \nabla F_i(z) + n \right\| \le \varepsilon.
	\end{equation}
\end{definition}
For smooth objectives, every local weak Pareto optimum satisfies this condition with $\varepsilon=0$; see Appendix~\ref{sec:stationarity_details} for the necessity proof and the normal-cone relations used below.
As explained in Section~\ref{sec:framework}, the Moreau gap gives an explicit constraint for the lower-level stationary relaxation. To characterize penalty iterates with a positive gap, we use Pareto stationarity under a dynamically relaxed gap constraint in Section~\ref{sec:theory}.

\section{Our Methods}
\label{sec:framework}
We propose \dalg~and \salg~to address the lower-level stationary relaxations of the deterministic and stochastic MOBL formulations \eqref{eq:d_mobl} and \eqref{eq:s_mobl}, respectively.

\subsection{Our \dalg~Algorithm}
We present \dalg~for the deterministic lower-level stationary relaxation using the Moreau gap constraint and smooth Tchebycheff scalarization.

Following the Moreau envelope framework~\citep{gao2023moreau,liu2024moreau}, consider the constrained vector problem
\begin{align}
	\label{eq:moreau_mobl}
	\min_{(x,y)\in\mathbb R^{d_x}\times\mathbb R^{d_y}}F(x,y)
	\quad\text{s.t.}\quad q_\gamma(x,y):=g(x,y)-\upsilon_\gamma(x,y)\le0,
\end{align}
where
\[
\upsilon_\gamma(x,y):=\min_{\theta\in\mathbb R^{d_y}}
\left\{g(x,\theta)+\frac1{2\gamma}\|\theta-y\|^2\right\}.
\]
Recall the lower-level global minimizer set $S_*(x)$ from \eqref{eq:d_mobl}, and define the stationary set
\begin{equation}
	\label{eq:lower_stationary_set}
	S_{\mathrm{stat}}(x):=\{y\in\mathbb R^{d_y}:\nabla_yg(x,y)=0\}.
\end{equation}
For continuously differentiable $g$ with $g(x,\cdot)$ $\rho_y$-weakly convex and $\gamma\in(0,1/(2\rho_y))$, the Moreau gap is nonnegative and satisfies
\[
q_\gamma(x,y)=0\iff\nabla_yg(x,y)=0.
\]
The complete equivalence and proof are given in Proposition~\ref{prop:moreau_stationary_relaxation} in Appendix~\ref{subsec:moreau_lower_stationarity}.
Thus \eqref{eq:moreau_mobl} is equivalent to the lower-level stationary relaxation
\begin{align}
	\label{eq:relaxed_mobl}
	\min_{(x,y)\in\mathbb R^{d_x}\times\mathbb R^{d_y}}F(x,y)
	\quad\text{s.t.}\quad y\in S_{\mathrm{stat}}(x),
\end{align}
with $(x,y)$ selected jointly. This formulation targets the lower-level stationary relaxation, whose feasible set agrees with the original bilevel feasible set when every lower-level stationary point is a global minimizer. Further discussion and a nonconvex example are given in Appendix~\ref{subsec:moreau_lower_stationarity}. The explicit gap constraint lets us approximate envelope gradients using first-order updates.

To equip our methods with the ability to explore the Pareto front, we adopt the smooth Tchebycheff scalarization (STCH) proposed in \citet{lin2024smooth}, which is defined as follows:
\begin{align}
	\label{eq:stch}
	F^{(\text{STCH})}_w(x, y) = \frac{1}{\mu} \log \left(\sum_{i=1}^m \exp(\mu w_i (f_i(x, y) - z_i))\right),
\end{align}
where $\mu$ is the smoothing parameter, $w \in \Delta_{m - 1}$ is the preference vector, and $z_i < f_i(x,y)$ is the ideal value of the $i$-th objective. As $\mu \rightarrow +\infty$, it uniformly approximates the true Tchebycheff maximum.
STCH avoids the non-differentiability and slow convergence caused by the $\max(\cdot)$ operator in the Tchebycheff scalarization, while still being able to explore the Pareto front by enumerating preference vectors. 

Scalarizing the vector objective in the lower-level stationary relaxation \eqref{eq:moreau_mobl} with STCH gives the optimization problem used by our methods:
\begin{align}
	\label{eq:momeha}
	\min_{(x, y) \in \mathbb{R}^{d_x} \times \mathbb{R}^{d_y}} F^\text{(STCH)}_w(x, y)\\
	\text{s.t.} \quad g(x, y) - \upsilon_\gamma(x, y) \le 0. \nonumber 
\end{align}
Since $q_\gamma\ge0$, we penalize the Moreau gap and optimize
\begin{align}
	\label{eq:penalty}
	\min_{(x, y) \in \mathbb{R}^{d_x} \times \mathbb{R}^{d_y}} F^\text{(STCH)}_w(x, y) - \underline{F} + c_t(g(x, y) - \upsilon_\gamma(x, y)),
\end{align}
where $\underline F$ is a fixed global lower bound for $F_w^{\mathrm{(STCH)}}$, as in Assumption~\ref{assumption:ul}, and $c_t$ is a non-decreasing positive penalty parameter. Denote the objective in \eqref{eq:penalty} by $\mathcal P_{c_t}(x,y)$.
Under the continuity, lower-bound, and feasibility conditions detailed in Proposition~\ref{prop:global_penalty_limit} (Appendix~\ref{subsec:global_penalty_limits}), fix the scalarization parameters and let $c_k\to\infty$. If a sequence of exact global penalty minimizers can be chosen, every accumulation point is a global minimizer of the scalarized envelope-constrained problem \eqref{eq:momeha}. The algorithmic iterates are analyzed through potential decrease and finite-iteration residual estimates in Section~\ref{sec:theory}.

\begin{algorithm}[tb]
	\caption{\dalg~Algorithm}
	\label{alg:dalg}
	\textbf{Input}: Number of iterations $T\ge2$, preference $w$, smoothing parameter $\mu$, Moreau envelope regularity $\gamma$, step sizes $\{\alpha_{\theta, t}\}$, $\{\alpha_{x, t}\}$, and $\{\alpha_{y, t}\}$, penalty factor $\{c_t\}$. \\
	\textbf{Initialization}: Given variables $x_0, y_0$, let $\theta_0=y_0$, $\ell_T=\lceil T/2\rceil$, and $\Psi_{\mathrm{best}}=+\infty$.
	\begin{algorithmic}[1] 
		\For{\(t=0,1,\dots, T - 1\)}
		\State Compute $d_{\theta, t}$ and update $\theta_t$ as in \eqref{eq:theta_gd_update}.
		\State Compute $d_{x, t}$ and update $x_t$ as in \eqref{eq:x_gd_update}.
		\State Compute $d_{y, t}$ and update $y_t$ as in \eqref{eq:y_gd_update}.
		\State Compute $\Psi_t=\|d_{\theta,t}\|^2+\|d_{x,t}\|^2+\|d_{y,t}\|^2$.
		\If{$t\ge\ell_T$ and $\Psi_t<\Psi_{\mathrm{best}}$}
		\State Set $\Psi_{\mathrm{best}}=\Psi_t$, $R_T=t$, and cache $(\widehat\theta_T,\widehat x_T,\widehat y_T)=(\theta_{t+1},x_{t+1},y_{t+1})$.
		\EndIf
		\EndFor
	\end{algorithmic}
	\textbf{Output:} $\widehat\theta_T,\widehat x_T,\widehat y_T$ selected by \eqref{eq:deterministic_output_selection}; .
\end{algorithm}

Given the resulting penalized problem \eqref{eq:penalty}, we adopt an alternating gradient descent strategy to update the optimization variables. According to Lemma A.6 of~\citet{liu2024moreau}, if $g(x, y)$ is $\rho_y$-weakly convex w.r.t. $y$ and $\gamma \in \left(0, \frac{1}{2\rho_y}\right)$, we can obtain:
\begin{equation}
	\label{eq:envelope_grad}
	\nabla \upsilon_\gamma(x, y) = 
	\begin{bmatrix}
		\nabla_x g(x, \theta^*_\gamma(x, y)) \\
		\frac{1}{\gamma}(y - \theta^*_\gamma(x, y))
	\end{bmatrix},
\end{equation}
where $\theta^*_\gamma(x, y) \coloneqq \arg \min_{\theta \in \mathbb{R}^{d_y}} \left\{ g(x, \theta) + \frac{1}{2\gamma} \left\|\theta - y\right\|^2 \right\}$.
As the gradient depends on the envelope optimum, we first obtain its approximation. Given the strong convexity of the envelope problem w.r.t. $\theta$, which ensures fast linear convergence, we nonetheless avoid expensive inner iterations by maintaining an auxiliary variable $\theta$ and updating it with a single step as follows:
\begin{align}
	\label{eq:theta_gd_update}
	d_{\theta, t} &= \nabla_y g(x_t, \theta_t) + \frac{1}{\gamma}(\theta_t - y_t), \nonumber \\
	\theta_{t+1} & = \theta_t - \alpha_{\theta, t} d_{\theta, t},
\end{align}
where $\alpha_{\theta, t}$ is the stepsize for $\theta_t$. Then for the upper-level variable $x_t$, we approximate $\nabla_x \upsilon_\gamma(x_t, y_t)$ using the $\theta_{t + 1}$ obtained from the update above, and obtain the following update rule for $x$:
\begin{align}
	\label{eq:x_gd_update}
	d_{x, t} &=  \  \frac{1}{c_t}\nabla_x F^{\text{(STCH)}}_{w}(x_t, y_t) + \nabla_x g(x_t, y_t) - \tilde{\nabla}_x \upsilon_\gamma(x_t, \theta_{t+1}) \nonumber \\
	& = \ \sum_{i = 1}^m \frac{\exp\left(\mu w_i \left(f_i(x_t, y_t) - z_i\right)\right) w_i \nabla_x f_i(x_t, y_t)}{c_t \sum_{j = 1}^m \exp\left(\mu w_j \left(f_j(x_t, y_t) - z_j\right)\right)}  + \nabla_x g(x_t, y_t) - \nabla_x g(x_t, \theta_{t+1}),  \nonumber \\
	x_{t+1} &= x_t - \alpha_{x, t} d_{x, t},
\end{align}
where $\alpha_{x, t}$ is the stepsize for $x_t$. Following a similar rationale, we approximate $\nabla_y \upsilon_\gamma(x_{t + 1}, y_t)$ using $\theta_{t + 1}$ and $x_{t + 1}$. The alternative update rule for $y$ is given by:
\begin{equation}
	\label{eq:y_gd_update}
	\begin{aligned}
		 d_{y, t} & = \frac{1}{c_t}\nabla_y F^{\text{(STCH)}}_{w}(x_{t + 1}, y_t) + \nabla_y g(x_{t+1}, y_t) - \tilde{\nabla}_y \upsilon_\gamma(x_{t+1}, \theta_{t+1}) \\
		& = \sum_{i = 1}^m \frac{\exp\left(\mu w_i \left(f_i(x_{t+1}, y_t) - z_i\right)\right) w_i \nabla_y f_i(x_{t+1}, y_t)}{c_t \sum_{j = 1}^m \exp\left(\mu w_j \left(f_j(x_{t+1}, y_t) - z_j\right)\right)} + \nabla_y g(x_{t+1}, y_t) + \frac{1}{\gamma}(\theta_{t + 1} - y_t), \\
		 y_{t + 1}  & = y_t - \alpha_{y, t} d_{y, t},
	\end{aligned}
\end{equation}
where $\alpha_{y, t}$ is the stepsize for $y_t$.

After $T\ge2$ updates, both algorithms return a point from the latter half of the computed trajectory. Define
\begin{equation}
	\label{eq:output_tail_window}
	\ell_T:=\lceil T/2\rceil,\qquad
	\mathcal J_T:=\{\ell_T,\ldots,T-1\},\qquad
	n_T:=|\mathcal J_T|=\lfloor T/2\rfloor.
\end{equation}
For \dalg, use the already computed full-gradient directions to select
\begin{equation}
	\label{eq:deterministic_output_selection}
	\Psi_t:=\|d_{\theta,t}\|^2+\|d_{x,t}\|^2+\|d_{y,t}\|^2,\qquad
	R_T\in\arg\min_{t\in\mathcal J_T}\Psi_t,
\end{equation}
with the earliest index chosen in case of a tie.

Finally, our MOMEHA algorithm for the deterministic case is described in Algorithm~\ref{alg:dalg}.

\begin{algorithm}[tb]
	\caption{\salg~Algorithm}
	\label{alg:salg}
	\textbf{Input}: Number of iterations $T\ge2$, preference $w$, smoothing parameter $\mu$, Moreau envelope regularity $\gamma$, step sizes $\{\alpha_{\theta,t}\}$, $\{\alpha_{x,t}\}$, and $\{\alpha_{y,t}\}$, penalty factor $\{c_t\}$, momentum parameter $\beta\in(0,1)$, batch sizes $\{B_t\}$. \\
	\textbf{Initialization}: Given $x_0, y_0$, set $\theta_0$ = $y_0$, $m_{\theta, 0} = 0$, $m_{x, 0} = 0$, $m_{y, 0} = 0$.
	\begin{algorithmic}[1] 
		\State Draw $R_T$ according to \eqref{eq:stochastic_output_selection}, independently of all update mini-batches.
		\For{\(t=0,1,\dots, T - 1\)}
		\State Independently draw mini-batches of size $B_t$: $\xi_{x, t, i}$, $\xi_{y, t, i}$ from $\mathcal{D}_i$ and $\varrho_{\theta, t}$, $\varrho_{x, t}$, $\varrho_{y, t}$ from $\mathcal{D}_g$.
		\State Compute \\
		$\hat{d}_{\theta, t} = \hat{\nabla}_y g(x_t, \theta_t; \varrho_{\theta, t}) + \frac{1}{\gamma}(\theta_t - y_t)$.
		\State Update $m_{\theta, t + 1} = \beta m_{\theta, t} + (1 - \beta) \hat{d}_{\theta, t}$.
		\State Update $\theta_{t + 1} = \theta_t - \alpha_{\theta,t} m_{\theta, t + 1}$.
		\State Compute $\hat{d}_{x, t} = \hat{\nabla}_x g(x_t, y_t; \varrho_{x, t})$ \\
		$ + \sum_{i = 1}^m \frac{\exp\left(\mu w_i \left(f_i(x_t, y_t; \xi_{x, t, i}) - z_i\right)\right) w_i \hat{\nabla}_x f_i(x_t, y_t; \xi_{x, t, i})}{c_t \sum_{j = 1}^m \exp\left(\mu w_j \left(f_j(x_t, y_t; \xi_{x, t, j}) - z_j\right)\right)} - \hat{\nabla}_x g(x_t, \theta_{t+1}; \varrho_{x, t})$.
		\State Update $m_{x, t + 1} = \beta m_{x, t} + (1 - \beta) \hat{d}_{x, t}$.
		\State Update $x_{t + 1} = x_t - \alpha_{x,t} m_{x, t + 1}$.
		\State Compute \\
		$\hat{d}_{y, t} = \frac{1}{\gamma}(\theta_{t + 1} - y_t) + \hat{\nabla}_y g(x_{t+1}, y_t; \varrho_{y, t})$ \\
		$ + \sum_{i = 1}^m \frac{\exp\left(\mu w_i \left(f_i(x_{t+1}, y_t; \xi_{y, t, i}) - z_i\right)\right) w_i \hat{\nabla}_y f_i(x_{t+1}, y_t; \xi_{y, t, i})}{c_t \sum_{j = 1}^m \exp\left(\mu w_j \left(f_j(x_{t+1}, y_t; \xi_{y, t, j}) - z_j\right)\right)}$.
		\State Update $m_{y, t + 1} = \beta m_{y, t} + (1 - \beta) \hat{d}_{y, t}$.
		\State Update $y_{t + 1} = y_t - \alpha_{y,t} m_{y, t + 1}$.
		\If{$t=R_T$}
		\State Cache $(\widehat\theta_T,\widehat x_T,\widehat y_T)=(\theta_{t+1},x_{t+1},y_{t+1})$.
		\EndIf
		\EndFor
	\end{algorithmic}
	\textbf{Output:} $\widehat\theta_T,\widehat x_T,\widehat y_T$ selected by \eqref{eq:stochastic_output_selection}.
\end{algorithm}

\subsection{Our \salg~Algorithm}
We propose a stochastic variant with momentum, \salg, targeting the lower-level stationary relaxation of \eqref{eq:s_mobl}. Its updates are described in Algorithm~\ref{alg:salg}.

In our Algorithm~\ref{alg:salg}, we replace all deterministic directions $d_\theta, d_x, d_y$ with mini-batch stochastic estimates $\hat{d}_\theta, \hat{d}_x, \hat{d}_y$ and update $\theta, x, y$ with the same Polyak-style momentum scheme~\citep{polyak1964some}, using a parameter $\beta\in(0,1)$ (taking $\theta$ as an example):
\begin{equation*}
	m_{\theta, t + 1} = \beta m_{\theta, t} + (1 - \beta) \hat{d}_{\theta, t}, \ \theta_{t + 1} = \theta_t - \alpha_{\theta,t} m_{\theta, t + 1}.
\end{equation*}
Here the positive step sizes $\alpha_{\theta,t},\alpha_{x,t},\alpha_{y,t}$ and the integer batch sizes $B_t\ge1$ are deterministic. In Algorithm~\ref{alg:salg}, an objective evaluated on a mini-batch denotes its sample average, and each hatted gradient is the corresponding sample-average gradient. In particular, the STCH weights are computed from the averaged objective values. An objective value and its gradient may use the same mini-batch.

For \salg, the step-size, penalty, and batch-size schedules are deterministic. The output index is sampled independently of all iteration sampling randomness with probabilities
\begin{equation}
	\label{eq:stochastic_output_selection}
	\omega_t:=\frac{\alpha_{\theta,t}^2}{c_t^2},\qquad
	\Pr(R_T=t)=\pi_{T,t}:=\frac{\omega_t}{\sum_{j\in\mathcal J_T}\omega_j},
	\quad t\in\mathcal J_T.
\end{equation}

\section{Theoretical Analysis}
\label{sec:theory}
We analyze the Moreau gap and Pareto stationarity residuals of \dalg~and \salg~under dynamically relaxed gap constraints for the lower-level stationary relaxation \eqref{eq:relaxed_mobl}. We first state the assumptions used in these bounds.
\begin{assumption}[UL Lower Bounded]
	\label{assumption:ul}
	For each $i \in [m]$, the upper-level objective $f_i(x, y)$ is proper and bounded below on $\mathbb{R}^{d_x} \times \mathbb{R}^{d_y}$. Consequently, the corresponding $F^{\text{(STCH)}}_w(x, y)$ is bounded below by a constant $\underline{F} > -\infty$ on $\mathbb{R}^{d_x} \times \mathbb{R}^{d_y}$.
\end{assumption}
\begin{assumption}[Smoothness]
	\label{assumption:ll}
	For any $i \in [m]$, $\nabla f_i(x, y)$ is $L_f$-Lipschitz continuous and bounded. $\nabla g(x, y)$ is $L_g$-Lipschitz continuous on $\mathbb{R}^{d_x} \times \mathbb{R}^{d_y}$. Consequently, $F^{\text{(STCH)}}_w(x, y)$ is differentiable and the following continuity relations hold:
	For any $(x^\prime, y^\prime), (x, y) \in \mathbb{R}^{d_x} \times \mathbb{R}^{d_y}$
	\begin{align}
		& \left\|\nabla F^{\text{(STCH)}}_w(x^\prime, y^\prime) - \nabla F^{\text{(STCH)}}_w(x, y)\right\|  \le  L_F \left\|(x^\prime, y^\prime) - (x, y)\right\|, \nonumber \\
		& \left\|\nabla g(x^\prime, y^\prime) - \nabla g(x, y)\right\| \le L_g \left\|(x^\prime, y^\prime) - (x, y)\right\|. \nonumber 
	\end{align}
\end{assumption}
The smoothness of $g(x, y)$ above implies its $(\rho_x, \rho_y)$-weak convexity, which follows from Lemma 5.7 of \citet{beck2017first}. 

\begin{assumption}[Unbiased Component Oracles with Bounded Variance]
	\label{assumption:var_ub}
	For every query point $z=(x,y)$ and $i\in[m]$, the single-sample objective and gradient oracles satisfy
	\begin{align*}
		\mathbb E[f_i(z;\xi)]&=f_i(z),
		&\mathbb E[|f_i(z;\xi)-f_i(z)|^2]&\le\nu_i^2,\\
		\mathbb E[\nabla f_i(z;\xi)]&=\nabla f_i(z),
		&\mathbb E[\|\nabla f_i(z;\xi)-\nabla f_i(z)\|^2]&\le\sigma_i^2,\\
		\mathbb E[\nabla g(z;\varrho)]&=\nabla g(z),
		&\mathbb E[\|\nabla g(z;\varrho)-\nabla g(z)\|^2]&\le\sigma_g^2,
	\end{align*}
	with finite constants independent of $z$ and $t$. Each mini-batch consists of $B_t$ independent, identically distributed samples and is independent of the history available before its direction is formed; the distinct mini-batches in Algorithm~\ref{alg:salg} are mutually independent. Within an upper-level mini-batch, objective values and gradients need not be independent.

\end{assumption}
$\nabla F^{\text{(STCH)}}$ involves a nonlinear softmax map, so unbiased component oracles generally produce biased composite directions. Lemma~\ref{lemma:component_batch_bias_variance} in Appendix~\ref{subsec:component_oracle_proof} derives their conditional bias and variance bounds from these primitive oracles.
\subsection{Stationarity Measure}
We measure feasibility and stationarity through the following $\varepsilon_c$-$\varepsilon_s$ Pareto stationarity under a relaxed Moreau gap constraint.
\begin{definition}[$\varepsilon_c$-$\varepsilon_s$ Pareto stationarity for Our Methods]
	\label{def:momeha_PS}
	Let $z=(x,y)$, $\varepsilon_c\ge0$, $\varepsilon_s\ge0$, and $\mathcal{F}_c\coloneqq\{z:q_\gamma(z)\le\varepsilon_c\}$. A point $z\in\mathcal{F}_c$ reaches $\varepsilon_c$-$\varepsilon_s$ Pareto stationarity if there exist $\lambda\in\Delta_{m-1}$ and $n\in\mathcal{N}_{\mathcal{F}_c}(z)$ such that
	\begin{equation}
		\left\|\sum_{i=1}^m\lambda_i\nabla f_i(z)+n\right\|\le\varepsilon_s.
	\end{equation}
	The corresponding minimum Pareto stationarity residual is
	\begin{equation}
		\label{eq:pareto_stationarity_residual}
		H(z;\varepsilon_c):=
		\inf_{\substack{\lambda\in\Delta_{m-1}\\n\in\mathcal{N}_{\mathcal{F}_c}(z)}}
		\left\|\sum_{i=1}^m\lambda_i\nabla f_i(z)+n\right\|.
	\end{equation}
	Equivalently, $H(z;\varepsilon_c)\le\varepsilon_s$.
\end{definition}
Here $\varepsilon_c$ bounds the Moreau gap, while $\varepsilon_s$ bounds the Pareto stationarity residual on the relaxed feasible set. When $\varepsilon_c=0$, this set is the exact envelope-constrained set.

For a fixed preference $w\in\Delta^{++}_{m-1}$, the convergence results use the STCH-based Pareto stationarity residual
\begin{equation}
	\label{eq:stch_pareto_residual}
	H_{w,p}(z;\varepsilon_c):=
	\left\|\sum_{i=1}^m\lambda_i(w,z)\nabla f_i(z)+p\nabla q_\gamma(z)\right\|.
\end{equation}
Here $\lambda(w,z)$ is the normalized STCH weight vector, and $p\ge0$ satisfies $p(q_\gamma(z)-\varepsilon_c)=0$ for $z\in\mathcal F_c$. This residual upper-bounds $H(z;\varepsilon_c)$. The weight definitions, the multiplier choice for our methods, and the relation to the penalty gradient are given in Appendix~\ref{sec:stationarity_details}.

For both algorithms, with $p_t$ defined in Appendix~\ref{subsec:stch_stationarity_residual}, set
\begin{equation}
	\label{eq:selected_output}
	\widehat z_T:=(\widehat x_T,\widehat y_T)=z_{R_T+1},\qquad
	\widehat\varepsilon_T:=q_\gamma(\widehat z_T),\qquad
	\widehat p_T:=p_{R_T}.
\end{equation}
The auxiliary output is $\widehat\theta_T=\theta_{R_T+1}$. Expectations for \salg~include both the update randomness and the independent output-index draw.

\subsection{Convergence Results}
The joint convergence bounds below balance the Moreau gap and the Pareto stationarity residual. Table~\ref{tab:wc_gradient_comparison} compares the methods at a common lower-level gradient accuracy. Accordingly, the joint convergence rows use $c_t=c_0(1+t)^{1/2}$ for both of our methods and $B=\Theta(T^{3/2})$ for \salg.
\begin{theorem}[Non-asymptotic Convergence Rate of \dalg]
	\label{theorem:dalg}
	Under Assumptions~\ref{assumption:ul} and~\ref{assumption:ll}, for any preference $w \in \Delta^{++}_{m - 1}$, suppose $\gamma \in \left(0, \frac{1}{2\rho_y}\right), c_t = c_0 > 0$, and $0 < \underline{\alpha}_\theta \le \alpha_{\theta, t} < \frac{2\gamma}{\gamma(L_g-\rho_y) + 2}$. Take $\overline\alpha_x=u_x$ and $\overline\alpha_y=u_y$ from \eqref{eq:d_descent_constants}--\eqref{eq:alpha_x_y_ub} in Lemma~\ref{lemma:potential_diff_ub}. Choose fixed lower bounds $0<\underline\alpha_x\le u_x$ and $0<\underline\alpha_y\le u_y$, and require $\alpha_{x,t}\in[\underline\alpha_x,u_x]$ and $\alpha_{y,t}\in[\underline\alpha_y,u_y]$ for all $t$. For $T\ge2$, the output selected by \eqref{eq:deterministic_output_selection} and \eqref{eq:selected_output} satisfies
	\[
	\widehat\varepsilon_T = \mathcal{O}\left(1\right),\quad H_{w,\widehat p_T}(\widehat z_T;\widehat\varepsilon_T) = \mathcal{O}\left(T^{-\frac{1}{2}}\right),
	\]
	i.e. $\mathcal{O}\left(1\right)$-$\mathcal{O}\left(T^{-\frac{1}{2}}\right)$-stationarity at the returned point.
	
	Furthermore, if $c_t = c_0(1 + t)^{\frac{1}{3}}$ and $\mathcal{P}_{c_t}(x_t, y_t)$ is uniformly upper-bounded along the resulting trajectory, with the same step-size bounds and output rule, then
	\begin{equation}
		\widehat\varepsilon_T = \mathcal{O}\left(T^{-\frac{1}{3}}\right), \quad H_{w,\widehat p_T}(\widehat z_T;\widehat\varepsilon_T) = \mathcal{O}\left(T^{-\frac{1}{3}}\right), \nonumber
	\end{equation}
	i.e. $\mathcal{O}\left(T^{-\frac{1}{3}}\right)$-$\mathcal{O}\left(T^{-\frac{1}{3}}\right)$-stationarity at the returned point.
\end{theorem}
\begin{remark}
	In single-objective setting, our refined analysis of the decay of the potential function imporves the balanced joint rate in the MEHA analysis~\citep[Theorem~3.4]{liu2024moreau} from $\mathcal O(T^{-1/4})$ to $\mathcal O(T^{-1/3})$ for the Moreau gap and penalty stationarity residual. Moreover, the upper-level gradient boundedness condition in Assumption~\ref{assumption:ll}, which is additional to the MEHA assumptions, can be omitted in this deterministic single-objective setting.
\end{remark}
\begin{theorem}[Non-asymptotic Convergence Rate of \salg]
	\label{theorem:salg}
	Under Assumptions~\ref{assumption:ul},~\ref{assumption:ll} and~\ref{assumption:var_ub}, let $w\in\Delta^{++}_{m-1}$ and $\beta\in(0,1)$ be independent of $T$. Suppose $\gamma\in(0,1/(2\rho_y))$, and let $\rho_h:=\gamma^{-1}-\rho_y$. Choose fixed constants $0<a\le a_{\max}(\beta)$, as specified in \eqref{eq:momentum_step_rule} of Lemma~\ref{lemma:momentum_stationarity}, and $0<\kappa_x\le\kappa_y\le\rho_h/8$, all independent of $T$. For every $t\ge0$, set
	\begin{equation}
		\label{eq:s_fixed_step_sizes}
		\alpha_{\theta,t}=a,\qquad
		\alpha_{x,t}=\kappa_xa^2,\qquad
		\alpha_{y,t}=\kappa_ya^2.
	\end{equation}
	For fixed penalty $c_t=c_0>0$ and a prescribed horizon $T\ge2$, use a fixed batch size $B_t=B\ge b_0T$ for $0\le t<T$, where $b_0>0$ is independent of $T$. The output selected by \eqref{eq:stochastic_output_selection} and \eqref{eq:selected_output} satisfies
	\[
	\mathbb E[\widehat\varepsilon_T]=\mathcal O(1),\qquad
	\mathbb E\left[H_{w,\widehat p_T}(\widehat z_T;\widehat\varepsilon_T)\right]
	=\mathcal O(T^{-1/2}),
	\]
	i.e. $\mathcal O(1)$-$\mathcal O(T^{-1/2})$-stationarity at the returned point.
	
	Furthermore, set $c_t=c_0(1+t)^{1/3}$, retaining the same constant step sizes. Use either $B_t=B\ge b_0T^{4/3}$ for $0\le t<T$, or $B_t\ge b_0(1+t)^{4/3}$. Suppose $\mathbb E[\mathcal P_{c_t}(x_t,y_t)]\le\overline{\mathcal P}$ uniformly over $1\le t\le T$, with a constant independent of $T$, including when the fixed batch size is chosen for each horizon. With the same output rule,
	\begin{equation}
		\label{eq:stochastic_growing_joint_rates}
		\mathbb E[\widehat\varepsilon_T]=\mathcal O(T^{-1/3}),\qquad
		\mathbb E\left[H_{w,\widehat p_T}(\widehat z_T;\widehat\varepsilon_T)\right]
		=\mathcal O(T^{-1/3}).
	\end{equation}
	i.e. $\mathcal O(T^{-1/3})$-$\mathcal O(T^{-1/3})$-stationarity for both batch-size choices. All expectations include the update randomness and the independent output-index draw.
\end{theorem}

\subsection{Lower-Level Optimality Under a PL Condition}
\label{subsec:pl_extension}
The preceding results apply to weakly convex, smooth lower-level objectives. We now give an optional extension that relates the Moreau gap to global lower-level suboptimality under a uniform Polyak--\L{}ojasiewicz (PL) condition.

\begin{assumption}[Uniform Lower-Level PL Condition]
	\label{assumption:lower_pl}
	For every $x\in\mathbb R^{d_x}$, the set $S_*(x)$ is nonempty. Write $g_*(x):=\min_{u\in\mathbb R^{d_y}}g(x,u)$. There is a constant $\mu_{\mathrm{PL}}>0$, independent of $x$ and $y$, such that
	\begin{equation}
		\label{eq:uniform_lower_pl}
		\frac12\|\nabla_y g(x,y)\|^2
		\ge\mu_{\mathrm{PL}}\bigl(g(x,y)-g_*(x)\bigr)
		\quad\text{for all }(x,y).
	\end{equation}
	In the stochastic setting, this assumption concerns the population lower-level objective $g(x,y)=\mathbb E_{\varrho}[g(x,y;\varrho)]$.
\end{assumption}
Write $\Delta_g(x,y):=g(x,y)-g_*(x)\ge0$. The PL condition permits a nonconvex lower-level objective and does not require a unique lower-level minimizer.

\begin{lemma}[Moreau Gap and Global Lower-Level Suboptimality]
	\label{lemma:pl_gap_comparison}
	Under Assumptions~\ref{assumption:ll} and~\ref{assumption:lower_pl}, suppose that $g(x,\cdot)$ is $\rho_y$-weakly convex and $\gamma\in(0,1/(2\rho_y))$. Let $L_h:=L_g+\gamma^{-1}$ and $C_{\mathrm{PL}}:=L_h/\mu_{\mathrm{PL}}$. Then, for every $(x,y)$,
	\begin{equation}
		\label{eq:pl_gap_comparison}
		0\le q_\gamma(x,y)\le\Delta_g(x,y)
		\le C_{\mathrm{PL}}q_\gamma(x,y).
	\end{equation}
	Consequently,
	\begin{equation}
		\label{eq:pl_feasible_equivalence}
		q_\gamma(x,y)=0
		\iff \nabla_y g(x,y)=0
		\iff y\in S_*(x),
	\end{equation}
	so $S_{\mathrm{stat}}(x)=S_*(x)$ for every $x$. In particular, the vector problems \eqref{eq:moreau_mobl} and \eqref{eq:d_mobl} have the same feasible set and objectives. The same conclusion holds for the problem \eqref{eq:s_mobl}.
\end{lemma}

\begin{corollary}[Lower-Level Optimality Guarantees Under PL]
	\label{corollary:pl_lower_optimality}
	Suppose Assumption~\ref{assumption:lower_pl} holds in addition to all the hypotheses and parameter choices of the relevant part of Theorem~\ref{theorem:dalg} or Theorem~\ref{theorem:salg}. Fix the preference vector, scalarization parameters, and envelope parameter $\gamma$ independently of $T$, and use the same output rule. Define $\widehat\Delta_{g,T}:=\Delta_g(\widehat x_T,\widehat y_T)$. At the returned point,
	\[
	0\le\widehat\Delta_{g,T}\le C_{\mathrm{PL}}\widehat\varepsilon_T.
	\]
	Under the fixed- and growing-penalty hypotheses of Theorem~\ref{theorem:dalg}, respectively,
	\[
		\widehat\Delta_{g,T}=\mathcal O(1), \quad
		\widehat\Delta_{g,T}=\mathcal O(T^{-1/3}).
	\]
	Under the corresponding hypotheses of Theorem~\ref{theorem:salg},
	\[
		\mathbb E[\widehat\Delta_{g,T}]=\mathcal O(1), \quad
		\mathbb E[\widehat\Delta_{g,T}]=\mathcal O(T^{-1/3}).
	\]
	The stationarity guarantees in both theorems remain valid at the same returned points, including the upper bounds on the minimum residual $H(\widehat z_T;\widehat\varepsilon_T)$. Expectations include the update randomness and the independent output-index draw.
\end{corollary}
The proofs of Lemma \ref{lemma:pl_gap_comparison} and Corollary \ref{corollary:pl_lower_optimality} are given in Appendix~\ref{subsec:pl_extension_proofs}.
\begin{remark}
	The PL condition converts the Moreau-gap bounds into global lower-level suboptimality guarantees, while preserving the convergence orders established in Theorems~\ref{theorem:dalg} and~\ref{theorem:salg}. Lower-level PL does not generally imply a PL condition for the Moreau gap, see Remark~\ref{remark:pl_gap_counterexample} in the appendix.
\end{remark}
\section{Numerical Experiments}
In this work, we conduct the multi-domain few-shot meta-learning and multi-objective neural architecture search (NAS) for validating \dalg~and \salg~respectively. Both experiments involve non-convex lower-level problems. The complete experimental details and supplementary results are provided in the appendix~\ref{sec:exp_details}.
\subsection{Multi-Domain Few-Shot Meta-Learning}
For the deterministic case, a 4-domain 5-way 5-shot meta-learning experiment is performed on the FC-100 data set~\citep{oreshkin2018tadam} for sensitivity analysis and the Caltech-256 data set~\citep{griffin2007caltech} for comparison, where the MOBL problem seeks domain-specific learners with the minimal adaptation losses on the support set and subsequently optimizes the meta model based on the learners with the query set.

\begin{figure}[htbp]
	\centering
	\begin{minipage}[b]{0.49\linewidth}
		\centering
		\includegraphics[width=\linewidth]{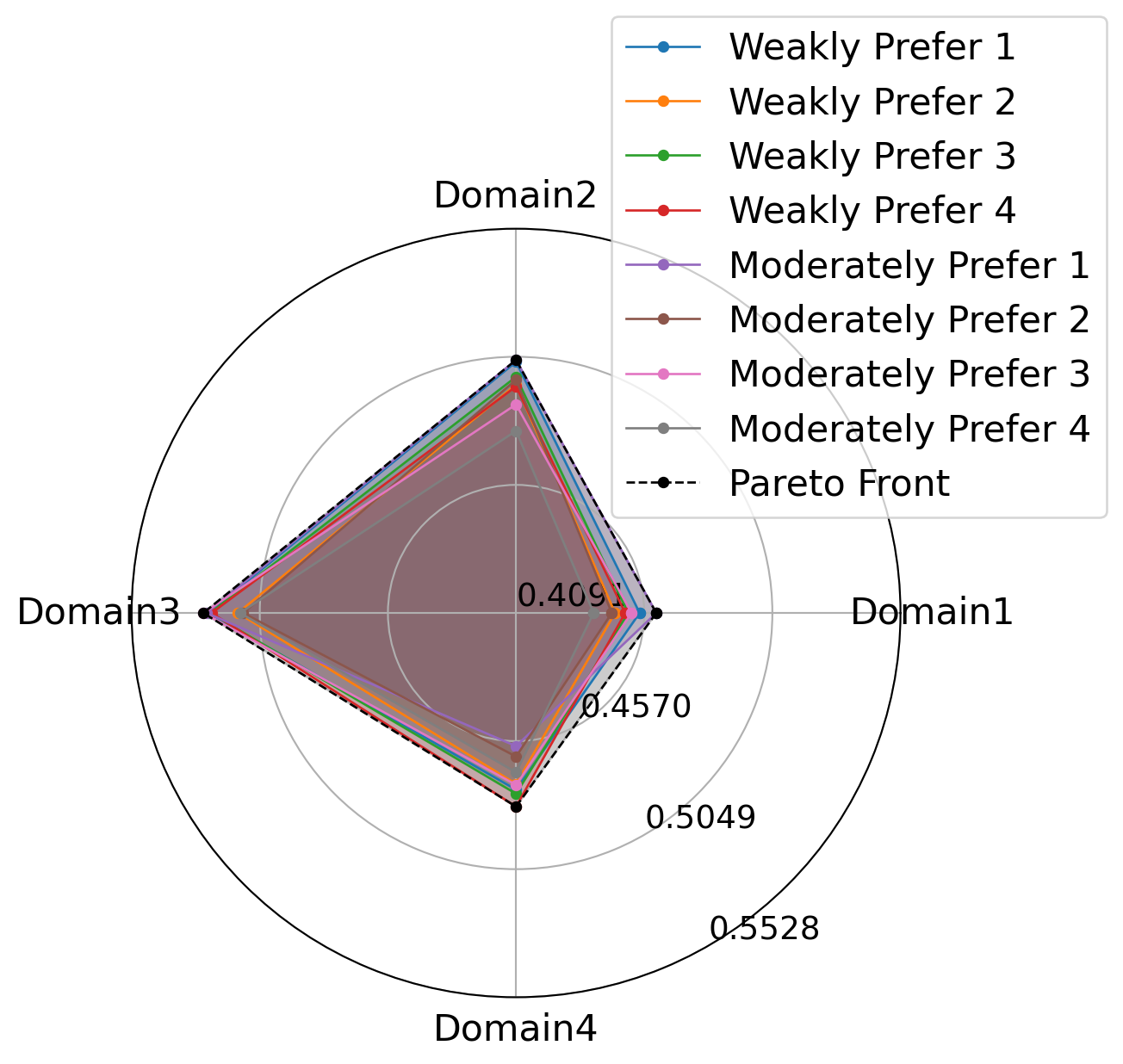}
		\caption{Pareto front exploration.}
		\label{fig:caltech_momeha_pareto}
	\end{minipage}\hfill
	\begin{minipage}[b]{0.49\linewidth}
		\centering
		\includegraphics[width=\linewidth]{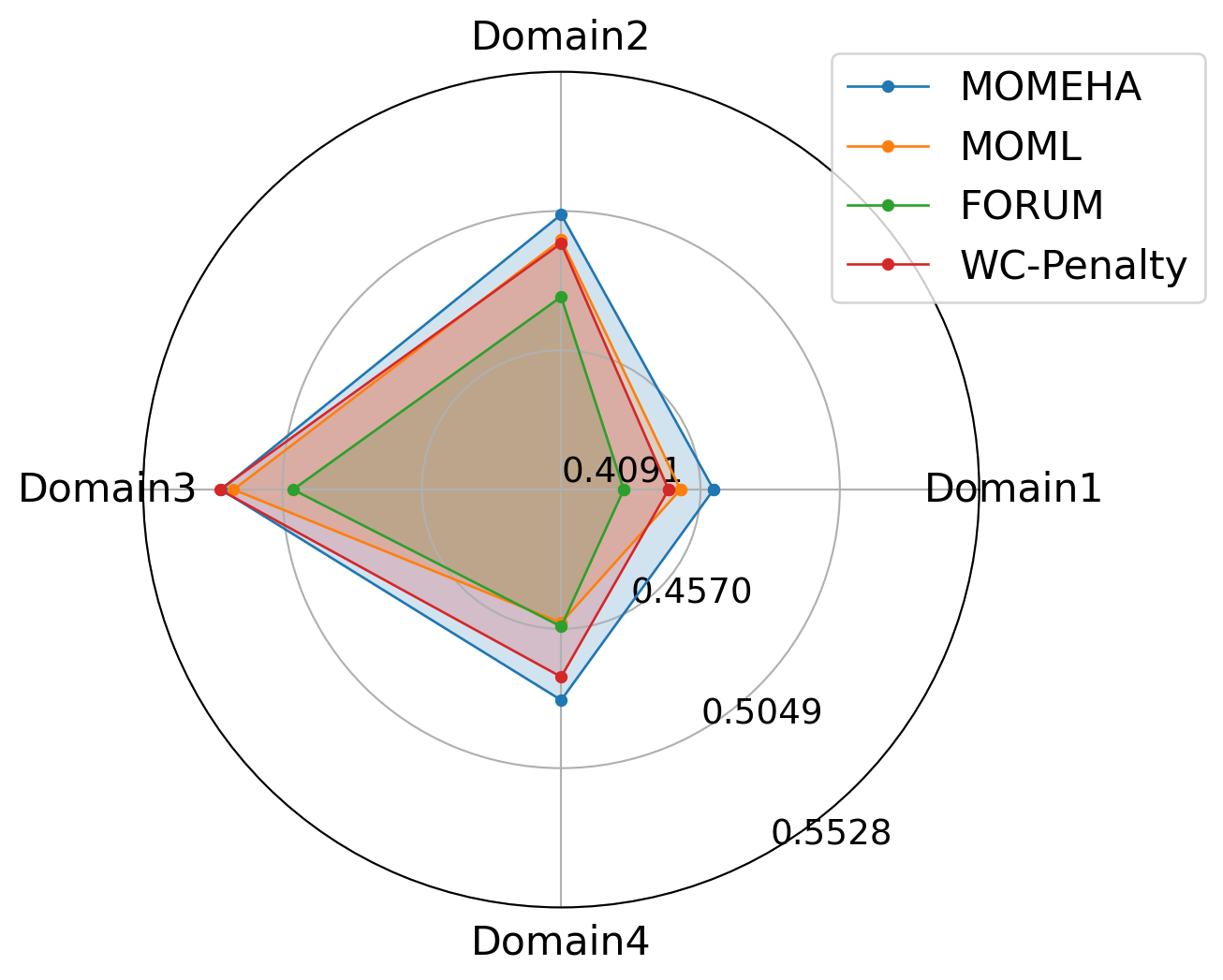}
		\caption{test accuracy in Meta-Learning.}
		\label{fig:caltech_comparsion}
	\end{minipage}
\end{figure}

Figure~\ref{fig:caltech_momeha_pareto} illustrates the Pareto front exploration of \dalg~under different preferences in the meta-learning comparison experiment. For each of the four domains, the optimal performance is achieved under the preference setting that favors the corresponding domain, demonstrating the effectiveness of the preference-guided search.
We only include the weak and moderate preference settings in the figure, as stronger preferences lead to a universal performance degradation across all domains, which is also observed in the WC-penalty baseline (see Appendix~\ref{subsec:meta_learning_details}).

Figure~\ref{fig:caltech_comparsion} compares the Pareto fronts (or single solutions) obtained by different algorithms in the meta-learning experiment. \dalg~achieves a broader coverage of the performance space on Domains 1, 2, and 4. On Domain 3, WC-penalty yields a marginally better front than our method, while both still outperform the remaining baselines. Table~\ref{tab:hv_caltech_comparison} also shows that \dalg~obtains the front with better quality than other baselines. WC-MHGD is excluded from the comparison due to its failure to converge in our experiment.

\begin{table}[htbp]
	\centering

	\begin{tabular}{l c c c c}
		\toprule
		\textbf{Alg.}   & \textbf{Ours} & WC-penalty    & MOML    & FORUM \\
		\midrule
		\textbf{HV}     & 1.127         & 1.092         & 1.072   & 1.013 \\
		\bottomrule
	\end{tabular}
	\par
	\caption{hypervolume comparison in Meta-Learning.}
	\label{tab:hv_caltech_comparison}
\end{table}
\subsection{Multi-Objective Neural Architecture Search}
For the stochastic setting, we conduct the NAS experiment on the CIFAR-10 data set~\citep{krizhevsky2009learning}. The experiment alternately optimizes the network parameters and the architecture weights to search for the optimal architecture under the different objectives. We consider 4 objectives in this experiment: validation loss, FLOPS loss, skip connection density and pooling density.
In addition to the full 4-objective setting, we also conduct a comparison experiment on the 2-objective setting that includes only the validation loss and FLOPS loss. 

\begin{figure}[tbp]
	\centering
	\includegraphics[width=0.65\columnwidth]{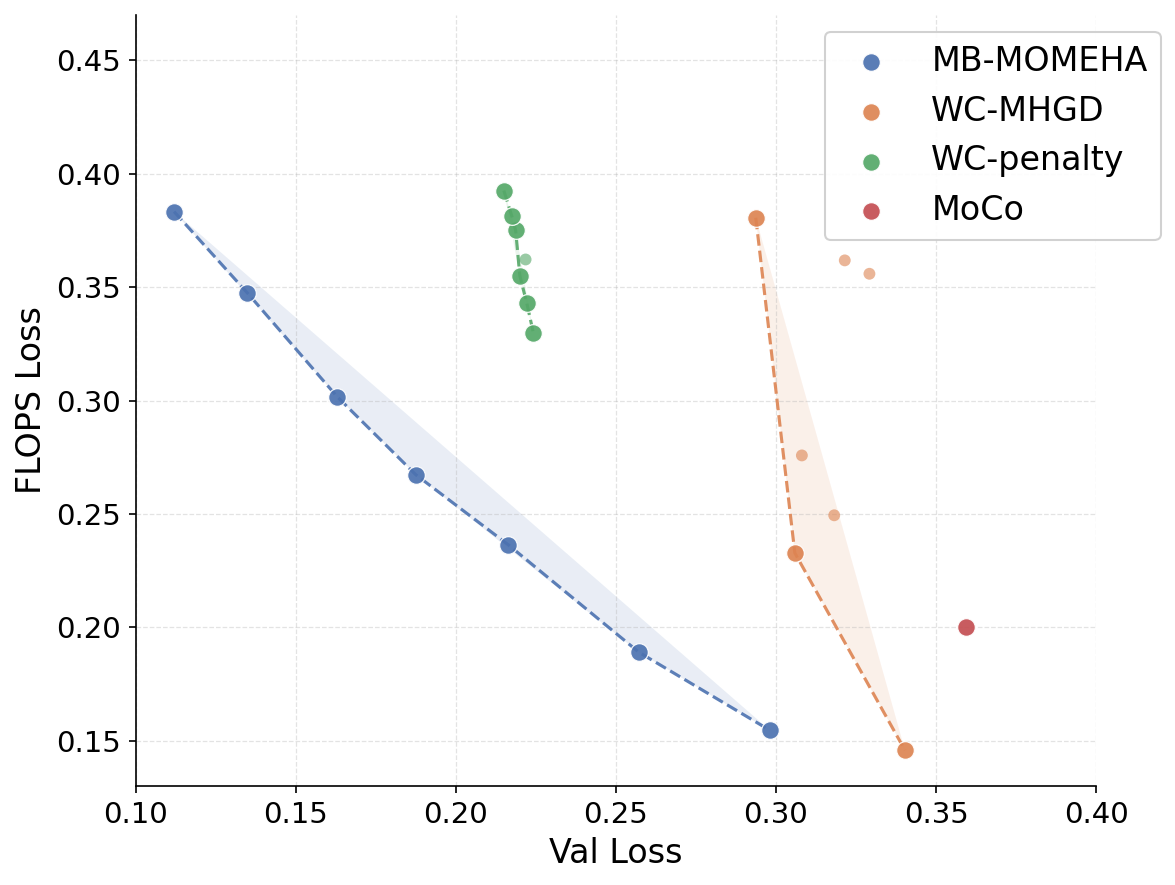}
	\caption{2-objective NAS comparison.}
	\label{fig:nas_2t_comparison}
\end{figure}

\begin{table}[tbp]
	\centering

	\begin{tabular}{l c c c c}
		\toprule
		\textbf{Alg.}   & \textbf{Ours} & WC-MHGD   & WC-penalty    & MoCo \\
		\midrule
		\textbf{HV}     & 1.522         & 1.323     & 1.216         & 1.192 \\
		\bottomrule
	\end{tabular}
	\par
	\caption{hypervolume comparison in 2-objective NAS.}
	\label{tab:hv_nas_comparison}
\end{table}

Figure~\ref{fig:nas_2t_comparison} shows the Pareto fronts (or the single solution) obtained by different algorithms in the 2-objective NAS experiment. Compared with the baselines, \salg~achieves superior performance in both the quality and the coverage of the obtained Pareto fronts, which corroborates the quantitative hypervolume comparison in Table~\ref{tab:hv_nas_comparison} and demonstrates the applicability and effectiveness of our algorithm in the non-convex lower-level problem with stochastic setting.

\section{Conclusion}
We developed \dalg~and \salg~for the lower-level stationary relaxation of multi-objective bilevel optimization with a nonconvex lower-level objective. The methods combine a Moreau gap constraint with smooth Tchebycheff scalarization to provide single-loop, Hessian-free updates guided by preference vectors.
Under the stated assumptions, our deterministic and stochastic analyses give finite-iteration bounds on the Moreau gap and Pareto stationarity residuals under dynamically relaxed gap constraints. Under an additional lower-level PL condition, the Moreau gap also controls global lower-level suboptimality, and the zero-gap feasible set coincides with that of the original bilevel problem. Empirically, the methods outperform the compared baselines on few-shot meta-learning and neural architecture search.




\newpage

\appendix
\section{Moreau Reformulation and Penalty Limits}
\label{sec:moreau_penalty_details}
This appendix gives the complete statements, proofs, and scope of the Moreau gap equivalence and the limit property of global penalty minimizers summarized in Section~\ref{sec:framework}.

\subsection{Moreau Gap and Lower-Level Stationarity}
\label{subsec:moreau_lower_stationarity}
The following smooth-case equivalence specifies the target of the Moreau formulation. It is also a specialization of Theorem A.2 of \citet{liu2024moreau}.
\begin{proposition}[Moreau constraint and lower-level stationarity]
	\label{prop:moreau_stationary_relaxation}
	Suppose $g$ is continuously differentiable and $g(x,\cdot)$ is $\rho_y$-weakly convex for each $x$. Let $\gamma\in(0,1/(2\rho_y))$, and denote the unique minimizer in the definition of $\upsilon_\gamma$ by $\theta_\gamma^*(x,y)$. Then $q_\gamma(x,y)\ge0$ for every $(x,y)$, and
	\begin{equation}
		\label{eq:moreau_stationary_equivalence}
		q_\gamma(x,y)=0
		\iff y=\theta_\gamma^*(x,y)
		\iff\nabla_yg(x,y)=0.
	\end{equation}
	If Assumption~\ref{assumption:ll} holds, then, with $\rho_h:=\gamma^{-1}-\rho_y$ and $L_h:=L_g+\gamma^{-1}$,
	\begin{equation}
		\label{eq:moreau_gap_gradient_comparison}
		\frac{\|\nabla_y g(x,y)\|^2}{2L_h}
		\le q_\gamma(x,y)
		\le\frac{\|\nabla_y g(x,y)\|^2}{2\rho_h}.
	\end{equation}
\end{proposition}
\begin{proof}
	For fixed $(x,y)$, let $h(\theta;x,y):=g(x,\theta)+\|\theta-y\|^2/(2\gamma)$. This function is $(1/\gamma-\rho_y)$-strongly convex, so its minimizer exists and is unique. Since $h(y;x,y)=g(x,y)$, evaluating the minimum at $\theta=y$ gives $q_\gamma(x,y)\ge0$. The equality $q_\gamma(x,y)=0$ holds exactly when $y$ is that unique minimizer. Strong convexity and differentiability make this equivalent to $\nabla_\theta h(y;x,y)=\nabla_yg(x,y)=0$.
	Under Assumption~\ref{assumption:ll}, $h$ is $L_h$-smooth and $\rho_h$-strongly convex in $\theta$. For fixed $(x,y)$, write $h(\theta)=h(\theta;x,y)$, $\theta^*=\theta_\gamma^*(x,y)$, and $a=\nabla h(y)=\nabla_y g(x,y)$. By the optimality of $\theta^*$ and $L_h$-smoothness,
	\[
	\begin{aligned}
		h(\theta^*)
		&\le h(y-L_h^{-1}a)\\
		&\le h(y)-\frac{\|a\|^2}{L_h}
		+\frac{L_h}{2}\|L_h^{-1}a\|^2\\
		&=h(y)-\frac{\|a\|^2}{2L_h}.
	\end{aligned}
	\]
	By $\rho_h$-strong convexity,
	\[
	\begin{aligned}
		h(\theta^*)
		&\ge h(y)+\langle a,\theta^*-y\rangle
		+\frac{\rho_h}{2}\|\theta^*-y\|^2\\
		&=h(y)-\frac{\|a\|^2}{2\rho_h}
		+\frac{\rho_h}{2}\left\|\theta^*-y+\frac{a}{\rho_h}\right\|^2\\
		&\ge h(y)-\frac{\|a\|^2}{2\rho_h}.
	\end{aligned}
	\]
	Since $q_\gamma(x,y)=h(y)-h(\theta^*)$, these inequalities prove \eqref{eq:moreau_gap_gradient_comparison}.
\end{proof}
\begin{remark}[Relation to the original bilevel feasible set]
	Proposition~\ref{prop:moreau_stationary_relaxation} shows that the Moreau-constrained problem \eqref{eq:moreau_mobl} is equivalent to the lower-level stationary relaxation \eqref{eq:relaxed_mobl}, with $(x,y)$ selected jointly. Differentiability gives $S_*(x)\subseteq S_{\mathrm{stat}}(x)$, and this inclusion can be strict for a nonconvex lower-level objective. Stationary points can include local maxima and saddle points. For example, $g(y)=1-\cos y$ is $1$-weakly convex. For $\gamma\in(0,1/2)$, $y=\pi$ satisfies $q_\gamma(y)=0$ while being a strict local maximum of $g$, whose global minimizers are $2k\pi$, $k\in\mathbb Z$.
	
	The original and stationary-relaxed feasible sets coincide when $S_{\mathrm{stat}}(x)=S_*(x)$ for every $x$, as holds when $g(x,\cdot)$ is convex or every lower-level stationary point is a global minimizer. This is a feasible-set equivalence. Section~\ref{subsec:pl_extension} gives a sufficient uniform lower-level PL condition and quantitative lower-level suboptimality bounds.
\end{remark}
\subsection{Limits of Global Penalty Minimizers}
\label{subsec:global_penalty_limits}
The following conditional statement concerns exact global minimizers of the penalty problems, as in Theorem A.4 of \citet{liu2024moreau}.
\begin{proposition}[Limits of global penalty minimizers]
	\label{prop:global_penalty_limit}
	Fix the scalarization parameters. Suppose $F_w^{\mathrm{(STCH)}}$ and $q_\gamma$ are continuous, $F_w^{\mathrm{(STCH)}}\ge\underline F$, and $\{z=(x,y):q_\gamma(x,y)=0\}$ is nonempty. Let $c_k>0$ tend to infinity, and suppose global minimizers can be chosen as
	\[
	z_k^\star\in\arg\min_z\mathcal P_{c_k}(z).
	\]
	If this sequence has an accumulation point, every such point is a global minimizer of the scalarized stationary-relaxation problem \eqref{eq:momeha}.
\end{proposition}
\begin{proof}
	For any $u$ satisfying $q_\gamma(u)=0$, global optimality gives
	\[
	\mathcal P_{c_k}(z_k^\star)\le\mathcal P_{c_k}(u)=F_w^{\mathrm{(STCH)}}(u)-\underline F.
	\]
	Thus $0\le q_\gamma(z_k^\star)\le(F_w^{\mathrm{(STCH)}}(u)-\underline F)/c_k\to0$ and $F_w^{\mathrm{(STCH)}}(z_k^\star)\le F_w^{\mathrm{(STCH)}}(u)$. Along a subsequence converging to an accumulation point $\bar z$, continuity yields $q_\gamma(\bar z)=0$ and $F_w^{\mathrm{(STCH)}}(\bar z)\le F_w^{\mathrm{(STCH)}}(u)$ for every feasible $u$, proving the assertion.
\end{proof}

\subsection{Proofs and Scope of the PL Extension}
\label{subsec:pl_extension_proofs}

\begin{proof}[Proof of Lemma~\ref{lemma:pl_gap_comparison}]
	Fix $(x,y)$ and set $h(\theta):=g(x,\theta)+\|\theta-y\|^2/(2\gamma)$. Weak convexity and the restriction on $\gamma$ make $h$ strongly convex with a unique minimizer, and Assumption~\ref{assumption:ll} makes it $L_h$-smooth in $\theta$, where $L_h=L_g+\gamma^{-1}>0$. At $\theta=y$, we have $h(y)=g(x,y)$ and $\nabla h(y)=\nabla_y g(x,y)$. The smoothness inequality evaluated at $y-L_h^{-1}\nabla h(y)$ gives
	\[
	\upsilon_\gamma(x,y)=\min_\theta h(\theta)
	\le h\bigl(y-L_h^{-1}\nabla h(y)\bigr)
	\le h(y)-\frac{1}{2L_h}\|\nabla_y g(x,y)\|^2.
	\]
	Therefore, using Assumption~\ref{assumption:lower_pl},
	\begin{equation}
		\label{eq:pl_gap_gradient_lower_bound}
		q_\gamma(x,y)
		\ge\frac{1}{2L_h}\|\nabla_y g(x,y)\|^2
		\ge\frac{\mu_{\mathrm{PL}}}{L_h}\Delta_g(x,y).
	\end{equation}
	For every $\theta$, $g(x,\theta)+\|\theta-y\|^2/(2\gamma)\ge g_*(x)$, so $\upsilon_\gamma(x,y)\ge g_*(x)$ and $q_\gamma(x,y)\le\Delta_g(x,y)$. Nonnegativity of $q_\gamma$ follows from Proposition~\ref{prop:moreau_stationary_relaxation}. These inequalities prove \eqref{eq:pl_gap_comparison} with a constant uniform in $x$ and $y$.
	
	At any lower-level stationary point, the PL inequality and $\Delta_g\ge0$ imply $\Delta_g=0$, so the point is a global lower-level minimizer. Conversely, differentiability of $g(x,\cdot)$ on $\mathbb R^{d_y}$ makes every global minimizer stationary. Combining these facts with Proposition~\ref{prop:moreau_stationary_relaxation} proves \eqref{eq:pl_feasible_equivalence}. Since all formulations use the same joint variables and vector objectives, the claimed feasible-set and problem equivalences follow.
\end{proof}

\begin{proof}[Proof of Corollary~\ref{corollary:pl_lower_optimality}]
	Apply \eqref{eq:pl_gap_comparison} at the selected output of the relevant algorithm. For \dalg, it gives
	\[
	0\le\widehat\Delta_{g,T}
	\le C_{\mathrm{PL}}q_\gamma(\widehat z_T)
	=C_{\mathrm{PL}}\widehat\varepsilon_T.
	\]
	For \salg, the same inequality holds for every realization. The expectations are well-defined: continuity of $g(x,\cdot)$ allows
	\[
	g_*(x)=\inf_{u\in\mathbb Q^{d_y}}g(x,u),
	\]
	so $g_*$ is Borel measurable as a countable infimum of continuous functions of $x$. It is finite because $S_*(x)$ is nonempty, and the gap bounds in Theorem~\ref{theorem:salg} and the pointwise comparison ensure integrability at the returned point. Taking expectations yields
	\[
	0\le\mathbb E[\widehat\Delta_{g,T}]
	\le C_{\mathrm{PL}}\mathbb E[\widehat\varepsilon_T].
	\]
	The constant $C_{\mathrm{PL}}=(L_g+\gamma^{-1})/\mu_{\mathrm{PL}}$ is independent of $T$, since $\gamma$ and the uniform PL constant are fixed. Substituting the fixed- and growing-penalty gap bounds of Theorems~\ref{theorem:dalg} and~\ref{theorem:salg} proves all four lower-level suboptimality bounds. The corresponding stationarity bounds follow directly from the same theorems, at precisely the same selected outputs, and \eqref{eq:pareto_residual_penalty_bound} gives the assertion about the minimum residual $H$.
\end{proof}

\begin{remark}[Consequences of Vanishing Lower-Level Suboptimality]
	Under the growing-penalty hypotheses of the deterministic part of Corollary~\ref{corollary:pl_lower_optimality}, every accumulation point of the sequence of returned points is feasible for the original bilevel problem. Indeed, if $\widehat z_{T_k}\to\bar z=(\bar x,\bar y)$ with $T_k\to\infty$, continuity of $q_\gamma$ and $q_\gamma(\widehat z_{T_k})\to0$ give $q_\gamma(\bar z)=0$, and Lemma~\ref{lemma:pl_gap_comparison} gives $\bar y\in S_*(\bar x)$. This argument requires no continuity assumption on the value function $g_*$.
	
	In the stochastic growing-penalty case, the nonnegative error converges to zero in expectation and hence in probability: for every fixed $\eta>0$, Markov's inequality gives
	\[
	\Pr\bigl(\widehat\Delta_{g,T}>\eta\bigr)
	\le\frac{\mathbb E[\widehat\Delta_{g,T}]}{\eta}
	=\mathcal O(T^{-1/3}),
	\]
	under the growing-penalty schedule in Theorem~\ref{theorem:salg}.
\end{remark}

\begin{remark}[Lower-Level PL Does Not Imply a PL Gap]
	\label{remark:pl_gap_counterexample}
	Consider the $x$-independent one-dimensional lower-level objective
	\[
	g(y)=
	\begin{cases}
		y^2/2,&y\le1,\\
		y-1/2,&1\le y\le2,\\
		(y-1)^2/2+1,&y\ge2.
	\end{cases}
	\]
	The values and derivatives match at the two junctions. Its derivative is $y$ for $y\le1$, $1$ on $[1,2]$, and $y-1$ for $y\ge2$, so $g$ is convex with a $1$-Lipschitz gradient and $g_*=g(0)=0$. It satisfies the global PL inequality with $\mu_{\mathrm{PL}}=1/3$: for $y\le1$, $|g'(y)|^2/2=g(y)$. On $[1,2]$, $|g'(y)|^2/2=1/2\ge g(y)/3$. And for $y\ge2$, writing $u=y-1\ge1$ gives $u^2/2\ge(u^2/2+1)/3$.
	
	For $0<\gamma<1/2$ and $y\in(1+\gamma,2)$, the point $\theta=y-\gamma\in(1,2)$ satisfies the first-order condition for the strongly convex proximal subproblem. Hence
	\[
	\theta_\gamma^*(y)=y-\gamma,\qquad
	\upsilon_\gamma(y)=y-\frac12-\frac\gamma2,\qquad
	q_\gamma(y)=\frac\gamma2>0,\qquad q_\gamma'(y)=0.
	\]
	The function is also $1$-weakly convex, so this choice of $\gamma$ satisfies the envelope-parameter restriction used in the analysis. Consequently, no positive PL constant for $q_\gamma$ can satisfy $|q_\gamma'(y)|^2/2\ge\mu_q(q_\gamma(y)-\min q_\gamma)$ for all $y$, since $\min q_\gamma=q_\gamma(0)=0$. This example rules out inferring a PL condition for the gap from Assumption~\ref{assumption:lower_pl}. It explains why a faster rate cannot be obtained simply by applying a PL gradient-descent estimate to the gap.
\end{remark}

\section{Details of Pareto Stationarity and Residual Bounds}
\label{sec:stationarity_details}
This appendix gives the details behind Definitions~\ref{def:c_PS} and~\ref{def:momeha_PS}, and the normal vectors and residual bounds used in both convergence theorems.

The penalty method produces iterates that may violate the constraint in the problem~\eqref{eq:momeha}, so the normal cone to its feasible set is empty at those iterates. This hinders us from directly using Definition~\ref{def:c_PS} as a convergence indicator. For the vector objective $F(x,y)$, Definition~\ref{def:momeha_PS} instead expresses $\varepsilon_s$-Pareto stationarity on the dynamically relaxed feasible region of the lower-level stationary formulation:
\begin{align}
	\label{eq:momeha_c_relaxed}
	&\min_{(x, y) \in \mathbb{R}^{d_x} \times \mathbb{R}^{d_y}} F(x,y)\\
	&\text{s.t.} \quad g(x, y) - \upsilon_\gamma(x, y) \le \varepsilon_c. \nonumber 
\end{align}

\subsection{Normal Form, Optimality Necessity, and Residual Attainment}
\label{subsec:stationarity_normal_form}
For a closed set $C$ and $z\in C$, the regular normal cone is
\[
\widehat{\mathcal N}_C(z)
:=\left\{n:\limsup_{\substack{u\to z\\u\in C,\ u\ne z}}
\frac{\langle n,u-z\rangle}{\|u-z\|}\le0\right\}.
\]
The limiting normal cone used throughout the paper is
\[
\mathcal N_C(z)
:=\left\{n:\exists\,z^k\in C,\ n^k\in\widehat{\mathcal N}_C(z^k),
\ z^k\to z,\ n^k\to n\right\}.
\]
In particular, $\widehat{\mathcal N}_C(z)\subseteq\mathcal N_C(z)$, and the limiting normal cone at a fixed feasible point is closed and contains zero.

\begin{proposition}[Normal-cone necessity for weak Pareto optimality]
	\label{prop:pareto_normal_necessity}
	Let $C$ be closed and let $f_1,\ldots,f_m$ be continuously differentiable near $z^\star\in C$. If $z^\star$ is a local weak Pareto optimum of $(f_1,\ldots,f_m)$ over $C$, then there exist $\lambda\in\Delta_{m-1}$ and $n\in\mathcal N_C(z^\star)$ such that
	\[
	\sum_{i=1}^m\lambda_i\nabla f_i(z^\star)+n=0.
	\]
\end{proposition}
\begin{proof}
	Define $\phi(z):=\max_{i\in[m]}(f_i(z)-f_i(z^\star))$. Local weak Pareto optimality implies $\phi(z)\ge0=\phi(z^\star)$ for all nearby $z\in C$, so $z^\star$ locally minimizes $\phi+\delta_C$, where $\delta_C$ is the indicator function of $C$.
	Write $\partial$ for the limiting subdifferential. The finite maximum of smooth functions is locally Lipschitz, and all its terms are active at $z^\star$, giving
	\[
	\partial\phi(z^\star)
	=\operatorname{conv}\{\nabla f_i(z^\star):i\in[m]\}.
	\]
	The nonsmooth Fermat rule and the subdifferential sum rule give
	\[
	0\in\partial(\phi+\delta_C)(z^\star)
	\subseteq\partial\phi(z^\star)+\mathcal N_C(z^\star).
	\]
	The sum-rule qualification is automatic because $\phi$ is locally Lipschitz. No constraint qualification on a smooth description of $C$ is required. Expressing the resulting element of $\partial\phi(z^\star)$ as a convex combination proves the claim.
\end{proof}

The gap $q_\gamma$ is continuously differentiable under the smoothness and envelope-parameter conditions used in the analysis. Hence $\mathcal F_c=\{z:q_\gamma(z)\le\varepsilon_c\}$ is closed, and Proposition~\ref{prop:pareto_normal_necessity} applies to \eqref{eq:momeha_c_relaxed}. In particular, any local weak Pareto optimum of that problem has $H(z;\varepsilon_c)=0$.

\begin{lemma}[Attainment of the minimum Pareto residual]
	\label{lemma:pareto_residual_attainment}
	For every $z\in\mathcal F_c$, the infimum in \eqref{eq:pareto_stationarity_residual} is attained. Thus Definition~\ref{def:momeha_PS} is equivalent to $H(z;\varepsilon_c)\le\varepsilon_s$ for every $\varepsilon_s\ge0$.
\end{lemma}
\begin{proof}
	Since $0\in\mathcal N_{\mathcal F_c}(z)$, the infimum is finite. Let $(\lambda^k,n^k)$ be a minimizing sequence. The convex combinations of these fixed gradients have norm at most $\max_i\|\nabla f_i(z)\|$, so
	\[
	\|n^k\|
	\le\left\|\sum_i\lambda_i^k\nabla f_i(z)+n^k\right\|+\max_i\|\nabla f_i(z)\|.
	\]
	Thus the normal vectors are bounded. By this boundedness and compactness of the simplex, there is a subsequence with $\lambda^k\to\lambda\in\Delta_{m-1}$ and $n^k\to n$. Closedness of $\mathcal N_{\mathcal F_c}(z)$ implies $n\in\mathcal N_{\mathcal F_c}(z)$, and continuity of the norm shows that $(\lambda,n)$ attains the infimum. This also proves the claimed equivalence when $\varepsilon_s=0$.
\end{proof}

\subsection{Normal Vectors Without a Constraint Qualification}
\label{subsec:gap_normal_inclusion}
\begin{lemma}[Complementary gradient multipliers give normal vectors]
	\label{lemma:gap_normal_inclusion}
	Let $q_\gamma$ be continuously differentiable, $\varepsilon_c\ge0$, and $z\in\mathcal F_c$. For every $p\ge0$ satisfying $p(q_\gamma(z)-\varepsilon_c)=0$,
	\begin{equation}
		\label{eq:gap_normal_inclusion}
		p\nabla q_\gamma(z)\in\widehat{\mathcal N}_{\mathcal F_c}(z)
		\subseteq\mathcal N_{\mathcal F_c}(z).
	\end{equation}
\end{lemma}
\begin{proof}
	If $q_\gamma(z)<\varepsilon_c$, complementarity gives $p=0$, and the inclusion follows. Otherwise $q_\gamma(z)=\varepsilon_c$. For nearby $u\in\mathcal F_c$, differentiability gives
	\[
	0\ge q_\gamma(u)-q_\gamma(z)
	=\langle\nabla q_\gamma(z),u-z\rangle+r(u),
	\qquad \frac{r(u)}{\|u-z\|}\to0.
	\]
	Multiplying by $p\ge0$ yields
	\[
	\limsup_{\substack{u\to z\\u\in\mathcal F_c,\ u\ne z}}
	\frac{\langle p\nabla q_\gamma(z),u-z\rangle}{\|u-z\|}
	\le0.
	\]
	This is the regular normal-cone condition.
\end{proof}

\begin{remark}[Degeneracy of the exact gap constraint]
	\label{remark:gap_multiplier_equivalence}
	Since $q_\gamma\ge0$ is continuously differentiable, $q_\gamma(z)=0$ implies $\nabla q_\gamma(z)=0$. Thus the Mangasarian--Fromovitz constraint qualification (MFCQ) for the exact constraint $q_\gamma\le0$ fails at every feasible point. Definition~\ref{def:momeha_PS} retains the full normal cone, which need not be generated by the constraint gradient at a degenerate boundary point. The convergence proofs use only the inclusion in Lemma~\ref{lemma:gap_normal_inclusion} and require no complete gradient-multiplier representation.
\end{remark}

\subsection{Pareto Stationarity Residual and the Penalty Gradient}
\label{subsec:stch_stationarity_residual}
Fix $w\in\Delta^{++}_{m-1}$ and define
\begin{align}
	\label{eq:stch_pareto_weights}
	&\tau_i(z):=\frac{\exp\bigl(\mu w_i(f_i(z)-z_i)\bigr)}
	{\sum_{j=1}^m\exp\bigl(\mu w_j(f_j(z)-z_j)\bigr)},\\
	&s_w(z):=\sum_{i=1}^m w_i\tau_i(z),\qquad
	\lambda_i(w,z):=\frac{w_i\tau_i(z)}{s_w(z)}.\nonumber
\end{align}
Here $z_i$ are the reference values in the scalarization, and $\underline w:=\min_iw_i>0$.

For either algorithm, write $z_t=(x_t,y_t)$ and set the dynamic tolerance to $q_\gamma(z_{t+1})$. We use the explicit multipliers
\begin{equation}
	\label{eq:pareto_residual_multiplier}
	p_t:=\begin{cases}
		c_t/s_w(z_{t+1}),&q_\gamma(z_{t+1})>0,\\
		0,&q_\gamma(z_{t+1})=0.
	\end{cases}
\end{equation}
With $D_t:=\|\nabla\mathcal P_{c_t}(z_{t+1})\|$, the following lemma establishes the pointwise relation:
\begin{equation}
	\label{eq:pareto_residual_penalty_bound}
	H(z_{t+1};q_\gamma(z_{t+1}))
	\le H_{w,p_t}(z_{t+1};q_\gamma(z_{t+1}))
	=\frac{D_t}{s_w(z_{t+1})}\le\frac{D_t}{\underline w}.
\end{equation}

\begin{lemma}[Pareto stationarity residual bound]
	\label{lemma:stch_residual_bound}
	Fix $w\in\Delta^{++}_{m-1}$. For every $z\in\mathcal F_c$ and nonnegative complementary multiplier $p$,
	\[
	H(z;\varepsilon_c)\le H_{w,p}(z;\varepsilon_c).
	\]
	With the dynamic tolerance $\varepsilon_c=q_\gamma(z_{t+1})$ and $p_t$ in \eqref{eq:pareto_residual_multiplier}, the pointwise relation \eqref{eq:pareto_residual_penalty_bound} holds for every iterate.
\end{lemma}
\begin{proof}
	Differentiating the STCH objective gives
	\begin{equation}
		\label{eq:grad_S_mu_w}
		\nabla F_w^{\mathrm{(STCH)}}(z)
		=\sum_{i=1}^m w_i\tau_i(z)\nabla f_i(z).
	\end{equation}
	The weights in \eqref{eq:stch_pareto_weights} satisfy
	\begin{equation}
		\label{eq:pareto_stationary_simplex_1}
		\sum_i\tau_i(z)=1,\qquad
		s_w(z)=\sum_iw_i\tau_i(z)\ge\underline w>0.
	\end{equation}
	Hence the vector with entries $\lambda_i(w,z)=w_i\tau_i(z)/s_w(z)$ belongs to the simplex, and
	\begin{equation}
		\label{eq:pareto_stationary_simplex_2}
		\nabla F_w^{\mathrm{(STCH)}}(z)
		=s_w(z)\sum_i\lambda_i(w,z)\nabla f_i(z).
	\end{equation}
	By Lemma~\ref{lemma:gap_normal_inclusion}, $p\nabla q_\gamma(z)\in\mathcal N_{\mathcal F_c}(z)$. These weights and this normal form an admissible pair in the infimum defining $H$, which proves $H\le H_{w,p}$.
	
	For the dynamic tolerance, $p_t$ is nonnegative and complementary because $q_\gamma(z_{t+1})=\varepsilon_c$. If $q_\gamma(z_{t+1})>0$, the choice $p_t=c_t/s_w(z_{t+1})$ gives
	\[
	\sum_i\lambda_i(w,z_{t+1})\nabla f_i(z_{t+1})
	+p_t\nabla q_\gamma(z_{t+1})
	=\frac{\nabla F_w^{\mathrm{(STCH)}}(z_{t+1})
		+c_t\nabla q_\gamma(z_{t+1})}{s_w(z_{t+1})}
	=\frac{\nabla\mathcal P_{c_t}(z_{t+1})}{s_w(z_{t+1})}.
	\]
	If $q_\gamma(z_{t+1})=0$, nonnegativity and differentiability of $q_\gamma$ imply $\nabla q_\gamma(z_{t+1})=0$. The same identity then holds with $p_t=0$, and the chosen normal is $0\in\mathcal N_{\{q_\gamma\le0\}}(z_{t+1})$. Taking norms and using $s_w\ge\underline w$ proves \eqref{eq:pareto_residual_penalty_bound}.
\end{proof}

\section{Theoretical Proof Details for \dalg}
Throughout this section, Assumptions~\ref{assumption:ul} and~\ref{assumption:ll} hold, $g(x,\cdot)$ is $\rho_y$-weakly convex, $\gamma\in(0,1/(2\rho_y))$, and the penalty parameters and step sizes are positive.

\subsection{Auxiliary Lemmas}
\begin{lemma}
	\label{lemma:theta_lipschitz}
	Denote $h(\theta; x, y) \coloneqq g(x, \theta) + \frac{1}{2\gamma}\left\|\theta - y\right\|^2$. Under Assumption~\ref{assumption:ll}, suppose that $g(x, y)$ is $(\rho_x, \rho_y)$-weakly convex and $\gamma \in \left(0, \frac{1}{2\rho_y}\right)$. For all $(x^\prime, y^\prime), (x, y) \in \mathbb{R}^{d_x} \times \mathbb{R}^{d_y}$ we have:
	\begin{equation*}
		\left\|\theta^*_\gamma(x^\prime, y^\prime) - \theta^*_\gamma(x, y)\right\| \le L_\theta \left\|(x^\prime, y^\prime) - (x, y)\right\|,
	\end{equation*}
	where $L_\theta \coloneqq \frac{1}{\gamma \rho_h} + \frac{L_g}{\rho_h}$.
\end{lemma}
\begin{proof}
	Under Assumption~\ref{assumption:ll}, $h(\theta; x, y)$ is $\left(\frac{1}{\gamma} - \rho_y\right)$-strongly convex (denoted as $\rho_h$-strongly convex) and $\left(L_g + \frac{1}{\gamma}\right)$-smooth (denoted as $L_h$-smooth) in $\theta$.
	By the definition of strong convexity and the monotonicity of gradients, 
	for any $\theta^\prime, \theta \in \mathbb{R}^{d_y}$, we have
	\begin{equation*}
		\rho_h \|\theta^\prime - \theta\|^2 
		\le \left\langle \nabla_y g(x, \theta^\prime) - \nabla_y g(x, \theta) + \frac{1}{\gamma}(\theta^\prime - \theta),~ \theta^\prime - \theta \right\rangle.
	\end{equation*}
	Applying the Cauchy--Schwarz inequality yields
	\begin{equation*}
		\rho_h \|\theta^\prime - \theta\|^2 
		\le \left\|\nabla_y g(x, \theta^\prime) - \nabla_y g(x, \theta) + \frac{1}{\gamma}(\theta^\prime - \theta)\right\| \|\theta^\prime - \theta\|,
	\end{equation*}
	which simplifies to
	\begin{equation}
		\label{eq:theta_lipschitz_1}
		\rho_h \|\theta^\prime - \theta\| 
		\le \left\|\nabla_y g(x, \theta^\prime) - \nabla_y g(x, \theta) + \frac{1}{\gamma}(\theta^\prime - \theta)\right\|.
	\end{equation}
	
	Meanwhile, the first-order optimality condition of $h$ implies that for any $(x,y)$,
	\begin{equation*}
		\nabla_y g(x, \theta^*_\gamma(x, y)) + \frac{1}{\gamma}(\theta^*_\gamma(x, y) - y) = 0.
	\end{equation*}
	Subtracting the equation at $(x, y)$ from that at $(x^\prime, y^\prime)$, we obtain
	\begin{equation}
		\label{eq:theta_lipschitz_2}
		\begin{aligned}
			& \nabla_y g(x^\prime, \theta^*_\gamma(x^\prime, y^\prime)) - \nabla_y g(x, \theta^*_\gamma(x, y)) + \frac{1}{\gamma}(\theta^*_\gamma(x^\prime, y^\prime) - \theta^*_\gamma(x, y)) = \frac{1}{\gamma}(y^\prime - y).
		\end{aligned}
	\end{equation}
	
	To control the right-hand side of \eqref{eq:theta_lipschitz_1}, we decompose the gradient difference in \eqref{eq:theta_lipschitz_2} as
	\begin{equation*}
		\begin{aligned}
			& ~ \nabla_y g(x^\prime, \theta^*_\gamma(x^\prime, y^\prime)) - \nabla_y g(x, \theta^*_\gamma(x, y)) \\
			= & ~ \nabla_y g(x^\prime, \theta^*_\gamma(x^\prime, y^\prime)) - \nabla_y g(x^\prime, \theta^*_\gamma(x, y)) + ~ \nabla_y g(x^\prime, \theta^*_\gamma(x, y)) - \nabla_y g(x, \theta^*_\gamma(x, y)).
		\end{aligned}
	\end{equation*}
	Substituting this decomposition into \eqref{eq:theta_lipschitz_2} gives
	\begin{equation*}
		\begin{aligned}
			& \nabla_y g(x^\prime, \theta^*_\gamma(x^\prime, y^\prime)) - \nabla_y g(x^\prime, \theta^*_\gamma(x, y)) \\
			= & \frac{1}{\gamma}(y^\prime-y) - \frac{1}{\gamma}\left(\theta^*_\gamma(x^\prime,y^\prime)-\theta^*_\gamma(x,y)\right) \\
			& - \left(\nabla_y g(x^\prime, \theta^*_\gamma(x, y)) - \nabla_y g(x, \theta^*_\gamma(x, y))\right).
		\end{aligned}
	\end{equation*}
	
	Set $\theta^\prime=\theta^*_\gamma(x^\prime,y^\prime)$ and $\theta=\theta^*_\gamma(x,y)$. Apply \eqref{eq:theta_lipschitz_1} with the upper-level variable fixed at $x^\prime$. The identity above, the triangle inequality, and the $L_g$-smoothness of $g$ give
	\begin{equation}
		\label{eq:theta_lipschitz_3}
		\begin{aligned}
			& ~ \left\|\nabla_y g(x^\prime, \theta^\prime) - \nabla_y g(x^\prime, \theta) + \frac{1}{\gamma}(\theta^\prime - \theta)\right\| \\
			= & ~ \left\|\frac{1}{\gamma}(y^\prime-y) - \left(\nabla_y g(x^\prime,\theta)-\nabla_y g(x,\theta)\right)\right\| \\
			\le & ~ \frac{1}{\gamma} \|y^\prime - y\| + \left\|\nabla_y g(x^\prime, \theta^*_\gamma(x, y)) - \nabla_y g(x, \theta^*_\gamma(x, y))\right\| \\
			\le & ~ \frac{1}{\gamma} \|y^\prime - y\| + L_g \|x^\prime - x\|.
		\end{aligned}
	\end{equation}
	
	Combining \eqref{eq:theta_lipschitz_1} with \eqref{eq:theta_lipschitz_3} yields
	\begin{equation*}
		\|\theta^*_\gamma(x^\prime, y^\prime) - \theta^*_\gamma(x, y)\| \le \frac{1}{\gamma \rho_h}\|y^\prime - y\| + \frac{L_g}{\rho_h}\|x^\prime - x\|.
	\end{equation*}
	Each block norm on the right is bounded by $\|(x^\prime,y^\prime)-(x,y)\|$. Thus taking $L_\theta := \frac{1}{\gamma \rho_h} + \frac{L_g}{\rho_h}$ completes the proof.
\end{proof}

\begin{lemma}
	\label{lemma:v_grad_lipschitz}
	Suppose that $\gamma \in (0, \frac{1}{2\rho_y})$. Then, for all $(x^\prime, y^\prime), (x, y) \in \mathbb{R}^{d_x} \times \mathbb{R}^{d_y}$, we obtain:
	\begin{equation}
		\|\nabla \upsilon_\gamma(x^\prime, y^\prime) - \nabla \upsilon_\gamma(x, y)\| 
		\le L_\upsilon \|(x^\prime, y^\prime) - (x, y)\|,
	\end{equation}
	where $L_\upsilon := L_g\sqrt{1 + L^2_\theta} + \frac{1 + L_\theta}{\gamma}$.
\end{lemma}
\begin{proof}
	When $\gamma \in (0, \frac{1}{2\rho_y})$, $\theta^*_\gamma(x, y)$ is uniquely defined, and $\upsilon_\gamma(x, y)$ is differentiable with gradient given by \eqref{eq:envelope_grad}. Consequently,
	\begin{equation}
		\begin{aligned}
			& ~ \|\nabla \upsilon_\gamma(x^\prime, y^\prime) - \nabla \upsilon_\gamma(x, y)\| \\
			\le & ~ \|\nabla_x g(x^\prime, \theta^*_\gamma(x^\prime, y^\prime)) - \nabla_x g(x, \theta^*_\gamma(x, y))\| + \frac{1}{\gamma}\|y^\prime - y - (\theta^*_\gamma(x^\prime, y^\prime) - \theta^*_\gamma(x, y))\| \\
			\le & ~ L_g \|(x^\prime, \theta^*_\gamma(x^\prime, y^\prime)) - (x, \theta^*_\gamma(x, y))\| + \frac{1}{\gamma}\|y^\prime - y\| + \frac{1}{\gamma}\|\theta^*_\gamma(x^\prime, y^\prime) - \theta^*_\gamma(x, y)\| \\
			\le & ~ L_g \sqrt{\|x^\prime - x\|^2 + \|\theta^*_\gamma(x^\prime, y^\prime) - \theta^*_\gamma(x, y)\|^2} + \frac{1}{\gamma}\|y^\prime - y\| + \frac{L_\theta}{\gamma} \|(x^\prime, y^\prime) - (x, y)\| \\
			\le & ~ L_g \sqrt{\|x^\prime - x\|^2 + L_\theta^2 \|(x^\prime, y^\prime) - (x, y)\|^2} + \frac{1}{\gamma}\|y^\prime - y\| + \frac{L_\theta}{\gamma} \|(x^\prime, y^\prime) - (x, y)\| \\
			\le & ~ \left(L_g\sqrt{1 + L^2_\theta} + \frac{1 + L_\theta}{\gamma}\right) \|(x^\prime, y^\prime) - (x, y)\|.
		\end{aligned}
	\end{equation}
	Taking $L_\upsilon := L_g\sqrt{1 + L^2_\theta} + \frac{1 + L_\theta}{\gamma}$ completes the proof.
\end{proof}
\begin{lemma}
	\label{lemma:envelope_contraction}
	For any $\gamma \in (0, \frac{1}{2\rho_y})$ and $\alpha_{\theta, t} \in (0, \frac{2}{L_h + \rho_h}]$, 
	the iteration point $(x_t, y_t, \theta_t)$ of \dalg ~satisfies the following one-step contraction:
	\begin{equation}
		\label{eq:theta_convergence}
		\|\theta_{t + 1} - \theta^*_\gamma(x_t, y_t)\| 
		\le \sigma_t \|\theta_t - \theta^*_\gamma(x_t, y_t)\|,
	\end{equation}
	where $\sigma_t := 1 - \alpha_{\theta, t} \rho_h \in [0, 1)$.
\end{lemma}
\begin{proof}
	By the $\rho_h$-strong convexity and $L_h$-smoothness of $h(\theta; x, y)$, for any $(x, y) \in \mathbb{R}^{d_x} \times \mathbb{R}^{d_y}$, the following holds (see \citealp{nesterov2018lectures}, Theorem 2.1.12):
	\begin{equation}
		\label{eq:coercivity}
		\begin{aligned}
			& \langle \nabla h(\theta^\prime) - \nabla h(\theta),~ \theta^\prime - \theta \rangle \ge \frac{\|\nabla h(\theta^\prime) - \nabla h(\theta)\|^2}{L_h + \rho_h} + \frac{L_h \rho_h}{L_h + \rho_h} \|\theta^\prime - \theta\|^2.
		\end{aligned}
	\end{equation}
	
	Setting $\theta^\prime = \theta_t$ and $\theta = \theta^*_\gamma(x_t, y_t)$ in \eqref{eq:coercivity}, 
	and noting that $\nabla h(\theta^*_\gamma(x_t, y_t); x_t, y_t) = 0$, we obtain
	\begin{equation}
		\label{eq:coercivity_inst}
		\begin{aligned}
			& \langle \nabla h(\theta_t),~ \theta_t - \theta^*_\gamma(x_t, y_t) \rangle \ge \frac{\|\nabla h(\theta_t)\|^2}{L_h + \rho_h} + \frac{L_h \rho_h}{L_h + \rho_h} \|\theta_t - \theta^*_\gamma(x_t, y_t)\|^2.
		\end{aligned}
	\end{equation}
	
	Furthermore, from the update rule in \eqref{eq:theta_gd_update}, we have
	\begin{equation*}
		\begin{aligned}
			& ~ \|\theta_{t + 1} - \theta^*_\gamma(x_t, y_t)\|^2 \\
			= & ~ \|\theta_t - \alpha_{\theta, t} \nabla h(\theta_t) - \theta^*_\gamma(x_t, y_t)\|^2 \\
			= & ~ \|\theta_t - \theta^*_\gamma(x_t, y_t)\|^2 
			+ \alpha_{\theta, t}^2 \|\nabla h(\theta_t)\|^2 - 2\alpha_{\theta, t} \langle \theta_t - \theta^*_\gamma(x_t, y_t),~ \nabla h(\theta_t) \rangle.
		\end{aligned}
	\end{equation*}
	Applying \eqref{eq:coercivity_inst} yields
	\begin{equation*}
		\begin{aligned}
			& \|\theta_{t + 1} - \theta^*_\gamma(x_t, y_t)\|^2 \\
			\le & ~ \left(1 - \frac{2\alpha_{\theta, t} L_h \rho_h}{L_h + \rho_h}\right) 
			\|\theta_t - \theta^*_\gamma(x_t, y_t)\|^2 + \left(\alpha_{\theta, t}^2 - \frac{2\alpha_{\theta, t}}{L_h + \rho_h}\right) 
			\|\nabla h(\theta_t)\|^2 \\
			= & ~ \left(1 - \frac{2\alpha_{\theta, t} L_h \rho_h}{L_h + \rho_h} - \frac{2\alpha_{\theta, t}\rho^2_h}{L_h + \rho_h} + \alpha^2_{\theta, t} \rho^2_h\right) \|\theta_t - \theta^*_\gamma(x_t, y_t)\|^2 \\
			& ~ + \left(\alpha_{\theta, t}^2 - \frac{2\alpha_{\theta, t}}{L_h + \rho_h}\right)\left(\|\nabla h(\theta_t)\|^2 - \rho^2_h\left\|\theta_t - \theta^*_\gamma(x_t, y_t)\right\|^2\right) \\
			= & ~ \left(1 - \alpha_{\theta, t} \rho_h\right)^2 \|\theta_t - \theta^*_\gamma(x_t, y_t)\|^2 \\
			& ~ + \alpha_{\theta, t} \left(\alpha_{\theta, t} - \frac{2}{L_h + \rho_h} \right) 
			\left(\|\nabla h(\theta_t)\|^2 - \rho_h^2 \|\theta_t - \theta^*_\gamma(x_t, y_t)\|^2\right).
		\end{aligned}
	\end{equation*}
	
	Since $\|\nabla h(\theta_t)\|^2 \ge \rho_h^2 \|\theta_t - \theta^*_\gamma(x_t, y_t)\|^2$ 
	by the $\rho_h$-strong convexity of $h$, and $\alpha_{\theta, t} \le \frac{2}{L_h + \rho_h} \le \frac{1}{\rho_h}$,
	the last term in the above inequality is non-positive. Therefore,
	\begin{equation*}
		\|\theta_{t + 1} - \theta^*_\gamma(x_t, y_t)\|^2 
		\le (1 - \alpha_{\theta, t} \rho_h)^2 \|\theta_t - \theta^*_\gamma(x_t, y_t)\|^2.
	\end{equation*}
	Taking the square root on both sides completes the proof.
\end{proof}
\begin{lemma}
	\label{lemma:scaled_penalty_diff_ub_control}
	Let $I_t(x, y) \coloneqq \frac{1}{c_t}\mathcal{P}_{c_t}(x, y)$ and suppose $\gamma \in \left(0, \frac{1}{2\rho_y}\right)$. Then the iterates $(x_t, y_t, \theta_t)$ of \dalg~satisfy:
	\begin{equation}
		\begin{aligned}
			& I_t(x_{t+1}, y_{t+1}) - I_t(x_t, y_t) \\
			& \le \left(\frac{\alpha_{x, t} L_g^2}{2} + \frac{\alpha_{y, t}}{\gamma^2}\right) \|\theta_{t+1} - \theta^*_\gamma(x_t, y_t)\|^2 - \left(\frac{1}{2\alpha_{x, t}} - \frac{L_{I_t}}{2} - \frac{\alpha_{y, t} L_\theta^2}{\gamma^2}\right) \|x_{t+1} - x_t\|^2 \\
			& - \left(\frac{1}{2\alpha_{y, t}} - \frac{L_{I_t}}{2}\right) \|y_{t+1} - y_t\|^2.
		\end{aligned}
	\end{equation}
\end{lemma}
\begin{proof}
	From the smoothness properties stated in Assumption~\ref{assumption:ll}, together with Lemma~\ref{lemma:v_grad_lipschitz}, it readily follows that $I_t(x, y)$ is also gradient Lipschitz continuous. Specifically, we have
	\begin{equation}
		\label{eq:I_t_lipschitz}
		\|\nabla I_t(x_{t+1}, y_{t+1}) - \nabla I_t(x_t, y_t)\| 
		\le L_{I_t} \|(x_{t+1}, y_{t+1}) - (x_t, y_t)\|,
	\end{equation}
	with $L_{I_t} \coloneqq \frac{L_F}{c_t} + L_g + L_\upsilon$.
	
	To facilitate the subsequent upper-bound analysis, we decompose the difference of $I_t$ as follows:
	\begin{equation}
		\label{eq:potential_penalty_diff}
		\begin{aligned}
			& I_t(x_{t+1}, y_{t+1}) - I_t(x_t, y_t) \\
			= ~ & \bigl(I_t(x_{t+1}, y_{t+1}) - I_t(x_{t+1}, y_t)\bigr) 
			+ \bigl(I_t(x_{t+1}, y_t) - I_t(x_t, y_t)\bigr),
		\end{aligned}
	\end{equation}
	which separates the contributions of the $x$- and $y$-updates.
	
	We first examine the $x$-component. By the quadratic upper bound of $I_t$ with respect to $x$, we have
	\begin{equation}
		\label{eq:x_quadratic_ub}
		I_t(x_{t+1}, y_t) - I_t(x_t, y_t) 
		\le \langle \nabla_x I_t(x_t, y_t),~ x_{t+1} - x_t \rangle 
		+ \frac{L_{I_t}}{2} \|x_{t+1} - x_t\|^2.
	\end{equation}
	To control the inner product term using the gradient Lipschitz property, we introduce $d_{x, t}$ and rewrite it as
	\begin{equation}
		\label{eq:x_inner_product_control}
		\begin{aligned}
			& \langle \nabla_x I_t(x_t, y_t),~ x_{t+1} - x_t \rangle \\
			= ~ & \langle \nabla_x I_t(x_t, y_t) - d_{x, t},~ x_{t+1} - x_t \rangle 
			+ \langle d_{x, t},~ x_{t+1} - x_t \rangle \\
			= ~ & \langle \tilde{\nabla}_x \upsilon_\gamma(x_t, \theta_{t+1}) - \nabla_x \upsilon_\gamma(x_t, y_t),~ x_{t+1} - x_t \rangle 
			- \frac{1}{\alpha_{x, t}} \|x_{t+1} - x_t\|^2 \\
			\le ~ & \|\tilde{\nabla}_x \upsilon_\gamma(x_t, \theta_{t+1}) - \nabla_x \upsilon_\gamma(x_t, y_t)\| 
			\|x_{t+1} - x_t\| 
			- \frac{1}{\alpha_{x, t}} \|x_{t+1} - x_t\|^2 \\
			\le ~ & \frac{\alpha_{x, t}}{2} \|\tilde{\nabla}_x \upsilon_\gamma(x_t, \theta_{t+1}) - \nabla_x \upsilon_\gamma(x_t, y_t)\|^2 
			- \frac{1}{2\alpha_{x, t}} \|x_{t+1} - x_t\|^2 \\
			= ~ & \frac{\alpha_{x, t}}{2} \|\nabla_x g(x_t, \theta_{t+1}) - \nabla_x g(x_t, \theta^*_\gamma(x_t, y_t))\|^2 
			- \frac{1}{2\alpha_{x, t}} \|x_{t+1} - x_t\|^2 \\
			\le ~ & \frac{\alpha_{x, t} L_g^2}{2} \|\theta_{t+1} - \theta^*_\gamma(x_t, y_t)\|^2 
			- \frac{1}{2\alpha_{x, t}} \|x_{t+1} - x_t\|^2,
		\end{aligned}
	\end{equation}
	where the second inequality follows from the Cauchy--Schwarz inequality and the inequality $2ab \le \alpha a^2 + \frac{1}{\alpha}b^2 (\alpha > 0)$.
	
	Combining \eqref{eq:x_quadratic_ub} with \eqref{eq:x_inner_product_control} yields
	\begin{equation}
		\label{eq:x_quad_ub_control}
		\begin{aligned}
			I_t(x_{t+1}, y_t) - I_t(x_t, y_t) 
			\le ~ & \frac{\alpha_{x, t} L_g^2}{2} \|\theta_{t+1} - \theta^*_\gamma(x_t, y_t)\|^2 + \left(\frac{L_{I_t}}{2} - \frac{1}{2\alpha_{x, t}}\right) \|x_{t+1} - x_t\|^2.
		\end{aligned}
	\end{equation}
	
	We now turn to the $y$-component. Following a similar argument, we obtain
	\begin{equation}
		\label{eq:y_quadratic_ub}
		I_t(x_{t+1}, y_{t+1}) - I_t(x_{t+1}, y_t) 
		\le \langle \nabla_y I_t(x_{t+1}, y_t),~ y_{t+1} - y_t \rangle 
		+ \frac{L_{I_t}}{2} \|y_{t+1} - y_t\|^2.
	\end{equation}
	Introducing $d_{y, t}$ and proceeding analogously gives
	\begin{equation}
		\label{eq:y_inner_product_control}
		\begin{aligned}
			& \langle \nabla_y I_t(x_{t+1}, y_t),  y_{t+1} - y_t \rangle \\
			& =    \langle \nabla_y I_t(x_{t+1}, y_t) - d_{y, t},  y_{t+1} - y_t \rangle + \langle d_{y, t},  y_{t+1} - y_t \rangle \\
			& =    \langle \tilde{\nabla}_y \upsilon_\gamma(x_{t+1}, \theta_{t+1}) - \nabla_y \upsilon_\gamma(x_{t+1}, y_t),  y_{t+1} - y_t \rangle - \frac{1}{\alpha_{y, t}} \|y_{t+1} - y_t\|^2 \\
			& \le    \|\tilde{\nabla}_y \upsilon_\gamma(x_{t+1}, \theta_{t+1}) - \nabla_y \upsilon_\gamma(x_{t+1}, y_t)\| \|y_{t+1} - y_t\| - \frac{1}{\alpha_{y, t}} \|y_{t+1} - y_t\|^2 \\
			& \le    \frac{\alpha_{y, t}}{2} \|\tilde{\nabla}_y \upsilon_\gamma(x_{t+1}, \theta_{t+1}) - \nabla_y \upsilon_\gamma(x_{t+1}, y_t)\|^2 - \frac{1}{2\alpha_{y, t}} \|y_{t+1} - y_t\|^2 \\
			& =    \frac{\alpha_{y, t}}{2\gamma^2} \|\theta_{t+1} - \theta^*_\gamma(x_{t+1}, y_t)\|^2 - \frac{1}{2\alpha_{y, t}} \|y_{t+1} - y_t\|^2 \\
			& =    \frac{\alpha_{y, t}}{2\gamma^2} \left\|\bigl(\theta_{t+1} - \theta^*_\gamma(x_t, y_t)\bigr) - \bigl(\theta^*_\gamma(x_{t+1}, y_t) - \theta^*_\gamma(x_t, y_t)\bigr)\right\|^2 - \frac{1}{2\alpha_{y, t}} \|y_{t+1} - y_t\|^2 \\
			& \le \frac{\alpha_{y, t}}{2\gamma^2} \left(\begin{aligned}
				& \|\theta_{t+1} - \theta^*_\gamma(x_t, y_t)\|^2
				+ \|\theta^*_\gamma(x_{t+1}, y_t) - \theta^*_\gamma(x_t, y_t)\|^2 \\
				& + 2 \|\theta_{t+1} - \theta^*_\gamma(x_t, y_t)\|
				\|\theta^*_\gamma(x_{t+1}, y_t) - \theta^*_\gamma(x_t, y_t)\|
			\end{aligned}\right) \\
			& - \frac{1}{2\alpha_{y, t}} \|y_{t+1} - y_t\|^2 \\
			& \le    \frac{\alpha_{y, t}}{\gamma^2} \left(\|\theta_{t+1} - \theta^*_\gamma(x_t, y_t)\|^2 + \|\theta^*_\gamma(x_{t+1}, y_t) - \theta^*_\gamma(x_t, y_t)\|^2\right) - \frac{1}{2\alpha_{y, t}} \|y_{t+1} - y_t\|^2 \\
			& \le    \frac{\alpha_{y, t}}{\gamma^2} \|\theta_{t+1} - \theta^*_\gamma(x_t, y_t)\|^2 + \frac{\alpha_{y, t} L_\theta^2}{\gamma^2} \|x_{t+1} - x_t\|^2 - \frac{1}{2\alpha_{y, t}} \|y_{t+1} - y_t\|^2,
		\end{aligned}
	\end{equation}
	where the last inequality invokes Lemma~\ref{lemma:theta_lipschitz}.
	
	Combining \eqref{eq:y_quadratic_ub} with \eqref{eq:y_inner_product_control} yields
	\begin{equation}
		\label{eq:y_quad_ub_control}
		\begin{aligned}
			I_t(x_{t+1}, y_{t+1}) - I_t(x_{t+1}, y_t) 
			\le ~ & \frac{\alpha_{y, t}}{\gamma^2} \|\theta_{t+1} - \theta^*_\gamma(x_t, y_t)\|^2 + \frac{\alpha_{y, t} L_\theta^2}{\gamma^2} \|x_{t+1} - x_t\|^2 \\
			& + \left(\frac{L_{I_t}}{2} - \frac{1}{2\alpha_{y, t}}\right) \|y_{t+1} - y_t\|^2.
		\end{aligned}
	\end{equation}
	
	Finally, combining \eqref{eq:x_quad_ub_control} and \eqref{eq:y_quad_ub_control} via \eqref{eq:potential_penalty_diff} gives the overall upper bound:
	\begin{equation}
		\label{eq:potential_full_descent}
		\begin{aligned}
			& I_t(x_{t+1}, y_{t+1}) - I_t(x_t, y_t) \\
			& \le \left(\frac{\alpha_{x, t} L_g^2}{2} + \frac{\alpha_{y, t}}{\gamma^2}\right) \|\theta_{t+1} - \theta^*_\gamma(x_t, y_t)\|^2 - \left(\frac{1}{2\alpha_{x, t}} - \frac{L_{I_t}}{2} - \frac{\alpha_{y, t} L_\theta^2}{\gamma^2}\right) \|x_{t+1} - x_t\|^2 \\
			& - \left(\frac{1}{2\alpha_{y, t}} - \frac{L_{I_t}}{2}\right) \|y_{t+1} - y_t\|^2.
		\end{aligned}
	\end{equation}
	This completes the proof.
\end{proof}

\begin{lemma}
	\label{lemma:potential_diff_ub}
	Let $V_t \coloneqq \frac{1}{c_t}\mathcal{P}_{c_t}(x_t, y_t) + \left(L_g^2 + \frac{1}{\gamma^2}\right)\left\|\theta_t - \theta^*_\gamma(x_t, y_t)\right\|^2$. Suppose $\gamma\in(0,1/(2\rho_y))$, $\alpha_{\theta,t}\in[\underline\alpha_\theta,2/(L_h+\rho_h))$ with $\underline\alpha_\theta>0$, and $c_t$ is nondecreasing with $c_0>0$. Define the fixed constants
	\begin{equation}
		\label{eq:d_descent_constants}
		\begin{aligned}
			C_2&:=\frac{L_{I_0}}2+
			\left(1+\frac{2}{\underline\alpha_\theta\rho_h}\right)(L_g^2+\gamma^{-2})L_\theta^2,\\
			C_1&:=C_2+\frac{\underline\alpha_\theta\rho_h L_\theta^2}{2\gamma^2}.
		\end{aligned}
	\end{equation}
	Choose fixed $0<\underline\alpha_x\le u_x$ and $0<\underline\alpha_y\le u_y$, and require $\alpha_{x,t}\in[\underline\alpha_x,u_x]$ and $\alpha_{y,t}\in[\underline\alpha_y,u_y]$, where
	\begin{equation}
		\label{eq:alpha_x_y_ub}
		u_x:=\min\left\{\frac1{4C_1},\underline\alpha_\theta\rho_h\right\},\qquad
		u_y:=\min\left\{\frac1{4C_2},\frac{\underline\alpha_\theta\rho_h}{2}\right\}.
	\end{equation}
	Then the iterates of \dalg~satisfy
	\begin{equation}
		\begin{aligned}
			& V_{t + 1} - V_t \le -\alpha^2_{\theta, t}\rho_h^2\left(L^2_g + \frac{1}{\gamma^2}\right)\left\|\theta_t - \theta^*_\gamma(x_t, y_t)\right\|^2 - \frac{1}{4\alpha_{x, t}}\left\|x_{t + 1} - x_t\right\|^2 \\
			& -\frac{1}{4\alpha_{y, t}}\left\|y_{t + 1} - y_t\right\|^2,
		\end{aligned}
	\end{equation}
\end{lemma}
\begin{proof}
	We now examine the difference of the potential function.
	We begin with the following upper bound:
	\begin{equation}
		\label{eq:potential_preprocess_ub}
		\begin{aligned}
			& V_{t+1} - V_t \\
			= ~ & \frac{1}{c_{t+1}}\mathcal{P}_{c_{t+1}}(x_{t+1}, y_{t+1}) 
			- \frac{1}{c_t}\mathcal{P}_{c_t}(x_t, y_t) \\
			& + \left(L_g^2 + \frac{1}{\gamma^2}\right)
			\left( \|\theta_{t+1} - \theta^*_\gamma(x_{t+1}, y_{t+1})\|^2 
			- \|\theta_t - \theta^*_\gamma(x_t, y_t)\|^2 \right) \\
			\le ~ & I_t(x_{t+1}, y_{t+1}) - I_t(x_t, y_t) \\
			& + \left(L_g^2 + \frac{1}{\gamma^2}\right)
			\left( \|\theta_{t+1} - \theta^*_\gamma(x_{t+1}, y_{t+1})\|^2 
			- \|\theta_t - \theta^*_\gamma(x_t, y_t)\|^2 \right).
		\end{aligned}
	\end{equation}
	
	Applying Lemma~\ref{lemma:scaled_penalty_diff_ub_control} yields
	\begin{equation}
		\label{eq:potential_diff_ub_1}
		\begin{aligned}
			& V_{t+1} - V_t \\
			\le ~ & \left(\frac{\alpha_{x,t}}{2} L_g^2 + \frac{\alpha_{y,t}}{\gamma^2}\right)
			\|\theta_{t+1} - \theta^*_\gamma(x_t, y_t)\|^2 + \left(L_g^2 + \frac{1}{\gamma^2}\right)
			\|\theta_{t+1} - \theta^*_\gamma(x_{t+1}, y_{t+1})\|^2 \\
			& - \left(L_g^2 + \frac{1}{\gamma^2}\right)
			\|\theta_t - \theta^*_\gamma(x_t, y_t)\|^2 - \left(\frac{1}{2\alpha_{x,t}} - \frac{L_{I_t}}{2} 
			- \frac{\alpha_{y,t} L_\theta^2}{\gamma^2}\right) \|x_{t+1} - x_t\|^2 \\
			& - \left(\frac{1}{2\alpha_{y,t}} - \frac{L_{I_t}}{2}\right) \|y_{t+1} - y_t\|^2.
		\end{aligned}
	\end{equation}
	
	We now isolate the terms involving $\theta_{t+1}$ and handle the coefficients of $L_g^2$ and $\gamma^{-2}$ separately. 
	For the former, we have
	\begin{equation}
		\label{eq:potential_theta_diff_ub}
		\begin{aligned}
			& \frac{\alpha_{x,t}}{2} \|\theta_{t+1} - \theta^*_\gamma(x_t, y_t)\|^2 
			+ \|\theta_{t+1} - \theta^*_\gamma(x_{t+1}, y_{t+1})\|^2 \\
			= ~ & \frac{\alpha_{x,t}}{2} \|\theta_{t+1} - \theta^*_\gamma(x_t, y_t)\|^2 + \left\| \bigl(\theta_{t+1} - \theta^*_\gamma(x_t, y_t)\bigr) 
			+ \bigl(\theta^*_\gamma(x_t, y_t) - \theta^*_\gamma(x_{t+1}, y_{t+1})\bigr) \right\|^2 \\
			= ~ & \left(\frac{\alpha_{x,t}}{2} + 1\right) 
			\|\theta_{t+1} - \theta^*_\gamma(x_t, y_t)\|^2 + \|\theta^*_\gamma(x_t, y_t) - \theta^*_\gamma(x_{t+1}, y_{t+1})\|^2 \\
			& + 2 \langle \theta_{t+1} - \theta^*_\gamma(x_t, y_t),~
			\theta^*_\gamma(x_t, y_t) - \theta^*_\gamma(x_{t+1}, y_{t+1}) \rangle \\
			\le ~ & \left(\frac{\alpha_{x,t}}{2} + 1 + \delta_t\right)
			\|\theta_{t+1} - \theta^*_\gamma(x_t, y_t)\|^2 + \left(1 + \frac{1}{\delta_t}\right)
			\|\theta^*_\gamma(x_t, y_t) - \theta^*_\gamma(x_{t+1}, y_{t+1})\|^2 \\
			\le ~ & \left(\frac{\alpha_{x,t}}{2} + 1 + \delta_t\right) \sigma_t^2
			\|\theta_t - \theta^*_\gamma(x_t, y_t)\|^2 + \left(1 + \frac{1}{\delta_t}\right) L_\theta^2
			\left( \|x_{t+1} - x_t\|^2 + \|y_{t+1} - y_t\|^2 \right),
		\end{aligned}
	\end{equation}
	where the second inequality follows from Young's inequality, and the last inequality invokes the contraction property of $\theta_t$ and the Lipschitz continuity of $\theta^*_\gamma$.
	
	Since $\rho_h\le L_h$, the restriction $\alpha_{\theta,t}<2/(L_h+\rho_h)$ gives $a_t:=\alpha_{\theta,t}\rho_h\in(0,1)$. Thus $\sigma_t=1-a_t$. Choose $\delta_t=a_t/2$. The assumed step-size bound gives $\alpha_{x,t}\le u_x\le\underline\alpha_\theta\rho_h\le a_t$. Therefore,
	\begin{equation}
		\label{eq:potential_theta_coef_control}
		\left(\frac{\alpha_{x,t}}{2} + 1 + \delta_t\right) \sigma_t^2
		\le(1+a_t)(1-a_t)^2
		=(1-a_t)(1-a_t^2)
		\le1-a_t^2=1-\alpha_{\theta,t}^2\rho_h^2.
	\end{equation}
	The other Young coefficient is $1+\delta_t^{-1}=1+2/(\alpha_{\theta,t}\rho_h)$.
	
	Combining \eqref{eq:potential_theta_diff_ub} with \eqref{eq:potential_theta_coef_control} yields
	\begin{equation}
		\label{eq:potential_theta_with_alpha_x_ub_control}
		\begin{aligned}
			& \frac{\alpha_{x,t}}{2} \|\theta_{t+1} - \theta^*_\gamma(x_t, y_t)\|^2 
			+ \|\theta_{t+1} - \theta^*_\gamma(x_{t+1}, y_{t+1})\|^2 \\
			\le ~ & (1 - \alpha_{\theta,t}^2 \rho_h^2) \|\theta_t - \theta^*_\gamma(x_t, y_t)\|^2 + \left(1 + \frac{2}{\alpha_{\theta,t}\rho_h}\right) L_\theta^2
			\left( \|x_{t+1} - x_t\|^2 + \|y_{t+1} - y_t\|^2 \right).
		\end{aligned}
	\end{equation}
	
	For the second term, $\alpha_{y,t}\le u_y\le\underline\alpha_\theta\rho_h/2\le a_t/2$ gives $(\alpha_{y,t}+1+\delta_t)\sigma_t^2\le(1+a_t)(1-a_t)^2\le1-a_t^2$. Hence the same argument yields
	\begin{equation}
		\label{eq:potential_theta_with_alpha_y_ub_control}
		\begin{aligned}
			& \alpha_{y,t} \|\theta_{t+1} - \theta^*_\gamma(x_t, y_t)\|^2 
			+ \|\theta_{t+1} - \theta^*_\gamma(x_{t+1}, y_{t+1})\|^2 \\
			\le ~ & (1 - \alpha_{\theta,t}^2 \rho_h^2) \|\theta_t - \theta^*_\gamma(x_t, y_t)\|^2 + \left(1 + \frac{2}{\alpha_{\theta,t}\rho_h}\right) L_\theta^2
			\left( \|x_{t+1} - x_t\|^2 + \|y_{t+1} - y_t\|^2 \right).
		\end{aligned}
	\end{equation}
	
	Combining \eqref{eq:potential_diff_ub_1}, \eqref{eq:potential_theta_diff_ub}, 
	\eqref{eq:potential_theta_with_alpha_x_ub_control}, and \eqref{eq:potential_theta_with_alpha_y_ub_control}, we arrive at
	\begin{equation}
		\label{eq:potential_diff_ub_2}
		\begin{aligned}
			& V_{t+1} - V_t \\
			\le ~ & - \alpha_{\theta,t}^2 \rho_h^2 \left(L_g^2 + \frac{1}{\gamma^2}\right)
			\|\theta_t - \theta^*_\gamma(x_t, y_t)\|^2 \\
			& - \left( \frac{1}{2\alpha_{x,t}} - \frac{L_{I_t}}{2} 
			- \frac{\alpha_{y,t} L_\theta^2}{\gamma^2} 
			- \left(1 + \frac{2}{\alpha_{\theta,t}\rho_h}\right) 
			\left(L_g^2 + \frac{1}{\gamma^2}\right) L_\theta^2 \right)
			\|x_{t+1} - x_t\|^2 \\
			& - \left( \frac{1}{2\alpha_{y,t}} - \frac{L_{I_t}}{2} 
			- \left(1 + \frac{2}{\alpha_{\theta,t}\rho_h}\right) 
			\left(L_g^2 + \frac{1}{\gamma^2}\right) L_\theta^2 \right)
			\|y_{t+1} - y_t\|^2.
		\end{aligned}
	\end{equation}
	
	Since $c_t$ is nondecreasing, $L_{I_t}\le L_{I_0}$. Also, $\alpha_{\theta,t}\ge\underline\alpha_\theta$ and $\alpha_{y,t}\le u_y\le\underline\alpha_\theta\rho_h/2$. The constants in \eqref{eq:d_descent_constants} therefore give the following uniform coefficient bounds:
	\begin{equation}
		\label{eq:potential_x_coef_control_1}
		\begin{aligned}
			& - \left( \frac{1}{2\alpha_{x,t}} - \frac{L_{I_t}}{2} 
			- \frac{\alpha_{y,t} L_\theta^2}{\gamma^2} 
			- \left(1 + \frac{2}{\alpha_{\theta,t}\rho_h}\right) 
			\left(L_g^2 + \frac{1}{\gamma^2}\right) L_\theta^2 \right) \\
			\le ~ & -\left(\frac{1}{2\alpha_{x,t}}-C_1\right).
		\end{aligned}
	\end{equation}
	
	Similarly, for the coefficient of $\|y_{t+1} - y_t\|^2$ in \eqref{eq:potential_diff_ub_2}, 
	we obtain
	\begin{equation}
		\label{eq:potential_y_coef_control_1}
		\begin{aligned}
			& - \left( \frac{1}{2\alpha_{y,t}} - \frac{L_{I_t}}{2} 
			- \left(1 + \frac{2}{\alpha_{\theta,t}\rho_h}\right) 
			\left(L_g^2 + \frac{1}{\gamma^2}\right) L_\theta^2 \right) \\
			\le ~ & -\left(\frac{1}{2\alpha_{y,t}}-C_2\right).
		\end{aligned}
	\end{equation}
	
	By \eqref{eq:alpha_x_y_ub}, $C_1\le1/(4\alpha_{x,t})$ and $C_2\le1/(4\alpha_{y,t})$. Combining these inequalities with \eqref{eq:potential_x_coef_control_1} and \eqref{eq:potential_y_coef_control_1}, we obtain
	\begin{equation}
		\label{eq:potential_x_coef_control_2}
		- \left( \frac{1}{2\alpha_{x,t}} - \frac{L_{I_t}}{2} 
		- \frac{\alpha_{y,t} L_\theta^2}{\gamma^2} 
		- \left(1 + \frac{2}{\alpha_{\theta,t}\rho_h}\right) 
		\left(L_g^2 + \frac{1}{\gamma^2}\right) L_\theta^2 \right)
		\le -\frac{1}{4\alpha_{x,t}},
	\end{equation}
	and
	\begin{equation}
		\label{eq:potential_y_coef_control_2}
		- \left( \frac{1}{2\alpha_{y,t}} - \frac{L_{I_t}}{2} 
		- \left(1 + \frac{2}{\alpha_{\theta,t}\rho_h}\right) 
		\left(L_g^2 + \frac{1}{\gamma^2}\right) L_\theta^2 \right)
		\le -\frac{1}{4\alpha_{y,t}}.
	\end{equation}
	
	Finally, under the step-size conditions $\alpha_{x,t} \le u_x$ and $\alpha_{y,t} \le u_y$, 
	combining \eqref{eq:potential_diff_ub_2}, \eqref{eq:potential_x_coef_control_2}, 
	and \eqref{eq:potential_y_coef_control_2} with $\underline{\alpha}_\theta \le \alpha_{\theta,t}$ yields
	\begin{equation}
		\label{eq:potential_diff_ub_final}
		\begin{aligned}
			V_{t + 1} - V_t \le & -\alpha^2_{\theta, t}\rho_h^2\left(L^2_g + \frac{1}{\gamma^2}\right)\left\|\theta_t - \theta^*_\gamma(x_t, y_t)\right\|^2 - \frac{1}{4\alpha_{x, t}}\left\|x_{t + 1} - x_t\right\|^2 \\
			& -\frac{1}{4\alpha_{y, t}}\left\|y_{t + 1} - y_t\right\|^2.
		\end{aligned}
	\end{equation}
	This completes the proof.
\end{proof}
\begin{lemma}
	\label{lemma:d_penalty_stationarity}
	Under the assumptions and conditions of Lemma~\ref{lemma:potential_diff_ub}, let $c_t=c_0(1+t)^p$ with $c_0>0$ and $p\in[0,1/2]$, and define $D_t:=\|\nabla\mathcal P_{c_t}(z_{t+1})\|$. For $T\ge2$ and the index $R_T$ in \eqref{eq:deterministic_output_selection},
	\[
	\min_{0\le t<T}D_t\le D_{R_T}
	=\mathcal O\left(T^{p-\frac12}\right).
	\]
	If, in addition, $p>0$ and $\mathcal P_{c_t}(z_t)\le\overline{\mathcal P}$ uniformly in $t$, then
	\[
	V_t=\mathcal O\left((1+t)^{-p}\right),\qquad
	D_{R_T}=\mathcal O\left(T^{\frac{p-1}{2}}\right).
	\]
\end{lemma}
\begin{proof}
	By \eqref{eq:x_gd_update}, $d_{x,t}=-\alpha_{x,t}^{-1}(x_{t+1}-x_t)$. The uniform smoothness of $I_t$, the $L_g$-Lipschitz continuity of $\nabla g$, and Lemma~\ref{lemma:envelope_contraction} give
	\begin{equation}
		\label{eq:penalty_grad_x_ub}
		\begin{aligned}
			& \|\nabla_x \mathcal{P}_{c_t}(x_{t+1}, y_{t+1})\| \\
			\le ~ & \|\nabla_x \mathcal{P}_{c_t}(x_{t+1}, y_{t+1}) 
			- \nabla_x \mathcal{P}_{c_t}(x_t, y_t)\|
			+ \|\nabla_x \mathcal{P}_{c_t}(x_t, y_t) - c_t d_{x,t}\| 
			+ \frac{c_t}{\alpha_{x,t}} \|x_{t+1} - x_t\| \\
			\le ~ & c_t L_{I_0} \|(x_{t+1}, y_{t+1}) - (x_t, y_t)\| 
			+ c_t L_g \|\theta_{t+1} - \theta^*_\gamma(x_t, y_t)\|
			+ \frac{c_t}{\alpha_{x,t}} \|x_{t+1} - x_t\| \\
			\le ~ & c_t L_{I_0} \|(x_{t+1}, y_{t+1}) - (x_t, y_t)\| 
			+ c_t L_g \|\theta_t - \theta^*_\gamma(x_t, y_t)\|
			+ \frac{c_t}{\alpha_{x,t}} \|x_{t+1} - x_t\|,
		\end{aligned}
	\end{equation}
	
	Similarly, for the $y$-component $\nabla_y \mathcal{P}_{c_t}(x_{t+1}, y_{t+1})$, we obtain
	\begin{equation}
		\label{eq:penalty_grad_y_ub}
		\begin{aligned}
			& \|\nabla_y \mathcal{P}_{c_t}(x_{t+1}, y_{t+1})\| \\
			\le ~ & c_t L_{I_0} \|y_{t+1} - y_t\| 
			+ \frac{c_t}{\gamma} \|\theta_{t+1} - \theta^*_\gamma(x_{t+1}, y_t)\| 
			+ \frac{c_t}{\alpha_{y,t}} \|y_{t+1} - y_t\| \\
			\le ~ & c_t \left(L_{I_0} + \frac{1}{\alpha_{y,t}}\right) \|y_{t+1} - y_t\| 
			+ \frac{c_t}{\gamma} \|\theta_{t+1} - \theta^*_\gamma(x_t, y_t)\|
			+ \frac{c_t}{\gamma} \|\theta^*_\gamma(x_t, y_t) - \theta^*_\gamma(x_{t+1}, y_t)\| \\
			\le ~ & c_t \left(L_{I_0} + \frac{1}{\alpha_{y,t}}\right) \|y_{t+1} - y_t\| 
			+ \frac{c_t}{\gamma} \|\theta_t - \theta^*_\gamma(x_t, y_t)\|
			+ \frac{c_t L_\theta}{\gamma} \|x_{t+1} - x_t\|,
		\end{aligned}
	\end{equation}
	where the last inequality uses Lemma~\ref{lemma:theta_lipschitz} and Lemma~\ref{lemma:envelope_contraction}.
	
	Combining \eqref{eq:penalty_grad_x_ub} and \eqref{eq:penalty_grad_y_ub} yields
	\begin{equation}
		\label{eq:penalty_grad_ub}
		\begin{aligned}
			D_t = & ~ \|\nabla \mathcal{P}_{c_t}(x_{t+1}, y_{t+1})\| \\
			\le ~ & \|\nabla_x \mathcal{P}_{c_t}(x_{t+1}, y_{t+1})\| 
			+ \|\nabla_y \mathcal{P}_{c_t}(x_{t+1}, y_{t+1})\| \\
			\le ~ & c_t L_{I_0} \sqrt{\|x_{t+1} - x_t\|^2 + \|y_{t+1} - y_t\|^2} + c_t L_h \|\theta_t - \theta^*_\gamma(x_t, y_t)\| \\
			& + c_t \left(\frac{1}{\alpha_{x,t}} + \frac{L_\theta}{\gamma}\right) \|x_{t+1} - x_t\| + c_t \left(L_{I_0} + \frac{1}{\alpha_{y,t}}\right) \|y_{t+1} - y_t\|.
		\end{aligned}
	\end{equation}
	
	We now analyze the output rule \eqref{eq:deterministic_output_selection}. The $\theta$-direction is the gradient of the $\rho_h$-strongly convex and $L_h$-smooth function $h(\theta;x_t,y_t)$. Hence
	\[
	\rho_h\|\theta_t-\theta^*_\gamma(z_t)\|
	\le\|d_{\theta,t}\|
	\le L_h\|\theta_t-\theta^*_\gamma(z_t)\|.
	\]
	Using $x_{t+1}-x_t=-\alpha_{x,t}d_{x,t}$ and $y_{t+1}-y_t=-\alpha_{y,t}d_{y,t}$, define the finite constants
	\begin{equation}
		\label{eq:deterministic_proxy_constants}
		\begin{aligned}
			C_\Psi&:=\max\left\{\frac{L_h^2}{\underline\alpha_\theta^2\rho_h^2(L_g^2+\gamma^{-2})},\frac4{\underline\alpha_x},\frac4{\underline\alpha_y}\right\},\\
			C_{\mathrm{proxy}}&:=\left(\frac{L_h}{\rho_h}\right)^2
			+\left[1+u_x\left(L_{I_0}+\frac{L_\theta}{\gamma}\right)\right]^2
			+\left[1+2u_yL_{I_0}\right]^2.
		\end{aligned}
	\end{equation}
	The descent bound \eqref{eq:potential_diff_ub_final} gives $\Psi_t\le C_\Psi(V_t-V_{t+1})$. In addition, \eqref{eq:penalty_grad_ub} and the same update identities imply
	\[
	\begin{aligned}
		\frac{D_t}{c_t}\le{}&\frac{L_h}{\rho_h}\|d_{\theta,t}\|
		+\left[1+u_x\left(L_{I_0}+\frac{L_\theta}{\gamma}\right)\right]\|d_{x,t}\|\\
		&+\left[1+2u_yL_{I_0}\right]\|d_{y,t}\|.
	\end{aligned}
	\]
	Cauchy--Schwarz therefore yields
	\begin{equation}
		\label{eq:deterministic_proxy_control}
		\frac{D_t^2}{c_t^2}\le C_{\mathrm{proxy}}\Psi_t.
	\end{equation}
	Summing $\Psi_t\le C_\Psi(V_t-V_{t+1})$ over $\mathcal J_T$ and using the output rule, $c_{R_T}\le c_{T-1}=c_0T^p$, and $V_T\ge0$, we obtain
	\begin{equation}
		\label{eq:deterministic_window_residual_bound}
		\begin{aligned}
			\Psi_{R_T}&\le\frac1{n_T}\sum_{t\in\mathcal J_T}\Psi_t
			\le\frac{C_\Psi(V_{\ell_T}-V_T)}{n_T},\\
			D_{R_T}&\le c_{T-1}\sqrt{\frac{C_{\mathrm{proxy}}C_\Psi V_{\ell_T}}{n_T}}.
		\end{aligned}
	\end{equation}
	Since $V_{\ell_T}\le V_0$ and $n_T=\lfloor T/2\rfloor\ge T/3$ for $T\ge2$, this gives
	\begin{equation}
		\label{eq:deterministic_selected_residual_bound}
		D_{R_T}=\mathcal O\left(T^{p-\frac12}\right).
	\end{equation}
	Since $\min_{0\le t<T}D_t\le D_{R_T}$, this proves the first bound.
	
	For the additional conclusion, suppose $p>0$ and $\mathcal P_{c_t}(z_t)\le\overline{\mathcal P}$ uniformly in $t$. Set
	\[
	\delta:=(\underline\alpha_\theta\rho_h)^2\in(0,1),\qquad
	M:=\frac{\overline{\mathcal P}}{c_0},\qquad
	e_t:=\theta_t-\theta_\gamma^*(z_t).
	\]
	The restriction on $\alpha_{\theta,t}$ ensures $\underline\alpha_\theta\rho_h<1$. Since $I_t(z_t)\le M(1+t)^{-p}$ and $V_t=I_t(z_t)+(L_g^2+\gamma^{-2})\|e_t\|^2$, the descent bound \eqref{eq:potential_diff_ub_final} gives
	\begin{equation}
		\label{eq:deterministic_potential_decay_recursion}
		\begin{aligned}
			V_{t+1}
			&\le V_t-\delta(L_g^2+\gamma^{-2})\|e_t\|^2\\
			&=(1-\delta)V_t+\delta I_t(z_t)\\
			&\le(1-\delta)V_t+\delta M(1+t)^{-p}.
		\end{aligned}
	\end{equation}
	Choose an integer $N\ge0$ such that $p/(1+N)\le\delta/2$, and choose $K\ge\max\{2M,V_0(1+N)^p\}$. Monotonicity of $V_t$ gives $V_t\le K(1+t)^{-p}$ for $0\le t\le N$. If this bound holds at $t\ge N$, then
	\[
	\begin{aligned}
		V_{t+1}
		&\le K(1+t)^{-p}(1-\delta/2)\\
		&\le K(1+t)^{-p}\left(1-\frac p{1+t}\right)
		\le K(2+t)^{-p},
	\end{aligned}
	\]
	where the last inequality uses $(1+s)^{-p}\ge1-ps$ for $s\ge0$. Induction therefore proves
	\begin{equation}
		\label{eq:deterministic_potential_decay_bound}
		V_t=\mathcal O\left((1+t)^{-p}\right).
	\end{equation}
	Substituting \eqref{eq:deterministic_potential_decay_bound} into \eqref{eq:deterministic_window_residual_bound} and using $n_T\ge T/3$ gives
	\begin{equation}
		\label{eq:deterministic_selected_residual_decay_bound}
		D_{R_T}=\mathcal O\left(T^{\frac{p-1}{2}}\right).
	\end{equation}
	All constants above are independent of $T$. This proves the additional conclusion.
\end{proof}
\subsection{Proof of Theorem~\ref{theorem:dalg}}
\label{lemma:envelope_constraint_convergence}
\begin{proof}
	By Assumption~\ref{assumption:ul} and the monotonicity of $V_t$ in Lemma~\ref{lemma:potential_diff_ub}, for every $t$,
	\begin{equation}
		\label{eq:constraint_violation_trivial}
		0\le q_\gamma(z_t)\le I_t(z_t)\le V_t\le V_0.
	\end{equation}
	
	The bound \eqref{eq:constraint_violation_trivial} holds at every iterate, so $\widehat\varepsilon_T=q_\gamma(z_{R_T+1})\le V_0$. Moreover, if $\mathcal P_{c_t}(z_t)$ admits a uniform upper bound $\overline{\mathcal P}$, then for every $t$,
	\[
	0\le q_\gamma(z_t)
	\le\frac{\mathcal P_{c_t}(z_t)}{c_t}
	\le\frac{\overline{\mathcal P}}{c_t}.
	\]
	Since $R_T+1\ge\ell_T+1$, monotonicity of $c_t=c_0(1+t)^p$, $p\in[0,1/2]$, gives the gap bound at the selected output:
	\[
	\widehat\varepsilon_T\le\frac{\overline{\mathcal P}}{c_{R_T+1}}
	\le\frac{\overline{\mathcal P}}{c_{\ell_T+1}}.
	\]
	Thus
	\begin{equation}
		\label{eq:constraint_violation_deterministic}
		\widehat\varepsilon_T=\mathcal O(T^{-p}).
	\end{equation}
	Applying \eqref{eq:pareto_residual_penalty_bound} at $t=R_T$ gives
	\[
	H(\widehat z_T;\widehat\varepsilon_T)
	\le H_{w,\widehat p_T}(\widehat z_T;\widehat\varepsilon_T)
	=\frac{D_{R_T}}{s_w(\widehat z_T)}
	\le\frac{D_{R_T}}{\underline w}.
	\]
	For $p=0$, \eqref{eq:deterministic_selected_residual_bound} in Lemma~\ref{lemma:d_penalty_stationarity} yields $H_{w,\widehat p_T}(\widehat z_T;\widehat\varepsilon_T)=\mathcal O(T^{-1/2})$, while $\widehat\varepsilon_T\le V_0$. For $p=1/3$, the uniform penalty-value bound permits the strengthened estimate \eqref{eq:deterministic_selected_residual_decay_bound}, so
	\[
	H_{w,\widehat p_T}(\widehat z_T;\widehat\varepsilon_T)
	=\mathcal O\left(T^{\frac{p-1}{2}}\right)
	=\mathcal O(T^{-1/3}).
	\]
	Combining this with \eqref{eq:constraint_violation_deterministic} gives $\widehat\varepsilon_T=\mathcal O(T^{-1/3})$ at the same returned point. This proves Theorem~\ref{theorem:dalg}.
\end{proof}

\section{Theoretical Proof Details for \salg}
Throughout this section, let $\mathcal{F}_t$ denote the $\sigma$-algebra generated by all
randomness drawn before iteration $t$, so that $x_t, y_t, \theta_t$ are $\mathcal{F}_t$-measurable.

\subsection{Auxiliary Lemmas}
\label{subsec:component_oracle_proof}
Since the three directions in Algorithm~\ref{alg:salg} are formed sequentially within one
iteration, each direction is conditioned on the history available once its full-gradient
counterpart becomes measurable:
\begin{equation*}
	\mathcal{F}_{\theta,t} \coloneqq \mathcal{F}_{t}, \qquad
	\mathcal{F}_{x,t} \coloneqq \sigma\bigl(\mathcal{F}_{t},\, \varrho_{\theta,t}\bigr), \qquad
	\mathcal{F}_{y,t} \coloneqq \sigma\bigl(\mathcal{F}_{t},\, \varrho_{\theta,t},\,
	\varrho_{x,t},\, \xi_{x,t}\bigr),
\end{equation*}
so that $d_{\theta,t}\in\mathcal{F}_{\theta,t}$, $d_{x,t}\in\mathcal{F}_{x,t}$ (as it depends on
$\theta_{t+1}$), and $d_{y,t}\in\mathcal{F}_{y,t}$ (as it depends on $x_{t+1}$). The mini-batches
used to form $\hat{d}_{x,t}$ and $\hat{d}_{y,t}$ are drawn independently of $\mathcal{F}_{x,t}$
and $\mathcal{F}_{y,t}$, respectively. In particular, the upper-level samples $\xi_{y,t,i}$
used in the $y$-update are drawn independently of the samples $\xi_{x,t,i}$ used in the
$x$-update. All expectations in the stochastic lemmas are full expectations obtained by
applying the tower property to these conditional expectations.
The given initial points $x_0,y_0$ are fixed vectors. We use the initialization in Algorithm~\ref{alg:salg}: $\theta_0=y_0$ and $m_{\theta,0}=m_{x,0}=m_{y,0}=0$.
\begin{lemma}[Bias and Variance from Component Oracles]
	\label{lemma:component_batch_bias_variance}
	Under Assumptions~\ref{assumption:var_ub} and~\ref{assumption:ll}, choose $G_i<\infty$ with $\|\nabla f_i(z)\|\le G_i$, and define
	\begin{align*}
		v_f^2&:=\sum_iw_i^2\nu_i^2,\qquad
		v_\nabla^2:=\sum_iw_i^2\sigma_i^2,\qquad
		G_w^2:=\sum_iw_i^2G_i^2,\\
		C_{\mathrm{bias}}&:=\frac{\mu^2G_wv_f^2+\mu v_fv_\nabla}{2},\qquad
		V_{\mathrm{STCH}}:=\frac{\mu^2G_w^2v_f^2}{2}+2v_\nabla^2.
	\end{align*}
	For $j\in\{\theta,x,y\}$, let
	\[
	b_{j,t}:=\mathbb E[\hat d_{j,t}\mid\mathcal F_{j,t}]-d_{j,t},\qquad
	\mathcal V_{j,t}:=\mathbb E[\|\hat d_{j,t}-d_{j,t}-b_{j,t}\|^2\mid\mathcal F_{j,t}].
	\]
	Then, almost surely for all $t\ge0$,
	\begin{equation}
		\label{eq:component_batch_bias_variance}
		\begin{aligned}
			b_{\theta,t}&=0,\qquad
			\|b_{j,t}\|\le\frac{C_{\mathrm{bias}}}{c_tB_t}\quad(j\in\{x,y\}),\\
			\mathcal V_{\theta,t}&\le\frac{\sigma_g^2}{B_t},\qquad
			\mathcal V_{x,t}\le\frac{4\sigma_g^2+V_{\mathrm{STCH}}/c_t^2}{B_t},\qquad
			\mathcal V_{y,t}\le\frac{\sigma_g^2+V_{\mathrm{STCH}}/c_t^2}{B_t}.
		\end{aligned}
	\end{equation}
\end{lemma}
\begin{proof}
	Condition on the relevant history and fix its measurable query point $z$. Let $\tau(a)=\operatorname{softmax}(a)$. Its Jacobian is
	\[
	D\tau(a)=\operatorname{diag}(\tau(a))-\tau(a)\tau(a)^\top,
	\qquad \|D\tau(a)\|_{\mathrm{op}}\le\frac12.
	\]
	The matrix is positive semidefinite, and the absolute row sum in row $i$ is $2\tau_i(1-\tau_i)\le1/2$, which proves the norm bound. For a direction $h\in\mathbb R^m$, write $\bar h=\sum_i\tau_i h_i$ and $\operatorname{Var}_\tau(h)=\sum_i\tau_i(h_i-\bar h)^2$. Direct differentiation gives
	\[
	D^2\tau_i(a)[h,h]
	=\tau_i\bigl((h_i-\bar h)^2-\operatorname{Var}_\tau(h)\bigr),
	\]
	so
	\[
	\|D^2\tau(a)[h,h]\|_1
	\le2\operatorname{Var}_\tau(h)
	=2h^\top D\tau(a)h\le\|h\|^2.
	\]
	Consequently, globally in $a$ and $h$,
	\begin{equation}
		\label{eq:softmax_batch_remainder}
		\|\tau(a+h)-\tau(a)\|\le\frac12\|h\|,\qquad
		\|\tau(a+h)-\tau(a)-D\tau(a)h\|\le\frac12\|h\|^2.
	\end{equation}
	
	Fix $j\in\{x,y\}$, and define $a_i=\mu w_i(f_i(z)-z_i)$ and the matrix $U$ whose $i$-th column is $w_i\nabla_jf_i(z)$. With bars denoting batch averages, set
	\[
	\Delta a_i=\mu w_i(\bar f_i(z)-f_i(z)),\qquad
	(\Delta U)_i=w_i(\overline{\nabla_jf_i}(z)-\nabla_jf_i(z)).
	\]
	All expectations in the remainder of this argument are conditional on the fixed history. Independence within each batch and the primitive moment bounds imply
	\[
	\mathbb E[\Delta a]=0,\quad \mathbb E[\Delta U]=0,\quad
	\mathbb E\|\Delta a\|^2\le\frac{\mu^2v_f^2}{B_t},\quad
	\mathbb E\|\Delta U\|_{\mathrm F}^2\le\frac{v_\nabla^2}{B_t},\quad
	\|U\|_{\mathrm F}\le G_w.
	\]
	The STCH gradient estimator is $\hat G_j=(U+\Delta U)\tau(a+\Delta a)$, whereas its full-gradient counterpart is $G_j=U\tau(a)$. The zero-mean linear terms give
	\[
	\mathbb E[\hat G_j-G_j]
	=U\mathbb E[\tau(a+\Delta a)-\tau(a)-D\tau(a)\Delta a]
	+\mathbb E[\Delta U(\tau(a+\Delta a)-\tau(a))].
	\]
	Using \eqref{eq:softmax_batch_remainder} and Cauchy--Schwarz, without requiring independence between objective values and gradients, yields
	\begin{equation}
		\label{eq:stch_batch_bias_mse}
		\begin{aligned}
			\|\mathbb E[\hat G_j-G_j]\|
			&\le\frac{G_w}{2}\mathbb E\|\Delta a\|^2
			+\frac12\sqrt{\mathbb E\|\Delta U\|_{\mathrm F}^2\mathbb E\|\Delta a\|^2}
			\le\frac{C_{\mathrm{bias}}}{B_t},\\
			\mathbb E\|\hat G_j-G_j\|^2
			&\le\frac{G_w^2}{2}\mathbb E\|\Delta a\|^2
			+2\mathbb E\|\Delta U\|_{\mathrm F}^2
			\le\frac{V_{\mathrm{STCH}}}{B_t}.
		\end{aligned}
	\end{equation}
	For the second inequality, use $\hat G_j-G_j=U(\tau(a+\Delta a)-\tau(a))+\Delta U\tau(a+\Delta a)$ and $\|\tau(a+\Delta a)\|\le1$. The centered conditional variance is no larger than this mean-square error.
	
	The STCH contribution to each outer direction is divided by $c_t$, giving the stated bias bounds. The $g$-gradient estimates are conditionally unbiased. In the $x$-direction, the two $g$ gradients use the same batch, but their centered difference has conditional second moment at most $4\sigma_g^2/B_t$ by $\|u-v\|^2\le2\|u\|^2+2\|v\|^2$. Its batch is independent of the upper-level batches, so the centered variances add. In the $y$- and $\theta$-directions, there is one stochastic $g$ gradient, with variance at most $\sigma_g^2/B_t$; the regularization terms are measurable with respect to the corresponding histories. This proves all bounds in \eqref{eq:component_batch_bias_variance}.
\end{proof}
For the subsequent stationarity analysis, write
\[
\sigma^2:=4\sigma_g^2+V_{\mathrm{STCH}}/c_0^2,\qquad
\nu_t:=\frac{\sigma^2}{B_t}.
\]
By \eqref{eq:component_batch_bias_variance} and $c_t\ge c_0$, $\nu_t$ bounds the conditional variance of every direction, and $\|b_{j,t}\|\le C_{\mathrm{bias}}/(c_tB_t)$ for $j\in\{x,y\}$. In particular, the initial biases have finite second moments.

\subsection{Stationarity Analysis}
\label{subsec:momentum_proof}

\begin{lemma}[Stationarity Bound for \salg]
	\label{lemma:momentum_stationarity}
	Under Assumptions~\ref{assumption:ul},~\ref{assumption:ll}, and~\ref{assumption:var_ub}, suppose $\gamma\in(0,1/(2\rho_y))$ and let $\beta\in(0,1)$ be independent of $T$. Write $\rho_h=\gamma^{-1}-\rho_y$, $L_h=L_g+\gamma^{-1}$, and $I_t(z)=\mathcal P_{c_t}(z)/c_t$. Let
	\[
	\alpha_{\theta,t}=a,\qquad
	\alpha_{x,t}=\kappa_xa^2,\qquad
	\alpha_{y,t}=\kappa_ya^2,\qquad
	c_t=c_0(1+t)^p,
	\]
	for every $t\ge0$, where $a>0$, $0\le p\le1/2$, $c_0>0$, and $0<\kappa_x\le\kappa_y\le\rho_h/8$ are fixed independently of $T$. There is a positive threshold $a_{\max}(\beta)$, constructed in \eqref{eq:momentum_step_rule} below, depending only on $\beta$ and problem constants. If $0<a\le a_{\max}(\beta)$, then there are finite constants $C_{\beta,0},C_{\beta,1},C_{\beta,2},C_{\beta,3}$, independent of $T$ and of the deterministic batch sizes $B_t\ge1$, such that the following bounds hold. Define
	\[
	\mathcal S_B(T):=\sum_{t=0}^{T-1}\frac{a^2}{c_t^2B_t^2}.
	\]
	The quantity
	\begin{equation}
		\label{eq:momentum_noise_budget}
		\mathcal A_T^\beta:=C_{\beta,0}
		+C_{\beta,1}\sigma^2(a^3+a^4)\sum_{t=0}^{T-1}\frac1{B_t}
		+C_{\beta,2}C_{\mathrm{bias}}^2\mathcal S_B(T)
	\end{equation}
	satisfies, with $D_t=\|\nabla\mathcal P_{c_t}(z_{t+1})\|$,
	\begin{equation}
		\label{eq:momentum_weighted_bound}
		\sum_{t=0}^{T-1}\frac{a^2}{c_t^2}\mathbb E[D_t^2]
		\le\mathcal A_T^\beta,\qquad
		\max_{0\le n\le T}\mathbb E[I_n(z_n)]\le C_{\beta,3}\mathcal A_T^\beta.
	\end{equation}
	Consequently, the independent output selection \eqref{eq:stochastic_output_selection} gives
	\begin{equation}
		\label{eq:momentum_selected_bound}
		\mathbb E[D_{R_T}^2]
		\le\frac{\mathcal A_T^\beta}{\sum_{t\in\mathcal J_T}a^2/c_t^2}.
	\end{equation}
	For the potential-decay conclusion, let $p>0$ and $B_t\ge b_0(1+t)^r$ for $0\le t<T$, where $r\ge p$ and $b_0>0$ are fixed. If the potential $W_t$ defined in \eqref{eq:momentum_potential} satisfies $\mathbb E[W_t]\le C_W(1+t)^{-p}$ for $1\le t\le T$, with $C_W$ independent of $T$ and of the prescribed batch sizes, then
	\begin{equation}
		\label{eq:momentum_decay_selected_bound}
		\mathbb E[D_{R_T}^2]
		=\mathcal O\left(T^{p-1}+T^{2p-r}+T^{-2r}\right).
	\end{equation}
	In particular, $(p,r)=(1/3,4/3)$ gives $\mathbb E[D_{R_T}^2]=\mathcal O(T^{-2/3})$ and $\mathbb E[D_{R_T}]=\mathcal O(T^{-1/3})$.
\end{lemma}
\begin{proof}
	Write $s=1-\beta$, $r_\beta=\beta/s$, $h=a^2$, $m_{z,t}=(m_{x,t},m_{y,t})$, and $\mathsf K=\operatorname{diag}(\kappa_x\operatorname{Id}_{d_x},\kappa_y\operatorname{Id}_{d_y})$. For $t\ge0$, introduce variables used only in the proof:
	\begin{equation}
		\label{eq:momentum_shifted_variables}
		\bar\theta_t=\theta_t-r_\beta a m_{\theta,t},\qquad
		\bar z_t=z_t-r_\beta h\mathsf K m_{z,t}.
	\end{equation}
	For any block $j\in\{\theta,x,y\}$, the step size $\eta_{j,t}$ is constant in $t$, so the momentum update gives the exact identity
	\begin{equation}
		\label{eq:momentum_shift_identity}
		\bar z_{j,t+1}=\bar z_{j,t}-\eta_{j,t}\hat d_{j,t}.
	\end{equation}
	Here $z_{\theta,t}=\theta_t$, $z_{x,t}=x_t$, and $z_{y,t}=y_t$, with $\eta_{\theta,t}=a$, $\eta_{x,t}=\kappa_xa^2$, and $\eta_{y,t}=\kappa_ya^2$ for every $t$. The zero initial momenta imply $\bar\theta_0=\theta_0$ and $\bar z_0=z_0$.
	
	Set $\theta_t^*=\theta^*_\gamma(z_t)$. Throughout this proof let
	\[
	\begin{aligned}
		Q_t&=\mathbb E\|\bar\theta_t-\theta_t^*\|^2,
		&R_t&=\mathbb E\|\theta_t-\theta_t^*\|^2,\\
		M_{\theta,t}&=\mathbb E\|m_{\theta,t}\|^2,
		&M_{z,t}&=\mathbb E\|m_{z,t}\|^2,\\
		G_t&=\mathbb E\|\nabla I_t(z_t)\|^2,
		&F_t&=\mathbb E[I_t(\bar z_t)],\\
		\mathfrak b_t^2&=\mathbb E\|b_{x,t}\|^2+\mathbb E\|b_{y,t}\|^2.
	\end{aligned}
	\]
	The bias bound gives $\mathfrak b_t^2\le2C_{\mathrm{bias}}^2/(c_t^2B_t^2)$. All quantities are finite at each finite time. Indeed, the population directions have at most linear growth, the primitive centered noises have uniformly bounded second moments, and the bias is uniformly bounded for $c_t\ge c_0$ and $B_t\ge1$. Induction through the three sequential updates bounds each finite prefix uniformly over all choices $B_t\ge1$.
	
	We first derive four scalar recursions. Constants used to dominate their coefficients depend on $s,r_\beta,L_h,L_\theta,L_{I_0},\rho_h,\kappa_x,\kappa_y$, but not on $a,t,T$, or the batch sizes. A common such constant $C\ge1$ can be chosen so that, when $a\le\min\{1,\rho_h/(4L_h^2),(4C)^{-1/4}\}$,
	\begin{equation}
		\label{eq:momentum_four_recursions}
		\begin{aligned}
			F_{t+1}-F_t\le{}&-d a^2G_t+C a^4G_t+C a^2R_t+C a^4M_{\theta,t}\\
			&+C a^6M_{z,t}+C a^2\mathfrak b_t^2+C a^4\nu_t,\\
			a(Q_{t+1}-Q_t)\le{}&-u a^2R_t+C a^4(R_t+G_t)\\
			&+C a^4M_{\theta,t}+C a^4M_{z,t}
			+C a^2\mathfrak b_t^2+C a^3\nu_t,\\
			M_{\theta,t+1}-M_{\theta,t}\le{}&-sM_{\theta,t}+sL_h^2R_t+s^2\nu_t,\\
			M_{z,t+1}-M_{z,t}\le{}&-\frac{s}{2}M_{z,t}
			+Cs(G_t+R_t+a^2M_{\theta,t}+\mathfrak b_t^2)+Cs^2\nu_t,
		\end{aligned}
	\end{equation}
	where $d=\kappa_x/4$ and $u=\rho_h/8$.
	
	For the inner update, put $r_t=\theta_t-\theta_t^*$ and $v_{\theta,t}=\hat d_{\theta,t}-d_{\theta,t}$. Before moving the target from $\theta_t^*$ to $\theta_{t+1}^*$, \eqref{eq:momentum_shift_identity} yields
	\[
	\bar\theta_{t+1}-\theta_t^*
	=r_t-a d_{\theta,t}-r_\beta a m_{\theta,t}-a v_{\theta,t}.
	\]
	Since $r_t,d_{\theta,t},m_{\theta,t}$ are measurable before the inner batch and $\mathbb E[v_{\theta,t}\mid\mathcal F_{\theta,t}]=0$, expansion gives
	\begin{equation}
		\label{eq:momentum_inner_expansion}
		\begin{aligned}
			&\mathbb E\|\bar\theta_{t+1}-\theta_t^*\|^2-Q_t\\
			={}&-2a\mathbb E\langle r_t,d_{\theta,t}\rangle
			+a^2\mathbb E\|d_{\theta,t}\|^2
			+2r_\beta a^2\mathbb E\langle d_{\theta,t},m_{\theta,t}\rangle
			+a^2\mathbb E\|v_{\theta,t}\|^2.
		\end{aligned}
	\end{equation}
	Strong convexity gives $\langle r_t,d_{\theta,t}\rangle\ge\rho_h\|r_t\|^2$, and smoothness gives $\|d_{\theta,t}\|\le L_h\|r_t\|$. The step-size bound controls the quadratic direction term, while Young's inequality bounds the cross term by $\rho_h aR_t/4+C a^3M_{\theta,t}$. Hence
	\[
	\mathbb E\|\bar\theta_{t+1}-\theta_t^*\|^2-Q_t
	\le-\rho_h aR_t+C a^3M_{\theta,t}+a^2\nu_t.
	\]
	To move the target, use Young's inequality with parameter $\epsilon a$, where $\epsilon>0$ depends only on $\rho_h,L_h$, and Lemma~\ref{lemma:theta_lipschitz}:
	\[
	Q_{t+1}\le(1+\epsilon a)\mathbb E\|\bar\theta_{t+1}-\theta_t^*\|^2
	+(1+1/(\epsilon a))L_\theta^2\mathbb E\|z_{t+1}-z_t\|^2.
	\]
	The expression preceding \eqref{eq:momentum_inner_expansion} bounds its first squared norm by $3(1+a^2L_h^2)R_t+3r_\beta^2a^2M_{\theta,t}+a^2\nu_t$. Taking $\epsilon$ small enough retains a negative multiple of $\rho_h aR_t$. Also,
	\[
	\mathbb E\|z_{t+1}-z_t\|^2
	\le\kappa_y^2a^4M_{z,t+1}.
	\]
	Multiplication by $a$ gives
	\begin{equation}
		\label{eq:momentum_inner_precollect}
		a(Q_{t+1}-Q_t)
		\le-\frac{\rho_h}{2}a^2R_t
		+C a^4M_{\theta,t}+C a^4M_{z,t+1}+C a^3\nu_t.
	\end{equation}
	
	For the outer updates, stack $d_{z,t}=(d_{x,t},d_{y,t})$ and set $g_t=\nabla I_t(z_t)$. The population Lipschitz bounds and the sequential query points imply
	\begin{equation}
		\label{eq:momentum_direction_error}
		\mathbb E\|d_{z,t}-g_t\|^2
		\le C\bigl(R_t+a^2M_{\theta,t+1}+a^4M_{z,t+1}\bigr).
	\end{equation}
	For the $x$ block the error is bounded by $L_g\|\theta_{t+1}-\theta_t^*\|$. For the $y$ block, add the change of $\nabla_yI_t$ from $(x_t,y_t)$ to $(x_{t+1},y_t)$ and the change of $\theta_\gamma^*$ at these two points. Both changes are bounded by a fixed multiple of $\|x_{t+1}-x_t\|$. Finally, $\|\theta_{t+1}-\theta_t^*\|\le\|r_t\|+a\|m_{\theta,t+1}\|$, which establishes \eqref{eq:momentum_direction_error}.
	
	Write $\hat d_{j,t}=d_{j,t}+b_{j,t}+v_{j,t}$ for $j=x,y$. Conditional on its own query history, $v_{j,t}$ has zero mean and variance at most $\nu_t$. Applying the smoothness inequality to $I_t$ at $\bar z_t$ and using \eqref{eq:momentum_shift_identity}, the linear terms in $v_{x,t},v_{y,t}$ have expectation zero: $\nabla I_t(\bar z_t)$ is measurable before either batch. This remains true for the $y$ query formed after the random $x$ update. Further,
	\[
	\|\nabla I_t(\bar z_t)-g_t\|
	\le L_{I_0}r_\beta h\kappa_y\|m_{z,t}\|.
	\]
	Young's inequality applied to the direction error and bias, together with $\langle g_t,\mathsf K g_t\rangle\ge\kappa_x\|g_t\|^2$, gives
	\begin{equation}
		\label{eq:momentum_outer_precollect}
		\begin{aligned}
			F_{t+1}-F_t\le{}&-d a^2G_t+C a^4G_t
			+C a^2\mathbb E\|d_{z,t}-g_t\|^2\\
			&+C a^6M_{z,t}+C a^2\mathfrak b_t^2+C a^4\nu_t.
		\end{aligned}
	\end{equation}
	Here the smoothness remainder contributes $O(h^2\nu_t)$, and $I_{t+1}(z)\le I_t(z)$ pointwise removes any positive penalty-change term. This inequality uses the nonnegativity of $F_w^{\mathrm{(STCH)}}-\underline F$ and monotonicity of $c_t$.
	
	For each momentum block, conditional cancellation of $v_{j,t}$ and convexity of the squared norm give
	\[
	\mathbb E\|m_{j,t+1}\|^2
	\le\beta\mathbb E\|m_{j,t}\|^2
	+s\mathbb E\|d_{j,t}+b_{j,t}\|^2+s^2\nu_t.
	\]
	This proves the third recursion in \eqref{eq:momentum_four_recursions}, since $b_{\theta,t}=0$. Summing the two outer block bounds and using \eqref{eq:momentum_direction_error} gives
	\[
	M_{z,t+1}\le\beta M_{z,t}
	+Cs(G_t+R_t+a^2M_{\theta,t+1}+a^4M_{z,t+1}+\mathfrak b_t^2)
	+2s^2\nu_t.
	\]
	If $Ca^4\le1/4$, move the term in $M_{z,t+1}$ to the left. The old coefficient is at most
	\[
	\frac{1-s}{1-s/4}\le1-s/2.
	\]
	Substitute the inner momentum bound and enlarge the common constant $C$ to obtain the fourth recursion in \eqref{eq:momentum_four_recursions}. Substituting these momentum bounds in \eqref{eq:momentum_inner_precollect} and \eqref{eq:momentum_outer_precollect} proves the first two recursions. The resulting extra terms are bounded by $Ca^4(R_t+G_t)$, $Ca^4M_{\theta,t}$, and $Ca^6M_{z,t}$, or by their displayed bias and variance terms. Choosing $C$ to dominate all these finitely many coefficients and the coefficient needed for the absorption above establishes \eqref{eq:momentum_four_recursions} with one common constant.
	
	We now give a sufficient step-size rule and collect the recursions. With this common $C$, choose
	\[
	A=1+\frac{C+2}{u},\qquad
	\lambda_\theta=\frac{4C(A+1)}s,\qquad
	\lambda_z=\frac{8C(A+1)}s,\qquad
	Z=C(A+1)(1+8C).
	\]
	A sufficient rule is
	\begin{equation}
		\label{eq:momentum_step_rule}
		\begin{gathered}
			0<a\le a_{\max}(\beta),\\
			a_{\max}(\beta):=\min\left\{1,\frac{\rho_h}{4L_h^2},(4C)^{-1/4},(8C)^{-1/3},
			\frac1{8C(A+1)L_h^2},\sqrt{\frac{\min\{d/2,1/2\}}Z}\right\}.
		\end{gathered}
	\end{equation}
	Define the nonnegative potential
	\begin{equation}
		\label{eq:momentum_potential}
		W_t=I_t(\bar z_t)+Aa\|\bar\theta_t-\theta_t^*\|^2
		+\lambda_\theta a^3\|m_{\theta,t}\|^2
		+\lambda_z a^4\|m_{z,t}\|^2.
	\end{equation}
	Summing \eqref{eq:momentum_four_recursions} with the constant weights $1,A,\lambda_\theta a^3,\lambda_z a^4$, the coefficients of $G_t$ and $R_t$ are at most
	\[
	-d a^2+Z a^4\le-\frac d2a^2,\qquad
	-2a^2+4C(A+1)L_h^2a^3+Z a^4\le-a^2.
	\]
	The coefficients of $M_{\theta,t}$ and $M_{z,t}$ satisfy
	\[
	\begin{aligned}
		C(A+1)a^4+8C^2(A+1)a^6-4C(A+1)a^3
		&\le-C(A+1)a^3,\\
		Ca^6+ACa^4-4C(A+1)a^4
		&\le-C(3A+2)a^4.
	\end{aligned}
	\]
	The remaining terms give, for a finite constant $C_\beta$ and every $t\ge0$,
	\begin{equation}
		\label{eq:momentum_potential_descent}
		\begin{aligned}
			\mathbb E[W_{t+1}]-\mathbb E[W_t]\le{}&-\frac d2a^2G_t-a^2R_t\\
			&-C(A+1)a^3M_{\theta,t}-C(3A+2)a^4M_{z,t}\\
			&+C_\beta\bigl((a^3+a^4)\nu_t+a^2\mathfrak b_t^2\bigr).
		\end{aligned}
	\end{equation}
	
	To pass from $G_t$ to the gradient at the actual returned iterate, uniform smoothness gives
	\[
	\frac{D_t^2}{c_t^2}=\|\nabla I_t(z_{t+1})\|^2
	\le2\|\nabla I_t(z_t)\|^2+2L_{I_0}^2\kappa_y^2a^4\|m_{z,t+1}\|^2.
	\]
	The fourth recursion in \eqref{eq:momentum_four_recursions}, $a\le1$, and multiplication by $a^2$ thus imply
	\begin{equation}
		\label{eq:momentum_residual_descent_control}
		\frac{a^2}{c_t^2}\mathbb E[D_t^2]
		\le C_\beta\bigl(a^2G_t+a^2R_t+a^3M_{\theta,t}+a^4M_{z,t}
		+a^2\mathfrak b_t^2+a^4\nu_t\bigr).
	\end{equation}
	Summing \eqref{eq:momentum_potential_descent} from $t=0$ and using $W_t\ge0$ controls these terms. The initial value $W_0$ is finite and independent of the batches. Enlarging $C_{\beta,0}$ to include this value and using the primitive bias bound proves the first inequality in \eqref{eq:momentum_weighted_bound}.
	
	Nonnegativity and $L_{I_0}$-smoothness give $\|\nabla I_t(v)\|^2\le2L_{I_0}I_t(v)$. The smoothness inequality and Young's inequality therefore imply, for every $u,v$,
	\begin{equation}
		\label{eq:nonnegative_smooth_value_comparison}
		I_t(u)\le2I_t(v)+L_{I_0}\|u-v\|^2.
	\end{equation}
	For the actual penalty value, apply \eqref{eq:nonnegative_smooth_value_comparison} with $u=z_t$ and $v=\bar z_t$:
	\[
	I_t(z_t)\le2I_t(\bar z_t)+C_\beta a^4\|m_{z,t}\|^2
	\le C_\beta W_t.
	\]
	The same potential recursion bounds $\mathbb E[W_n]$ by the initial potential and the accumulated noise for every $n\le T$. This proves the second inequality in \eqref{eq:momentum_weighted_bound}. The output probabilities then give \eqref{eq:momentum_selected_bound}.
	
	For the window conclusions, the output window in \eqref{eq:output_tail_window} satisfies, for every fixed $v>0$,
	\begin{equation}
		\label{eq:stochastic_window_power_sum}
		\sum_{t\in\mathcal J_T}(1+t)^{-v}\le2^{v-1}T^{1-v},\qquad T\ge2,
	\end{equation}
	since $|\mathcal J_T|\le T/2$ and $1+t\ge T/2$. Also, $1+t\le T$ and $n_T\ge T/3$ give
	\begin{equation}
		\label{eq:stochastic_tail_denominator}
		\sum_{t\in\mathcal J_T}\frac{a^2}{c_t^2}
		\ge\frac{a^2}{3c_0^2}T^{1-2p}.
	\end{equation}
	The same window comparison gives an upper bound of this order, so the denominator is $\Theta(T^{1-2p})$.
	
	For the potential-decay conclusion, the auxiliary terms in \eqref{eq:momentum_residual_descent_control} are bounded by a fixed multiple of the negative terms in \eqref{eq:momentum_potential_descent}. Summing over $\mathcal J_T$ and using $W_T\ge0$ yields
	\[
	\sum_{t\in\mathcal J_T}\frac{a^2}{c_t^2}\mathbb E[D_t^2]
	\le C_\beta\mathbb E[W_{\ell_T}]
	+C_\beta\sum_{t\in\mathcal J_T}
	\left((a^3+a^4)\frac{\sigma^2}{B_t}
	+\frac{C_{\mathrm{bias}}^2a^2}{c_t^2B_t^2}\right).
	\]
	The assumed potential decay and $B_t\ge b_0(1+t)^r$ give
	\[
	\sum_{t\in\mathcal J_T}\frac{a^2}{c_t^2}\mathbb E[D_t^2]
	=\mathcal O\left(T^{-p}+T^{1-r}+T^{1-2p-2r}\right).
	\]
	Division by the window denominator proves \eqref{eq:momentum_decay_selected_bound}. At $(p,r)=(1/3,4/3)$ its three terms are $T^{-2/3}$, $T^{-2/3}$, and $T^{-8/3}$, respectively. Taking square roots proves the final claim.
\end{proof}

\subsection{Proof of Theorem~\ref{theorem:salg}}
\begin{proof}
	Write $a=\alpha_{\theta,t}$, which is independent of $t$. For the fixed-penalty part, $p=0$ and $B_t=B\ge b_0T$ give
	\[
	\sum_{t=0}^{T-1}\frac1{B_t}=\frac{T}{B}\le\frac1{b_0},\qquad
	\mathcal S_B(T)=\frac{a^2T}{c_0^2B^2}
	\le\frac{a^2}{c_0^2b_0^2T}.
	\]
	Thus $\mathcal A_T^\beta=\mathcal O(1)$ uniformly in the prescribed horizon. Lemma~\ref{lemma:momentum_stationarity} and $q_\gamma(z_n)\le I_n(z_n)$ give $\mathbb E[\widehat\varepsilon_T]=\mathcal O(1)$. Since $\sum_{t\in\mathcal J_T}a^2/c_t^2=\Theta(T)$, \eqref{eq:momentum_selected_bound} gives
	\[
	\mathbb E[D_{R_T}^2]=\mathcal O(T^{-1}),\qquad
	\mathbb E[D_{R_T}]\le\sqrt{\mathbb E[D_{R_T}^2]}=\mathcal O(T^{-1/2}).
	\]
	
	For the growing-penalty part, set $(p,r)=(1/3,4/3)$. Both batch choices satisfy $B_t\ge b_0(1+t)^r$ for $0\le t<T$, since $1+t\le T$. The assumed penalty bound gives
	\[
	i_t:=\mathbb E[I_t(z_t)],\qquad
	0\le i_t\le\frac{\overline{\mathcal P}}{c_0}(1+t)^{-p},
	\qquad 1\le t\le T.
	\]
	Use the shifted variables, moments and constants in the proof of Lemma~\ref{lemma:momentum_stationarity}, and write $w_t:=\mathbb E[W_t]$. Apply \eqref{eq:nonnegative_smooth_value_comparison} with $u=\bar z_t$ and $v=z_t$. Also, $Q_t\le2R_t+2r_\beta^2a^2M_{\theta,t}$. Consequently,
	\begin{equation}
		\label{eq:momentum_potential_comparison}
		w_t\le2i_t+2AaR_t+K_\theta a^3M_{\theta,t}+K_z a^4M_{z,t},
	\end{equation}
	where
	\[
	K_\theta:=\lambda_\theta+2Ar_\beta^2,\qquad
	K_z:=\lambda_z+L_{I_0}r_\beta^2\kappa_y^2.
	\]
	Define the positive constants
	\[
	\lambda_\beta:=\min\left\{\frac1{2A},\frac{C(A+1)}{K_\theta a},
	\frac{C(3A+2)}{K_z a},\frac1{2a}\right\},\qquad
	\chi:=\lambda_\beta a\in(0,1/2].
	\]
	By \eqref{eq:momentum_potential_comparison}, the negative terms in \eqref{eq:momentum_potential_descent} dominate $\chi(w_t-2i_t)$. Its positive remainder satisfies
	\[
	\begin{aligned}
		n_t^\beta&:=C_\beta\left((a^3+a^4)\frac{\sigma^2}{B_t}
		+a^2\mathfrak b_t^2\right)\\
		&=\mathcal O\left((1+t)^{-r}+(1+t)^{-2p-2r}\right)
		=\mathcal O((1+t)^{-4/3}).
	\end{aligned}
	\]
	Hence, for $1\le t<T$,
	\begin{equation}
		\label{eq:stochastic_potential_decay_recursion}
		w_{t+1}\le(1-\chi)w_t+2\chi i_t+n_t^\beta
		\le(1-\chi)w_t+C_{\mathrm{dec}}(1+t)^{-p}.
	\end{equation}
	The discrete comparison argument used to obtain \eqref{eq:deterministic_potential_decay_bound} applies with contraction factor $\delta=\chi$ and forcing constant $\delta M=C_{\mathrm{dec}}$. The uniform finite-prefix moment bounds initialize the same induction. Consequently,
	\begin{equation}
		\label{eq:momentum_potential_decay_bound}
		\mathbb E[W_t]\le C_W(1+t)^{-1/3},\qquad 1\le t\le T,
	\end{equation}
	with a constant independent of the horizon and of the prescribed batches.
	
	Applying \eqref{eq:momentum_decay_selected_bound} with $(p,r)=(1/3,4/3)$ gives
	\[
	\mathbb E[D_{R_T}^2]=\mathcal O(T^{-2/3}),\qquad
	\mathbb E[D_{R_T}]\le\sqrt{\mathbb E[D_{R_T}^2]}=\mathcal O(T^{-1/3}).
	\]
	For either batch choice, $q_\gamma(z_n)\le\mathcal P_{c_n}(z_n)/c_n$. Independence of $R_T$ and monotonicity of $c_t$ therefore yield
	\[
	\mathbb E[\widehat\varepsilon_T]
	=\sum_{t\in\mathcal J_T}\pi_{T,t}\mathbb E[q_\gamma(z_{t+1})]
	\le\sum_{t\in\mathcal J_T}\pi_{T,t}\frac{\overline{\mathcal P}}{c_{t+1}}
	\le\frac{\overline{\mathcal P}}{c_{\ell_T+1}}
	=\mathcal O(T^{-1/3}).
	\]
	At the selected point, \eqref{eq:pareto_residual_penalty_bound} gives $H_{w,\widehat p_T}(\widehat z_T;\widehat\varepsilon_T)=D_{R_T}/s_w(\widehat z_T)$. Lemma~\ref{lemma:stch_residual_bound} therefore gives
	\[
	\mathbb E\left[H_{w,\widehat p_T}(\widehat z_T;\widehat\varepsilon_T)\right]
	\le\frac{\mathbb E[D_{R_T}]}{\underline w}.
	\]
	This proves the fixed-penalty residual rate $\mathcal O(T^{-1/2})$ and the growing-penalty joint rate $\mathcal O(T^{-1/3})$-$\mathcal O(T^{-1/3})$. The minimum Pareto residual $H(\widehat z_T;\widehat\varepsilon_T)$ is bounded above pointwise by the same STCH-based residual.
	
	For fixed $m$, each iteration uses $\mathcal O(B_t)$ primitive stochastic oracle calls. With $B=\lceil b_0T\rceil$, the fixed-penalty part uses $\mathcal O(T^2)$ calls and requires $T=\mathcal O(\epsilon^{-2})$ for the residual bound. With $B=\lceil b_0T^{4/3}\rceil$ or $B_t=\lceil b_0(1+t)^{4/3}\rceil$, the growing-penalty part uses $\mathcal O(T^{7/3})$ calls and requires $T=\mathcal O(\epsilon^{-3})$ for both expected errors. These are the stated $\mathcal O(\epsilon^{-4})$ and $\mathcal O(\epsilon^{-7})$ oracle complexities.
\end{proof}

\section{Experiment Details}
All experiments are conducted on a machine equipped with an Intel i5-13600K CPU, 32 GB RAM, and a single NVIDIA RTX 4090 GPU. The software environment is Python 3.12.11 with PyTorch 2.5.1 running on Windows 11.
We use a fixed random seed of 42 for Caltech-256 class clustering, episode sampling in meta-learning, data set splitting, and the NumPy, PyTorch, and CUDA random number generators.
\label{sec:exp_details}
\subsection{Multi-Domain Few-Shot Meta-Learning}
\label{subsec:meta_learning_details}
To verify the effectiveness of \dalg~in deterministic and nonconvex lower level scenarios, we conduct experiments on multi-domain few-shot meta-learning \citep{finn2017model, cheng2021few}.
Within each episode of meta-learning, the data are deterministic. The task of learning features across $m$ domains constitutes $m$ potentially conflicting objectives, which can be formulated mathematically as follows:
\begin{gather*}
	\min_x \left[\mathcal{L}_i(x, y^*_i; \mathcal{Q}_i)\right]^{m}_{i = 1}, \\
	\text{s.t.} \quad y^* \in {\arg\min}_y \frac{1}{m} \sum^{m}_{i = 1} \mathcal{L}_i(x, y_i; S_i),
\end{gather*}
where $x$ is the shared meta model parameter, $y_i$ is the specific parameter of domain $i$  , $\mathcal{Q}_i$ is the query set of domain $i$, $\mathcal{S}_i$ is the support set of domain $i$, $\mathcal{L}_i$ is the loss function of domain $i$.
The lower level aims to extract the features from all the domains; the upper level then needs to optimize the meta model using potentially conflicting feature information from the $m$ domains. Eventually we expect to obtain a meta model that can rapidly learn how to distinguish five classes never seen before. 

The sensitivity analysis is conducted on a subset of the FC-100 data set~\citep{oreshkin2018tadam}, and the comparison experiment is performed on the Caltech-256 data set~\citep{griffin2007caltech} resized to $32 \times 32$. For the sensitivity analysis, each domain contains 15 classes, which are split into training and test sets with a ratio of 3:2. For the comparison experiment, we group all classes of the Caltech-256 data set into four domains by semantic clustering using features extracted from a pre-trained ResNet-18. Each domain is then split into training, validation, and test sets with a ratio of 2:1:1.

Both experiments adopt the standard 5-way 5-shot protocol. We employ a 4-layer convolutional neural network (CNN) as the meta model, with each layer followed by batch normalization,
a $2\times2$ max-pooling operation, and ReLU activation. All convolutional layers have 64 channels. $\mathcal{L}$ is the Cross-Entropy loss.

Unless otherwise specified, the sensitivity analysis shares the following default hyperparameters: $T = 5000, c_t = (1 + t)^{\frac{1}{4}}, \gamma = 8.0, \mu = 4.0, \alpha_{x, t} = \alpha_{y, t} = \alpha_{\theta, t} = 0.01, w = [0.2, 0.4, 0.2, 0.2]^\top$. The final test accuracy is derived from the average of 192 random episodes after eight adaptation steps across four domains.

In the \textbf{sensitivity analysis}, we examine the test performance of the model under different choices of the smoothing parameter $\mu$, envelope regularity $\gamma$ and the preference vector $w$.
To investigate the effect of preference strength, we consider four settings using the first objective as the reference: a weakly preferred setting with 
$w = [0.31, 0.23, 0.23, 0.23]^\top$, a moderately preferred setting with $w = [0.4, 0.2, 0.2, 0.2]^\top$, a strongly preferred setting with $w = [0.52, 0.16, 0.16, 0.16]^\top$, and an extremely preferred setting with $w = [0.97, 0.01, 0.01, 0.01]^\top$.

\begin{figure}[tbp]
	\centering
	\includegraphics[width=\linewidth]{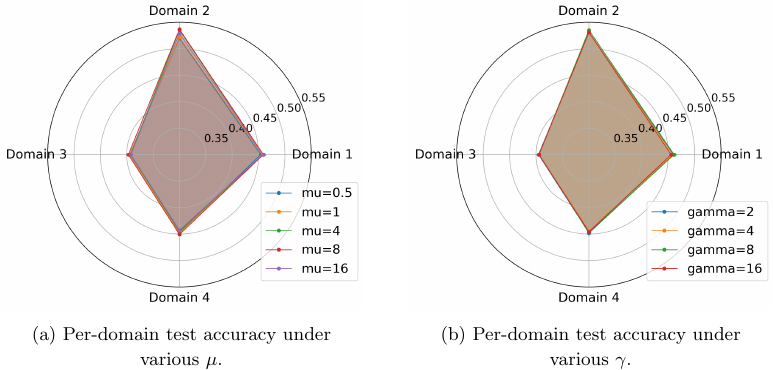}
	\caption{The result of the ablation study on $\gamma$ and $\mu$. }
	\label{fig:hp_ablation}
	\figurepanellabel{fig:mu_ablation}
	\figurepanellabel{fig:gamma_ablation}
\end{figure}

Figure~\ref{fig:hp_ablation} presents the per-domain test accuracy under different values of the hyperparameters $\mu$ and $\gamma$. Overall, both parameters have a relatively small impact on model performance, with $\mu$ exhibiting a slightly more pronounced influence. This indicates that our algorithm enjoys excellent robustness with respect to these hyperparameter choices. 
In theory, a larger $\mu$ yields a closer approximation to the ideal Tchebycheff scalarization, i.e., the algorithm focuses more on the currently worst-case weighted objective. However, in practice, the multiple objectives often exhibit both conflicts and synergies. A moderate $\mu$ allows the algorithm to prioritize the preferred objective while still leveraging beneficial information from other objectives to further improve the performance on the target one. As shown in Figure~\ref{fig:mu_ablation}, $\mu = 8$ achieves the best performance among the tested values under the preference setting, validating this intuition.

\begin{figure}[tbp]
	\centering
	\includegraphics[width=\linewidth]{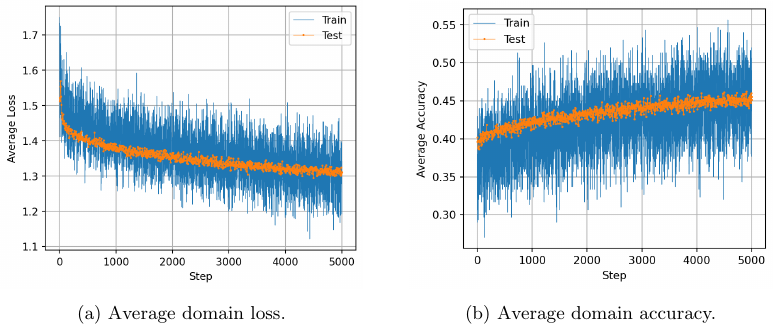}
	\caption{Convergence curves in the sensitivity analysis}
	\label{fig:curve_ablation}
	\figurepanellabel{fig:loss_curve_ablation}
	\figurepanellabel{fig:acc_curve_ablation}
\end{figure}

Figure~\ref{fig:curve_ablation} shows that the convergence curve under $\mu = 4$ and equal preference consistently improves over training epochs, indicating that our algorithm effectively minimizes the empirical loss. Notably, the training curve exhibits substantial fluctuations, as each point corresponds to only a single episode and one adaptation step. In contrast, the test curve is considerably smoother, since each evaluation point is averaged over 48 episodes.

\begin{figure}[tbp]
	\centering
	\includegraphics[width=\linewidth]{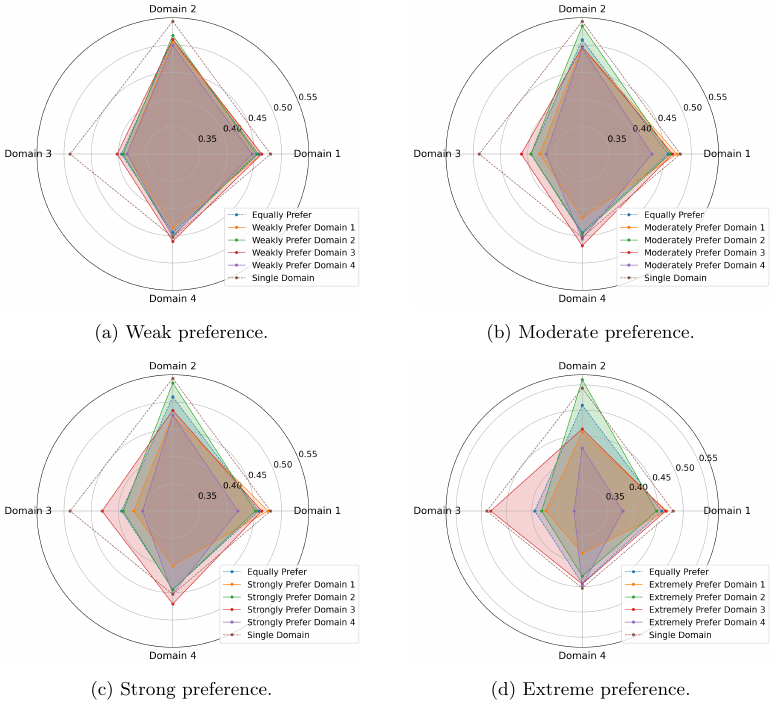}
	\caption{Test accuracy on the four domains under various preference intensities in the ablation studies.}
	\label{fig:pref_abaltion}
	\figurepanellabel{fig:weak_pref_ablation}
	\figurepanellabel{fig:moderate_pref_ablation}
	\figurepanellabel{fig:strong_pref_ablation}
	\figurepanellabel{fig:extreme_pref_ablation}
\end{figure}

Figure~\ref{fig:pref_abaltion} shows per-domain test accuracy under four different preference settings. For reference, we include the 'Single Domain' baseline, which corresponds to the degenerate case where the algorithm reduces to single-domain meta-learning.
This serves as an ideal performance benchmark for each individual domain. Overall, the test accuracy on Domains 2--4 increases monotonically with the strength of preference, eventually approaching or even surpassing the corresponding single domain baselines. This observation confirms that our algorithm can effectively navigate the Pareto front in a preference-guided manner.
In contrast, domain 1 exhibits a performance degradation under the extreme preference setting compared to the relatively milder ones. Upon examining the single domain training results across domains, we find that the performance of domain 1 under single domain 3 is already comparable to that under single domain 1 (see Figure~\ref{fig:single_domain}), indicating a strong synergy between domains 1 and 3.
Consequently, an excessively strong preference toward domain 1 may cause the information from domain 3 to be overlooked, leading to inferior performance compared to a more balanced preference setting.

\begin{table}[tbp]
	\centering
	\begin{minipage}[t]{0.54\linewidth}
		\vspace{0pt}
		\centering
		\begingroup
		\footnotesize
		\setlength{\tabcolsep}{3pt}
		\renewcommand{\arraystretch}{1.05}
		\begin{tabular}{@{}lcc@{}}
		\toprule
		\textbf{Algorithm} & \textbf{Parameter} & \textbf{Value} \\
		\midrule
		MOMEHA          & $\mu~[1.0, 4.0, 8.0]$         & 4.0 \\
		& $\gamma~[1.0, 4.0, 8.0]$      & 8.0 \\
		& $c_t$                         & 1.0 \\
		& $\alpha_{\theta, t}$          & 0.05 \\
		\cmidrule{1-3}
		MOML            & $K$                           & 4  \\
		& MGDA iteration number         & 4 \\
		& MGDA lr~$[0.01, 0.05, 0.1]$   & 0.05 \\
		\cmidrule{1-3}
		FORUM           & $K$                           & 4 \\
		& $\rho~[0.1, 0.5, 0.9]$        & 0.5 \\
		& $\beta_k$                     & $(1 + k)^{-\frac{3}{4}}$ \\
		\cmidrule{1-3}
		WC-penalty      
		& $\eta~[0.01, 0.1, 1.0]$       & 0.01 \\
		& $~u~[0.01, 0.05, 0.1]$         & 0.01 \\
		& $v~[0.1, 0.2, 0.5]$          & 0.2 \\
		\bottomrule
	\end{tabular}
		\endgroup
		\par
		\caption{Parameter settings for all compared algorithms in the meta-learning comparison experiment. The arrays in the Parameter column denote the corresponding parameter search spaces.}
	\label{tab:exp1_params}
	\end{minipage}\hfill
	\begin{minipage}[t]{0.44\linewidth}
		\vspace{0pt}
		\centering
		\includegraphics[width=\linewidth]{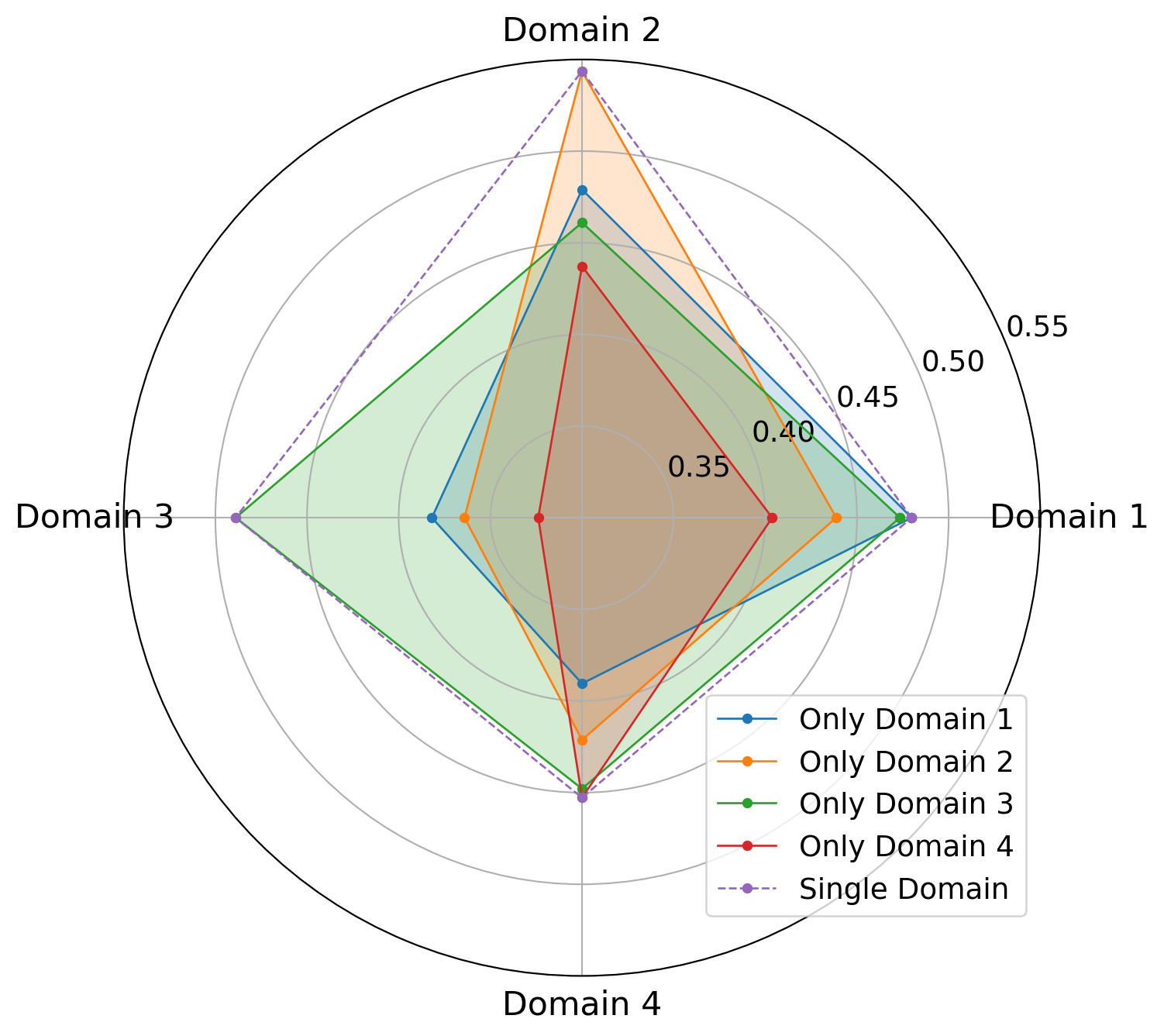}
	\captionof{figure}{Single domain test accuracy in the sensitivity analysis.}
	\label{fig:single_domain}
	\end{minipage}
\end{table}

\begin{figure}[htbp]
	\centering
	\includegraphics[width=\linewidth]{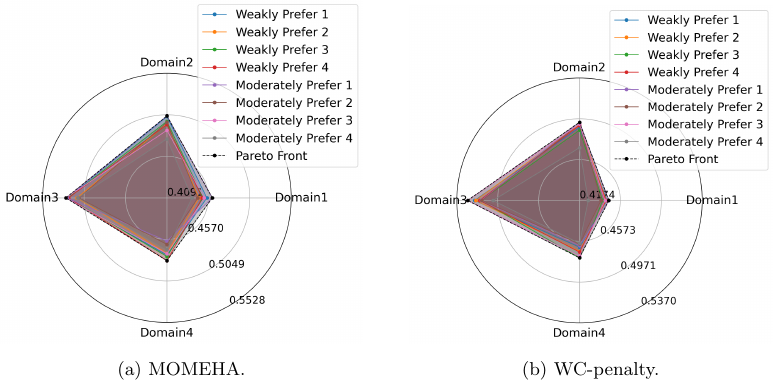}
	\caption{Test accuracy Pareto front exploration.}
	\label{fig:caltech_exploration}
	\figurepanellabel{fig:caltech_momeha_exploration}
	\figurepanellabel{fig:caltech_wcpenalty_exploration}
\end{figure}

\begin{figure}[htbp]
	\centering
	\includegraphics[width=\linewidth]{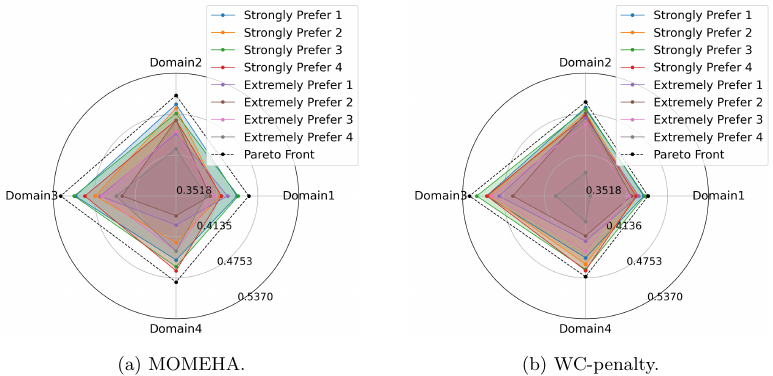}
	\caption{Performance deterioration in stronger preferences.}
	\label{fig:caltech_exploration_failure}
	\figurepanellabel{fig:caltech_momeha_exploration_failure}
	\figurepanellabel{fig:caltech_wcpenalty_exploration_failure}
\end{figure}

\begin{figure}[htbp]
	\centering
	\includegraphics[width=\linewidth]{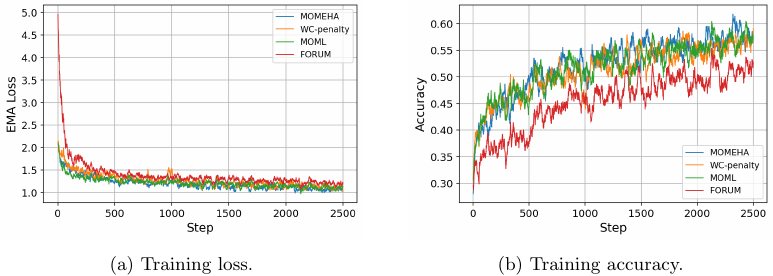}
	\caption{Training curves of the algorithms in the comparison experiment.}
	\label{fig:caltech_curves}
	\figurepanellabel{fig:caltech_loss}
	\figurepanellabel{fig:caltech_acc}
\end{figure}

In the \textbf{comparison experiment}, we set $T\text{ (outer loop for the other algorithms)}\allowbreak = 2500, \allowbreak\alpha_{x, t} = 0.1, \allowbreak\alpha_{y, t} = 0.05$ for all the algorithms. The final test accuracy Pareto front is derived from the best accuracy in each domain via enumerating different preference settings in the sensitivity analysis. The algorithm-specific parameter settings are presented in Table~\ref{tab:exp1_params}, some of which are chosen from tuning on the same validation set under equal preference. Despite our efforts to adapt WC-MHGD to the experiment, its accuracy performance stayed near $20\%$ under a wide range of hyperparameters, suggesting that the algorithm does not converge under our experiment settings. For this reason, we exclude it from the benchmark to ensure a fair and meaningful comparison.
Figure~\ref{fig:caltech_momeha_exploration} shows that \dalg~achieved its best performance on domains 1 and 4 under the moderate preference setting, and on Domains 2 and 3 under the weak preference setting. Figure~\ref{fig:caltech_wcpenalty_exploration} shows that WC-penalty exhibits a corresponding pattern on Domains 1 and 3.
We do not include the strong and extreme preference settings, as they lead to a universal performance degradation across all domains (see Figure~\ref{fig:caltech_exploration_failure}), with the deterioration becoming more pronounced as the preference intensifies. This observation suggests that the four domains may share certain synergies, and that the benefit of an overly strong single-domain preference may be outweighed by the loss of information from the other domains.

Figure~\ref{fig:caltech_curves} demonstrates the training convergence curves of domain 1, where \dalg~and WC-penalty moderately prefer domain 1. Both loss and accuracy curves validate that the algorithm’s behavior follows the preference guidance in the training phase as expected. For better visualization, an exponential moving average with $\alpha = 0.1$ is applied to the training loss.

\subsection{Multi-Objective Neural Architecture Search}
\label{subsec:nas_details}
To verify the effectiveness of \salg~in the stochastic setting, we perform an experiment on differentiable neural architecture search (DARTS)~\citep{liu2018darts} with multiple objectives, which can be formulated as follows:
\begin{gather*}
	\min_x \left[\mathcal{L}_i^{\text{val}}(x, y^*; \mathcal{V})\right]^{m}_{i = 1}, \\
	\text{s.t.} \quad y^* \in {\arg\min}_y \mathcal{L}^{\text{train}}(x, y; \mathcal{T}),
\end{gather*}
where $x$ is the architecture parameter, $y$ is the network parameter, $\mathcal{T}$ is the training set, $\mathcal{V}$ is the validation set, $\mathcal{L}_i^{\text{val}}$ is the loss function of objective $i$, $\mathcal{L}^{\text{train}}$ is the loss function of the currently parameterized architecture. 
The above MOBL problem seeks the parameters with the minimal training loss on a continuously parameterized architecture and then optimizes the architecture parameters for the minimal validation losses in each objective.

We implement the experiment on the CIFAR-10 data set~\citep{krizhevsky2009learning}, which consists of 50,000 color images across 10 classes, each of size $32 \times 32$ pixels. We split it into a training set and a validation set with a ratio of 1:1.

Following DARTS, we adopt the same search space and relevant hyperparameters. The search procedures are performed on a 3-cell network (Normal-Reduction-Normal) for 50 epochs.

We design four objectives for the experiment: validation loss (Obj. 1), FLOPS loss (Obj. 2), skip density (Obj. 3), and pooling density (Obj. 4). While the first two objectives naturally arise from the multi-objective nature of DARTS, the latter two serve as regularization terms to discourage the search procedure from converging to architectures dominated by parameter-free operations.
We consider the same preference settings as in the meta-learning experiment. We additionally conduct a 2-objective comparison, where we use four settings using objective 1 as the reference: an equally preferred setting with $w=[0.5, 0.5]^\top$, a weakly preferred setting with $w=[0.6, 0.4]^\top$, a strongly preferred setting with $w=[0.75, 0.25]^\top$, an extremely preferred setting with $w=[0.9, 0.1]^\top$.
Specifically, for the operation vector $e$ of each edge with given spatial resolution $H, W$, input and output channel counts $C_{in}, C_{out}$, and kernel size $K$, we compute the FLOPS loss and densities according to the following formulas:
\[
\begin{aligned}
	& \text{FLOPS(Conv2d)} = 2H \cdot W \cdot C_{\text{in}} \cdot C_{\text{out}} \cdot K^2, \quad \text{FLOPS}(\text{Pooling}(C_{\text{in}} \neq C_{\text{out}})) \\
	& = 2H \cdot W \cdot C_{\text{in}} \cdot C_{\text{out}}, \\
	& \text{FLOPS}(\text{Pooling}(C_{\text{in}} = C_{\text{out}})) = 0, \quad \text{FLOPS(SkipConnect)} = 0, \quad \text{FLOPS(Zero)} = 0, \\
	& \text{FLOPS loss} = \frac{1}{\left|E_N\right| + \left|E_R\right|}\sum_{e \in E_N \cup E_R} \sum^8_{i = 1} \text{softmax}(x_e)_i \cdot \frac{\text{FLOPS}(e_i)}{\max_i \text{FLOPS}(e_i)}, \\
	& \text{SkipConnect density} = \frac{1}{\left|E_N\right| + \left|E_R\right|} \sum_{e \in E_N \cup E_R} \text{softmax}(x_e)_{\text{SkipConnect}}, \\
	& \text{Pooling density} = \frac{1}{\left|E_N\right| + \left|E_R\right|} \sum_{e \in E_N \cup E_R} \text{softmax}(x_e)_{\text{Pooling}},
\end{aligned}
\]
where $K$ is the filter size, $E_N$ is the set of the operation vectors in a normal cell, $E_R$ is the set of the operation vectors in a reduction cell, $x_e$ is the architecture parameter vector of the operation vector $e$.

We set $\alpha_{x,t} = 0.05$ and $\alpha_{y,t} = 0.025$ for all the algorithms and use a momentum coefficient of $0.9$. The remaining hyperparameters follow the same settings (except for $c_t = (1 + t)^{\frac{1}{16}}$) in the deterministic experiment. For WC-MHGD, we set $u = 10, D = 4$ and implement projected gradient descent with eight steps and a learning rate of 0.01 to solve the QP. In addition, due to VRAM limitation, we approximate the hypergradient of WC-MHGD via finite differences ($\epsilon = 0.01$). For MoCo, we use three lower-level SGD steps per architecture update and set the gradient-tracking coefficient to $\beta = 0.5$, the simplex update step size to $\gamma = 0.1$, and the weight regularization coefficient to $\rho = 0.01$. We also approximate its validation-loss hypergradient via finite differences ($\epsilon = 0.01$).

Figure~\ref{fig:4task_exploration} presents the Pareto front exploration of different algorithms. Figure~\ref{fig:nas_momeha_exploration} shows that \salg~consistently aligns with the prescribed preferences across all four objectives, achieving better performance on the preferred objectives. 
Figure~\ref{fig:nas_wcmhgd_exploration} and Figure~\ref{fig:nas_wcpenalty_exploration} indicate that the other two algorithms exhibit a relatively limited exploration range on Obj. 1. Note that Obj. 3 and 4 serve as regularization objectives to prevent the architecture from collapsing to parameter-free operations. As such, their values are not necessarily the lower the better.
We therefore exclude the extreme preference setting from the figure, as it would impose excessive penalization on these regularization terms and distort the intended search objective. Figure~\ref{fig:nas_pareto_comparison} shows that \salg~achieves a larger exploration range on Obj. 1 and 4, while its performance on Obj. 2 and 3 is comparable to that of WC-MHGD. MoCo exhibits a significantly larger value on Obj. 3, but this comes at the cost of effectively abandoning the skip-connection operation, which is undesirable as it severely restricts the expressiveness of the searched architectures.

Figure~\ref{fig:4task_nas_curves} illustrates the convergence behaviors of the algorithms during the search phase. Each subplot corresponds to a specific objective preference and reports the averaged loss values over the four objectives under both the moderate and strong preference settings. Across all four objectives, the two baseline algorithms exhibit flattening curves toward the end of training, whereas our method not only achieves lower loss values but also maintains a consistently decreasing trend, indicating its superior performance and applicability in the lower-level nonconvex scenario.
\begin{figure}[tbp]
	\centering
	\includegraphics[width=\linewidth]{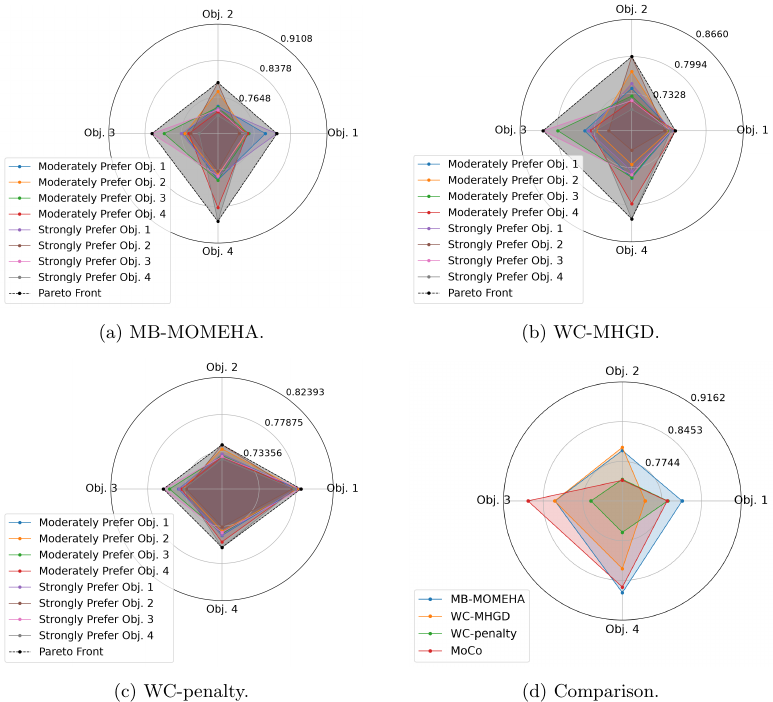}
	\caption{Pareto front exploration and comparison. For better visualization and based on the magnitude of the results, the losses for Obj. 2 and Obj. 3 are scaled by factors of 0.67 and 2, respectively, and the final results are plotted as $1 - \text{scaled loss}$}
	\label{fig:4task_exploration}
	\figurepanellabel{fig:nas_momeha_exploration}
	\figurepanellabel{fig:nas_wcmhgd_exploration}
	\figurepanellabel{fig:nas_wcpenalty_exploration}
	\figurepanellabel{fig:nas_pareto_comparison}
\end{figure}

\begin{figure}[tbp]
	\centering
	\includegraphics[width=\linewidth]{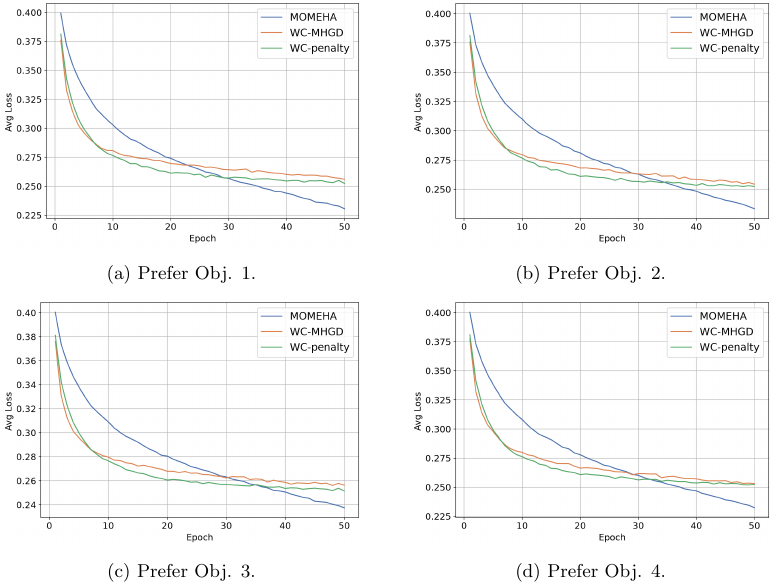}
	\caption{Validation convergence curves.}
	\label{fig:4task_nas_curves}
	\figurepanellabel{fig:nas_t1_curve}
	\figurepanellabel{fig:nas_t2_curve}
	\figurepanellabel{fig:nas_t3_curve}
	\figurepanellabel{fig:nas_t4_curve}
\end{figure}

\newpage

\vskip 0.2in
\bibliography{MOMEHA}

\end{document}